\documentclass[11pt]{article}
\usepackage[margin=0.75in]{geometry}

\usepackage{amsthm}
\usepackage{amssymb}
\usepackage{amsmath}
\usepackage{comment}
\usepackage{thm-restate}
\usepackage{url} 
\usepackage{hyperref}
\usepackage[noabbrev,capitalise]{cleveref}

\usepackage[utf8]{inputenc}

\usepackage{color}
\usepackage[normalem]{ulem}

\newcommand{\mc}[1]{\mathcal{#1}}

\newcounter{i}
\theoremstyle{plain}
\newtheorem{thm}{Theorem}[section]
\newtheorem{lem}[thm]{Lemma}
\newtheorem{proposition}[thm]{Proposition}

\newtheorem{cor}[thm]{Corollary}

\newtheorem{conj}[thm]{Conjecture}

\newenvironment{proofclaim}[1][]%
{\noindent \emph{Proof.} {}{#1}{}}{\hfill
	$\Diamond$\vspace{1em}}

\usepackage{amsthm,thmtools}

\declaretheorem[
  style=plain,
  name=Claim,
  within=theorem,
]{claim}

\newenvironment{lateproof}[1]
 {%
  \renewcommand{\theclaim}{\ref{#1}.\arabic{claim}}%
  \begin{proof}[Proof of~\cref{#1}]%
 }
 {\end{proof}}

\newenvironment{nonlateproof}[1]
 {%
  \renewcommand{\theclaim}{\ref{#1}.\arabic{claim}}%
  \begin{proof}%
 }
 {\end{proof}}

\usepackage{etoolbox}
 \AtEndEnvironment{proof}{\setcounter{claim}{0}}

\theoremstyle{plain} 
\newcommand{\thistheoremname}{}
\newtheorem{genericthm}{\thistheoremname}

\theoremstyle{definition}
\newtheorem{definition}[thm]{Definition}

\newtheorem{remark}[thm]{Remark}

\usepackage[dvipsnames, table]{xcolor}

\newcommand{\Prob}[1]{\ensuremath{%
\mathbb P\left[#1\right]
}}

\newcommand{\Expect}[1]{\ensuremath{%
\mathbb E\left[#1\right]
}}

\title{Proof of the High Girth Existence Conjecture via Refined Absorption}

\author{
Michelle Delcourt
\thanks{Department of Mathematics, Toronto Metropolitan University (formerly named Ryerson University),
Toronto, Ontario M5B 2K3, Canada {\tt mdelcourt@torontomu.ca}. Research supported by NSERC under Discovery Grant No. 2019-04269.}
\and
Luke Postle
\thanks{Combinatorics and Optimization Department,
University of Waterloo, Waterloo, Ontario N2L 3G1, Canada {\tt lpostle@uwaterloo.ca}. Partially supported by NSERC
under Discovery Grant No. 2019-04304.}}
\date{February 29, 2024}

\begin{document}

\maketitle

\begin{abstract} 
We prove the High Girth Existence Conjecture - the common generalization of the Existence Conjecture for Combinatorial Designs originating from the 1800s and Erd\H{o}s' Conjecture from 1973 on the Existence of High Girth Steiner Triple Systems. 
\end{abstract}

\section{Introduction}

\subsection{Existence Conjecture}

A \emph{Steiner system} with parameters $(n,q,r)$ is a set $S$ of $q$-subsets of an $n$-set $X$ such that every $r$-subset of $X$ belongs to exactly one element of $S$. More generally, a \emph{design} with parameters $(n,q,r,\lambda)$ is a set $S$ of $q$-subsets of an $n$-set $X$ such that every $r$-subset of $X$ belongs to exactly $\lambda$ elements of $S$. 

The notorious Existence Conjecture originating from the mid-1800's asserts that designs exist for large enough $n$ provided that the obvious necessary divisibility conditions are satisfied as follows.

\begin{conj}[Existence Conjecture]\label{conj:Existence}
Let $q > r \ge 2$ and $\lambda \geq 1$ be integers. If $n$ is sufficiently large and $\binom{q-i}{r-i}~|~\lambda \binom{n-i}{r-i}$ for all $0\le i \le r-1$, then there exists a design with parameters $(n,q,r,\lambda)$.
\end{conj}

In 1847, Kirkman~\cite{K47} proved this when $q=3$, $r=2$ and $\lambda=1$. In the 1970s, Wilson~\cite{WI, WII, WIII} proved the Existence Conjecture for graphs, i.e.~when $r=2$ (for all $q$ and $\lambda$). In 1985, R\"{o}dl~\cite{R85} introduced his celebrated ``nibble method'' to prove that there exists a set $S$ of $q$-subsets of an $n$-set $X$ with $|S| = (1-o(1))\binom{n-r}{q-r}$ such that every $r$-subset is in at most one element of $S$, thereby settling the approximate version of the Existence Conjecture (known as the Erd\H{o}s-Hanani Conjecture~\cite{EH63}). Only in the last decade was the Existence Conjecture fully resolved as follows.

\begin{thm}[Keevash~\cite{K14}]\label{thm:Existence}
Conjecture~\ref{conj:Existence} is true.   
\end{thm}

Namely in 2014, Keevash~\cite{K14} proved the Existence Conjecture using \emph{randomized algebraic constructions}. Thereafter in 2016, Glock, K\"{u}hn, Lo, and Osthus~\cite{GKLO16} gave a purely combinatorial proof of the Existence Conjecture via \emph{iterative absorption}. Both approaches offer different benefits, and each has led to substantial subsequent work. We refer the reader to~\cite{W03, K14, GKLO16} for more history on the conjecture.

In our previous paper~\cite{DPI}, we provided a new third proof of the Existence Conjecture via our novel method of \emph{refined absorption}; we note that our proof assumed the existence of $K_q^r$ absorbers (as proved by Glock, K\"uhn, Lo, and Osthus) but sidestepped the use of iterative absorption. Here we build on this new approach to prove the existence of high girth Steiner systems as follows.

\subsection{High Girth Steiner Systems}

In a partial $(n,q,r)$-Steiner system, a \emph{$(j, i)$-configuration} is a set of $i$ $q$-element subsets spanning at most $j$ vertices. The \emph{girth} of a partial $(n,q,r)$-Steiner system is the smallest integer $g \geq 2$ for which it has a $((q-r)g+r, g)$-configuration.

It is easy to see that an $(n,3,2)$-Steiner system contains an $(i+3,i)$-configuration for every fixed $i$ (hence the definition of girth in this case). In 1973, Erd\H{o}s~\cite{E73} conjectured the following. 

\begin{conj}[Existence of High Girth Steiner Triple Systems - Erd\H{o}s~\cite{E73}]\label{conj:Erdos}
For every integer $g\geq2$, there exists $n_g$ such that for all $n \geq n_g$ where $n\equiv 1,3 \mod 6$, there exists an $(n,3,2)$-Steiner system with girth at least $g$.
\end{conj}

In 1993, Lefmann, Phelps, and R\"odl~\cite{LPR93} proved that for every integer $g\geq2$, there exists a constant $c_g>0$ such that for all $n\geq 3$ there exists a partial $(n,3,2)$-Steiner system $S$ with $|S|\ge c_g \cdot n^2$ and girth at least $g$. Recently, Glock, K\"{u}hn, Lo, and Osthus~\cite{GKLO20} and independently Bohman and Warnke~\cite{BW19} proved there exists a partial $(n,3,2)$-Steiner system $S$ with $|S|=(1-o(1))n^2/6$ and girth at least $g$, thus settling the approximate version of Conjecture~\ref{conj:Erdos}. Thereafter, in 2022, Kwan, Sah, Sawhney, and Simkin~\cite{KSSS22}, building on these approximate versions while incorporating iterative absorption, impressively proved Conjecture~\ref{conj:Erdos} in full as follows.

\begin{thm}[Kwan, Sah, Sawhney, and Simkin~\cite{KSSS22}]
Conjecture~\ref{conj:Erdos} is true.    
\end{thm}

Glock, K\"{u}hn, Lo, and Osthus~\cite{GKLO20} observed that every $(n,q,r)$-Steiner system contains a $((q-r)i+r+1,i)$-configuration for every fixed $i\ge 2$ (hence the definition of the girth of Steiner systems above). Thus they (and also Keevash and Long~\cite{KL20} in 2020) made the following common generalization of the Existence Conjecture and Erd\H{o}s' conjecture.

\begin{conj}[Existence of High Girth Designs - Glock, K\"{u}hn, Lo, and Osthus~\cite{GKLO20}; Keevash and Long~\cite{KL20}]\label{conj:HighGirthAllUniformities}
For all integers $q > r \geq 2$ and every integer $g\ge 2$, there exists $n_0$ such that for all $n\ge n_0$ where $\binom{q-i}{r-i}~|~\binom{n-i}{r-i}$ for all $0\le i \le r-1$, there exists an $(n,q,r)$-Steiner system with girth at least $g$.
\end{conj}

For $r=2$ and all $q\geq 3$, this was asked by F\"uredi and Ruszink\'o~\cite{FR13} in 2013. In 2022, the authors~\cite{DP22} and independently Glock, Joos, Kim, K\"uhn, and Lichev~\cite{GJKKL22} proved the approximate version of this conjecture, namely that partial $(n,q,r)$-Steiner systems of high girth with almost full size exist.

Our main result of this paper is that we prove Conjecture~\ref{conj:HighGirthAllUniformities} in full as follows.

\begin{thm}[Existence of High Girth  Designs]\label{thm:HighGirthSteiner}
Conjecture~\ref{conj:HighGirthAllUniformities} is true.
\end{thm}

Furthermore, our proof of Theorem~\ref{thm:HighGirthSteiner} differs from that of Kwan, Sah, Sawhney, and Simkin's proof of the existence of high girth Steiner triple systems in several key aspects and hence provides an independent and conceptually streamlined proof of their result. 

\begin{remark} We note the conjectures and results above are for Steiner systems (i.e.~the $\lambda=1$ case of designs) instead of for designs (e.g.~the constant $\lambda$ case). The reason is that for all $\lambda \ge 2$, a design with parameters $(n,q,r,\lambda)$ necessarily contains a $((q-r)2+r,2)$-configuration (as there are at least two $q$-sets containing any given $r$-set) and hence has girth $2$. Thus the conjectures above would need to be amended in the $\lambda\ge 2$ regime to permit these obvious necessary configurations. We think our methods would also apply to prove such amended conjectures but we do not pursue this here.
\end{remark}

\subsection{Refined Absorbers}

We note the Existence Conjecture may be rephrased in graph theoretic terms. We refer the reader to the notation section (Section~\ref{ss:Notation} below) for the relevant background definitions and notation.

A transformative concept that upended design theory and many similar exact structural decomposition problems is the \emph{absorbing method}, a method for transforming almost perfect decompositions into perfect ones, wherein a set of absorbers are constructed that can absorb any particular random leftover of a random greedy or nibble process into a decomposition. 

For $K_q^r$ decompositions, the key definition of absorber is as follows.

\begin{definition}[Absorber]
Let $L$ be a $K_q^r$-divisible hypergraph. A hypergraph $A$ is a \emph{$K_q^r$-absorber} for $L$ if $V(L)\subseteq V(A)$ is independent in $A$ and both $A$ and $A\cup L$ admit $K_q^r$ decompositions.
\end{definition}

The main use of absorbers is to absorb the potential `leftover' $K_q^r$-divisible graph of a specific set/graph $X$ after a nibble/random-greedy process constructs a $K_q^r$-decomposition covering all edges outside of $X$. To that end, we need the following definition from~\cite{DPI}.

\begin{definition}[Omni-Absorber]
Let $q > r\ge 1$ be integers. Let $X$ be a hypergraph. We say a hypergraph $A$ is a \emph{$K_q^r$-omni-absorber} for $X$ with \emph{decomposition family} $\mathcal{H}$ and \emph{decomposition function} $\mathcal{Q}_A$ if $V(X)=V(A)$, $X$ and $A$ are edge-disjoint, such that $|H\cap X|\le 1$ for all $H\in\mathcal{H}$ and for every $K_q^r$-divisible subgraph $L$ of $X$, there exists $\mathcal{Q}_A(L)\subseteq \mathcal{H}$ that are pairwise edge-disjoint and such that $\bigcup \mathcal{Q}_A(L)=L\cup A$. 
\end{definition}

In our previous paper~\cite{DPI}, we proved that extremely efficient omni-absorbers exist (that is with $\Delta(A)=O(\Delta(X))$, provided $\Delta(X)$ is large). Furthermore, we proved that such omni-absorbers exist where each edge is only in a constant number of elements of the decomposition family. This extra property will be crucial to our proof of Theorem~\ref{thm:HighGirthSteiner} since it will permit us to `girth boost' the elements of the decomposition family in such a way as to increase the overall girth.

We recall the following definition from~\cite{DPI}.

\begin{definition}[Refined Omni-Absorber]
 Let $C\ge 1$ be real. We say a $K_q^r$-omni-absorber $A$ for a hypergraph $X$ with decomposition family $\mathcal{H}$ is \emph{$C$-refined} if $|\{H\in \mathcal{H} : e\in E(H) \}| \le C$ for every edge $e\in X\cup A$.
\end{definition}

Here was the main result for omni-absorbers from our previous paper~\cite{DPI}.

\begin{thm}[Refined Omni-Absorber Theorem]\label{thm:Omni}
For all integers $q > r\ge 1$, there exist an integer $C\ge 1$ and real $\varepsilon\in (0,1)$ such that the following holds: Let $G$ be an $r$-uniform hypergraph on $n$ vertices with $\delta(G)\ge (1-\varepsilon)n$. If $X$ is a spanning subhypergraph of $G$ with $\Delta(X) \le \frac{n}{C}$ and we let $\Delta:= \max\left\{\Delta(X),~n^{1-\frac{1}{r}}\cdot \log n\right\}$, then there exists a $C$-refined $K_q^r$-omni-absorber $A\subseteq G$ for $X$ such that $\Delta(A)\le C \cdot \Delta$. 
\end{thm}

Refined absorption thus provides a more conceptually streamlined proof of Theorem~\ref{thm:HighGirthSteiner} than is possible with iterative absorption. We will only need two other major ingredients. The first ingredient is a \emph{high girth nibble} theorem except with \emph{reserves}. This will follow from our more general \emph{Forbidden Submatching Method} from~\cite{DP22} (independently developed by Glock, Joos, Kim, K\"{u}hn, and Lichev~\cite{GJKKL22} as \emph{conflict-free hypergraph matchings}; in particular though, we need the new main nibble theorem (``Forbidden Submatchings with Reserves'') from the latest version of our earlier paper~\cite{DP22}. See Section~\ref{ss:ForbiddenSubmatching} for its statement.

The second main ingredient is to prove the existence of \emph{high girth omni-absorbers}, which is the main work of this paper. Specifically we first prove the existence of \emph{girth boosters} for $K_q^r$ decompositions for all $q > r \ge 2$ (the $q=3$, $r=2$ case was constructed by Kwan, Sah, Sawhney, and Simkin in~\cite{KSSS22}). Then we need to simultaneously boost the elements of the decomposition family of an omni-absorber (which we call an \emph{omni-booster}) in such a way that they collectively have high girth and such that there are not too many potential forbidden configurations in the remainder of the graph. Thus, even though we have access to two very general black-box theorems from our previous papers~\cite{DP22, DPI} whose full power we will require, the remaining work to prove Theorem~\ref{thm:HighGirthSteiner} is still highly non-trivial, very technical and requires an intricate weaving together of these various ingredients to actually complete the proof. We overview these novel contributions next. 

\subsection{High Level Proof Overview}

First we recall the outline of our new proof of the Existence Conjecture from~\cite{DPI}:

\begin{enumerate}
    \item[(1)] `Reserve' a random subset $X$ of $E(K_n^r)$.
    \item[(2)] Construct an omni-absorber $A$ of $X$.
    \item[(3)] ``Regularity boost'' $K_n^r\setminus (A\cup X)$.
    \item[(4)] Apply ``nibble with reserves'' theorem to find a $K_q^r$ packing of $K_n^r\setminus A$ covering $K_n^r\setminus (A\cup X)$ and then extend this to a $K_q^r$ decomposition of $K_n^r$ by definition of omni-absorber.  
\end{enumerate}
The required lemma for (1) followed easily from the Chernoff bounds. For (3), we used a special case of the Boosting Lemma of Glock, K\"{u}hn, Lo, and Osthus~\cite{GKLO16}. For (4), we invoked our ``nibble with reserves'' theorem from~\cite{DP22}.  For (2), we used Theorem~\ref{thm:Omni}, which is the general black-box theorem from our first paper on refined absorption~\cite{DPI}.

To prove the High Girth Existence Conjecture, we follow the same plan with modifications. (1) and (3) remain the same, even using the same results as before. For (4), we apply our ``Forbidden Submatchings with Reserves'' Theorem (see Theorem~\ref{thm:ForbiddenSubmatchingReserves} of this paper), which is our very general black-box theorem from our paper on hypergraph matchings~\cite{DP22} and whose full power we will require; since it requires many definitions, we defer those and its statement to Section~\ref{ss:ForbiddenSubmatching}.

Thus the main proof part of this paper is to modify (2) as follows:
\begin{enumerate}
    \item[(2')] Construct a \emph{high girth} omni-absorber $A$ of $X$ that \emph{does not shrink `too many' configurations}. 
\end{enumerate}

The required definition of high girth omni-absorber is straightforward as follows. 

\begin{definition}\label{def:OmniAbsorberCollectiveGirth}
Let $X$ be an $r$-uniform hypergraph. We say a graph a $K_q^r$-omni-absorber $A$ for $X$ with decomposition function $\mathcal{Q}_A$ has \emph{collective girth at least $g$} if for every $K_q^r$-divisible subgraph $L$ of $X$, $\mathcal{Q}_A(L)$ has girth at least $g$. 
\end{definition}

An omni-absorber having high collective girth is not enough for our purposes to prove Theorem~\ref{thm:HighGirthSteiner}. The issue is that the various $\mathcal{Q}_A(L)$, individually or collectively, may forbid subconfigurations in $K_n^r\setminus (X\cup A)$ that would extend to forbidden configurations with $\mathcal{Q}_A(L)$. We will need to ensure there are not too many such forbidden subconfigurations. Namely, the required assumption is whatever is necessary to apply the ``Forbidden Submatchings with Reserves'' Theorem (Theorem~\ref{thm:ForbiddenSubmatchingReserves} below). Again, to state the precise technical definition of ``not shrink too many'' requires many other definitions first and so we delay its statement to Section~\ref{ss:OmniProjection}. Our main High Girth Omni-Absorber Theorem (Theorem~\ref{thm:HighGirthAbsorber} below) will thus produce the required omni-absorber for (2'), but we similarly defer its statement to Section~\ref{ss:OmniProjection}.

There are two important proof parts then to enact (2'). First, how do we build an omni-absorber of collective girth at least $g$? Second, how do we ensure it does not shrink ``too many'' configurations?

For the first part, we start with Theorem~\ref{thm:Omni} which crucially provides a $C$-refined omni-absorber. Then we add a private gadget to each clique in the decomposition family that we call a \emph{booster}. If we add the right booster, we can ensure the new resulting omni-absorber has collective girth at least $g$ - at least with high probability provided we embed the boosters randomly. 

To formalize this, we introduce the following definitions.

\begin{definition}[Booster]
Let $q > r\ge 1$ be integers. A a \emph{$K_q^r$-booster} is an $r$-uniform hypergraph $B$ along with two $K_q^r$ decompositions $\mathcal{B}_1,\mathcal{B}_2$ of $B$ such that $\mathcal{B}_1\cap \mathcal{B}_2 =\emptyset$ (that is, they do not have a $K_q^r$ in common). 
\end{definition}

A booster has the desirable property that it has two disjoint decompositions. In practice, we use this to `boost' some particular $K_q^r$-clique $S$. For this then we think of $S$ as fixed and want to add an object to boost it. Thus we define a rooted booster as follows.

\begin{definition}[Rooted Booster]\label{def:rootedbooster}
Let $q > r\ge 1$ be integers. A \emph{rooted $K_q^r$-booster} rooted at $S\cong K_q^r$ is an $r$-uniform hypergraph $B$ that is edge-disjoint from $S$ along with two disjoint $K_q^r$ packings $\mathcal{B}_{{\rm off}},\mathcal{B}_{{\rm on}}$ where $\mathcal{B}_{{\rm off}}$ is a $K_q^r$-decomposition of $B$ and $\mathcal{B}_{{\rm on}}$ is a $K_q^r$-decomposition of $B\cup S$ with $S\not\in \mathcal{B}_{{\rm on}}$.
\end{definition}

The idea then is that $B_{\rm on}$ and $B_{\rm off}\cup \{S\}$ are the two decompositions of the $K_q^r$-booster $B\cup S$. However, now when building a decomposition, if we desired to use the clique $S$ in the decomposition, we could instead replace it with $\mathcal{B}_{{\rm on}}$ which decomposes $B\cup S$; where as if we do not desire to use $S$, we simply use $\mc{B}_{\rm off}$ to decompose the edges of $B$. In this way, we can replace $S$ with the cliques of the two decompositions $\mc{B}_{\rm on}$ and $\mc{B}_{\rm off}$ which may have more desirable properties than $S$ (in particular for us, they will have high girth).

To that end, we define the notion of the rooted girth of a booster as follows.

\begin{definition}[Rooted Girth]
Let $q > r\ge 1$ be integers. Let $\mathcal{B}$ be a $K_q^r$-packing of a graph $G$. For a vertex set $S\subseteq V(G)$, we define the \emph{rooted girth} of $\mathcal{B}$ at $S$ as the smallest integer $g\ge 1$ such that there exists a subset $\mathcal{B}'\subseteq \mathcal{B}$  with $|\mathcal{B}'|=g$ and $|V(\bigcup \mathcal{B}')\setminus S| < (q-r)\cdot g$.  

Similarly, if $B$ is a rooted $K_q^r$-booster rooted at $S$, then we define the \emph{rooted girth} of $B$ as the minimum of the girth of $\mathcal{B}_{{\rm on}}$, the girth of $\mathcal{B}_{{\rm off}}\cup \{S\}$, and the rooted girth of $\mathcal{B}_{{\rm on}}$ at $V(S)$.
\end{definition}

The first key technical result of this paper is the existence of rooted boosters of high rooted girth as follows.

\begin{thm}[High Rooted Girth Booster Theorem]\label{thm:HighRootedGirthBooster}
Let $q > r\ge 2$ and $g\ge 1$ be integers. If Theorem~\ref{thm:HighCogirthBooster} holds for $r':=r-1$, then there exists a rooted $K_q^r$-booster $B$ with rooted girth at least $g$. 
\end{thm}

 The proof of Theorem~\ref{thm:HighRootedGirthBooster} uses induction on the uniformity, namely the existence of $K_{q-1}^{r-1}$-boosters; however it is not enough to simply have these of high rooted girth. Rather, we need the existence of $K_{q-1}^{r-1}$-boosters of high \emph{cogirth} (see Definition~\ref{def:Cogirth}) as proved in Theorem~\ref{thm:HighCogirthBooster}. The only way we are able to construct these is to prove a stronger form of Theorem~\ref{thm:HighGirthSteiner}, namely Theorem~\ref{thm:HighCogirthSteiner}, which asserts that there exist two disjoint high girth Steiner systems with high cogirth. The proof of Theorem~\ref{thm:HighCogirthSteiner} follows similarly to that of Theorem~\ref{thm:HighGirthSteiner} with some minor modifications which we outline in Section~\ref{s:CogirthPair}.

To prove Theorem~\ref{thm:HighGirthAbsorber} then, we add a private rooted booster of rooted girth at least $g$ for each clique in the decomposition family of the omni-absorber; we do this simultaneously (and randomly such that with high probability this is done in a high girth manner). To that end, we define the following object. But first a definition.

\begin{definition}[Design Hypergraph]
Let $F$ be a hypergraph. If $G$ is a hypergraph, then the \emph{$F$-design hypergraph of $G$}, denoted ${\rm Design}_F(G)$, 
is the hypergraph $\mathcal{D}$ with $V(\mathcal{D}):=E(G)$ and $E(\mathcal{D}) := \{S\subseteq E(G): S \text{ is isomorphic to } F\}$.
\end{definition}

\begin{definition}[Omni-Booster]\label{def:OmniBooster}
Let $q > r\ge 1$ be integers. Let $X$ be an $r$-uniform hypergraph and let $A$ be a $K_q^r$-omni-absorber for $X$. We say an $r$-uniform hypergraph $B$ is a \emph{$K_q^r$-omni-booster} for $A$ and $X$ if $B$ is edge-disjoint from $A\cup X$ and $B$ is the edge-disjoint union of a set of hypergraphs $\mathcal{B}:= (B_H: H\in\mathcal{H})$ (called the \emph{booster family} of $B$) where for each $H\in \mathcal{H}$, $B_H$ is a rooted $K_q^r$-booster rooted at $H$.

The \emph{matching set} of $\mc{B}$ is 
$$\mc{M}(\mc{B}):= \bigg\{ M \in \prod_{H\in \mathcal{H}} \big\{ (\mathcal{B}_{H})_{{\rm on}},~ (\mathcal{B}_{H})_{{\rm off}} \big\} : M \text{ is a matching of } {\rm Design}_{K_q^r}(X\cup A\cup B)\bigg\}.$$
We say $B$ has \emph{collective girth at least $g$} if every $M\in \mc{M}(\mc{B})$ has girth at least $g$.
\end{definition}

We note that $K_q^r$-omni-boosters yield $K_q^r$-omni-absorbers with a canonical decomposition function (see Proposition~\ref{prop:CanonicalBoost}). Again though, it does not suffice just to construct an omni-booster of collective girth at least $g$ since we also have to ensure it does not shrink too many configurations. Thus the second key technical result of this paper is Theorem~\ref{thm:HighGirthOmniBooster} which guarantees the existence of an omni-booster of high collective girth that does not shrink ``too many'' configuration, but similarly due to its technical nature, we defer its statement to Section~\ref{s:GirthOmniBoosters} where we also prove it assuming certain key lemmas that are proved in Sections~\ref{s:Intrinsic} and~\ref{s:Extrinsic}.

The second part of not shrinking too many configurations is also in fact handled by embedding the boosters randomly. However, there was a technical issue on how to achieve this embedding proof. Ensuring collective girth at least $g$ would follow similarly to the other embedding proofs in this series of papers: we could have used the deterministic `avoid the bad sets' proof of our previous paper~\cite{DPI} or a Lov\'asz Local Lemma plus slots proof as in our other paper with Tom Kelly~\cite{DKPIV}. However neither seem to work to ensure not shrinking too many configurations. It is conceivable to us that a random greedy process does work here, but this would likely require redoing much of the work from Glock, Joos, Kim, K\"{u}hn, and Lichev~\cite{GJKKL22} and still require more work to actually ensure every clique is boosted.

The only proof we were able to come up with is a `sparsification proof'. Namely, we randomly sparsify the possible choices of booster for each clique to a very small amount but still large enough to concentrate. Then we apply many applications of the Kim-Vu concentration inequality~\cite{KV00,V02} to show the various parameters of the shrinkings are all highly concentrated (at least for their upper bound with a possible over-counting). Similarly, we show within such a sparsification, there is a choice of disjoint boosters of collective girth at least $g$ (with high probability). 

\subsection{Outline of Paper}

In Section~\ref{s:HighGirthExistence}, we first recall the results for (1) and (3), the machinery for (4'), and then define and state the theorem for (2'). Then we prove Theorem~\ref{thm:HighGirthSteiner} assuming the aforementioned theorems.

In Section~\ref{s:GirthBoosters}, we prove Theorem~\ref{thm:HighRootedGirthBooster}. First we define the notion of cogirth for a booster and explain the need for high co-girth boosters. Assuming their existence inductively (which follows from Theorem~\ref{thm:HighCogirthSteiner}, a stronger version of Theorem~\ref{thm:HighGirthSteiner}), we then provide the construction to prove Theorem~\ref{thm:HighRootedGirthBooster} starting with a basic construction and then iterating it.

In Section~\ref{s:GirthOmniBoosters}, we prove Theorem~\ref{thm:HighGirthAbsorber}, our high girth omni-absorber theorem. First we define girth-$g$ projections for omni-boosters and state the high girth omni-booster theorem, Theorem~\ref{thm:HighGirthOmniBooster}, and show how it implies Theorem~\ref{thm:HighGirthAbsorber}. We then define quantum omni-boosters and state a quantum omni-booster theorem, Theorem~\ref{thm:HighGirthQuantumOmniBooster}, which implies Theorem~\ref{thm:HighGirthOmniBooster}. The latter is broken into two key lemmas, Lemmas~\ref{lem:RandomQuantumIntrinsic} and~\ref{lem:RandomQuantumExtrinsic} about the intrinsic (low max degree, large collective girth) and extrinsic properties (a regular enough projection treasury) of a random sparsification of a full quantum omni-booster.

In Section~\ref{s:Intrinsic}, we recall a corollary (Corollary~\ref{cor:KimVu}) of the Kim-Vu concentration inequality~\cite{KV00} and apply it many times to prove Lemma~\ref{lem:RandomQuantumIntrinsic}. Similarly in Section~\ref{s:Extrinsic}, we prove Lemma~\ref{lem:RandomQuantumExtrinsic} via many applications of said corollary.

In Section~\ref{s:CogirthPair}, we describe the modifications necessary for a proof of Theorem~\ref{thm:HighCogirthSteiner}. 
In Section~\ref{s:Conclusion}, we discuss have some concluding remarks and future directions.

\subsection{Notation}\label{ss:Notation}

A \emph{(multi-)hypergraph} $H$ consists of a pair $(V, E)$ where $V$ is a set whose elements are called \emph{vertices} and $E$ is a
(multi-)set of subsets of $V$ called \emph{edges}; we also write $H$ for its set of edges $E(H)$ for brevity. Similarly, we write $v(H)$ for the number of vertices of $H$ and either $e(H)$ or alternatively $|H|$ for the number of edges of $H$. 

For an integer $r\ge 1$, a hypergraph $H$ is said to be \emph{$r$-bounded} if every edge of $H$ has size at most $r$ and  \emph{$r$-uniform} if every edge has size exactly $r$; an \emph{$r$-uniform hypergraph} is called an \emph{$r$-graph} for short.

Let $F$ and $G$ be hypergraphs. An \emph{$F$ packing} of $G$ is a collection of copies of $F$ such that every edge of $G$ is in at most one element of $\mathcal{F}$.  An \emph{$F$ decomposition} of $G$ is an $F$ packing such that every edge of $G$ is in exactly one element of $\mathcal{F}$. 

For a (multi)-hypergraph $G$ and subset $S\subseteq V(G)$, we let $G(S)$ denote the set $\{e\in G: S\subseteq e\}$, which as is practice, we also view as a hypergraph, in particular it is a subhypergraph of $G$. Note this differs from the normal notation of using $G(S)$ to denote the related concept of a `link hypergraph' where the set $S$ is removed from every edge of $G(S)$; since, we do not need the link hypergraph in this paper but extensively use $G(S)$, we opted to define our notation as such. If $G$ is an $r$-graph and $|S|=r$, then $|G(S)|$ is called the \emph{multiplicity} of $S$. 

The complete $r$-graph on $q$ vertices, denoted $K_q^r$ is the $r$-graph with vertex set $[q]$ and edge set $\binom{[q]}{r}$. Note that a Steiner system with parameters $(n,q,r)$ is equivalent to a $K_q^r$ decomposition of $K_n^r$. A necessary condition for an $r$-graph $G$ to admit a $K_q^r$ decomposition is that $G$ is \emph{$K_q^r$-divisible}, that is, $\binom{q-i}{r-i}~|~|G(S)|$ for all $0\le i \le r-1$ and $S\subseteq V(G)$ with $|S|=i$. 

Let $G$ be a (multi-)hypergraph. For an integer $k\ge 1$, we let $G^{(k)}$ denote the $k$-uniform subhypergraph of $G$ consisting of all edges of $G$ of size $k$. If $G$ is uniform, then for a vertex $v \in V(G)$, we let $d_G(v)$ denote the number of edges of $G$ containing $v$. If $G$ is not uniform, then for an integer $i\ge 1$, we thus write $d_{G^{(i)}}(v)$ for the \emph{$i$-degree of $v$ in $G$}, which is defined to be the number of edges of $G$ of size $i$ containing $v$. 

If $G$ is uniform, then \emph{maximum $i$-codegree of $G$}, denoted $\Delta_i(G)$ is the maximum of $|G(U)|$ for $U\subseteq V(G)$ with $|U|=i$. If $G$ is $k$-uniform, then the \emph{maximum degree of $G$}, denoted $\Delta(G)$ is $\Delta_{k-1}(G)$; similarly we write $\delta(G)$ for the \emph{minimum degree of $G$} which is $\delta_{k-1}(G)$. If $G$ is not uniform, then the \emph{maximum $(s,t)$-codegree of $G$}, denoted $\Delta_{t}\left(G^{(s)}\right)$ is the maximum of $H^{(s)}(U)$ for $U\in \binom{V(G)}{t}$.
 
\section{Proof of the High Girth Existence Conjecture}\label{s:HighGirthExistence}

In this section, we state the necessary definitions and theorems for formalizing the outline of the proof of Theorem~\ref{thm:HighGirthSteiner} given in the introduction. Then we prove Theorem~\ref{thm:HighGirthSteiner} assuming said theorems.

\subsection{Random Subsets and Regularity Boosting}

Rephrasing a $K_q^r$ decomposition problem in terms of a perfect matching of an auxiliary hypergraph is useful. Indeed, R\"odl's proof of the Erd\H{o}s-Hanani conjecture proceeded by using his nibble method to show that ${\rm Design}_{K_q^r}(K_n^r)$ has an almost perfect matching and hence that $K_n^r$ has an almost $K_q^r$ decomposition. Since we desire a complete $K_q^r$ decomposition, it is useful to force the leftover (the undecomposed edges) into a sparser `reserve' set of edges $X$ which we will choose uniformly at random with some small probability $p$. Forcing the leftover into $X$ will be accomplished with our ``Forbidden Submatchings with Reserves" theorem from~\cite{DP22}. To that end, it is useful to define a `reserve' design hypergraph as follows.

\begin{definition}[Reserve Design Hypergraph]
Let $F$ be a hypergraph. If $G$ is a hypergraph and $A,B$ are disjoint subsets of $E(G)$, then the \emph{$F$-design reserve hypergraph of $G$ from $A$ to $B$}, denoted ${\rm Reserve}_F(G,A,B)$, 
is the bipartite hypergraph $\mathcal{D}=(A,B)$ with $V(\mathcal{D}):=A\cup B$ and $$E(\mathcal{D}) := \{S\subseteq A\cup B: S \text{ is isomorphic to } F,~|S\cap A|=1\}.$$
\end{definition}

To apply our ``Forbidden Submatchings with Reserves" theorem from~\cite{DP22}, we need that vertices of $A$ in the definition above have large enough degree in $\mc{D}$. To that end, we recall the following lemma from our earlier paper~\cite{DPI} (which follows easily from the Chernoff bound, e.g. see~\cite{AS16}) which handles step (1) of our proof outline:

\begin{lem}\label{lem:RandomX}
For all integers $q> r\ge 1$, there exists a real $\varepsilon \in (0,1)$ such that the following holds: Let $G$ be an $r$-uniform hypergraph with $\delta(G)\ge (1-\varepsilon)\cdot v(G)$ and let $p$ be a real number with $v(G)^{-\varepsilon} \le p \le 1$. If $v(G)\ge \frac{1}{\varepsilon}$, then there exists a spanning subhypergraph  $X\subseteq G$ such that $\Delta(X)\le 2p\cdot v(G)$, and for all $e\in G\setminus X$, there exist at least $\varepsilon\cdot p^{\binom{q}{r}-1}\cdot v(G)^{q-r}$ $K_q^r$'s in $X\cup \{e\}$ containing $e$.
\end{lem}

Similarly to handle step (3) of our proof outline, we used a special case of the Boost Lemma of Glock, K\"{u}hn, Lo, and Osthus~\cite{GKLO16} and applied it to $J:= G\setminus (A\cup X)$. Here, we need a stronger special case of said Boost Lemma as follows.

\begin{lem}\label{lem:RegBoost}
For all integers $q> r \ge 1$ and reals $d, \xi, \varepsilon > 0$ such that $d \ge \xi \ge 2(2\sqrt{e})^r \varepsilon$, there exists an integer $n_0\ge 1$ such that the following holds for all $n\ge n_0$: Let $J$ be an $r$-uniform hypergraph on $n$ vertices. If $\mathcal{J} \subseteq {\rm Design}_{K_q^r}(J)$ satisfies both of the following:
\begin{enumerate}
    \item[(a)] $\mathcal{J}$ is $(d\pm \varepsilon)\cdot n^{q-r}$-regular, and
    \item[(b)] for every $e\in E(J)=V(\mathcal{J})$, we have that
    $\left|\left\{ S\subseteq \binom{V(J)\setminus V(e)}{q}: \binom{S\cup V(e)}{q}\subseteq \mathcal{J} \right\}\right| \ge \xi\cdot n^q$,
\end{enumerate} 
then there exists a subhypergraph $H$ of $\mathcal{J}$ such that $d_H(v) = \left(1 \pm n^{-(q-r)/3}\right) \cdot \frac{d}{2} \cdot n^{q-r}$ for all $v\in V(H)$. 
\end{lem}

\subsection{Forbidden Submatching with Reserves}\label{ss:ForbiddenSubmatching}

To finish our proof of Theorem~\ref{thm:Existence}, we used a ``nibble with reserves'' theorem. Similarly, to finish the proof of Theorem~\ref{thm:HighGirthSteiner}, we need a ``forbidden submatching with reserves'' theorem from our work in~\cite{DP22}; in our original version of~\cite{DP22}, we prove a forbidden submatching bipartite hypergraph theorem. However, this does not seem to imply the more general `reserves' version. Hence in our latest version of~\cite{DP22}, we prove an even stronger version, a ``forbidden submatchings with reserves'' theorem. In order to state this theorem, we first need to recall a few definitions from~\cite{DP22}. 

First we recall notions of $A$-perfect matching and of a bipartite hypergraph.

\begin{definition}[$A$-perfect matching]
Let $G$ be a hypergraph and $A\subseteq V(G)$. A matching of $G$ is \emph{$A$-perfect} if every vertex of $A$ is in an edge of the matching.
\end{definition}

\begin{definition}[Bipartite Hypergraph]
We say a hypergraph $G=(A,B)$ is \emph{bipartite with parts $A$ and $B$} if $V(G)=A\cup B$ and  every edge of $G$ contains exactly one vertex from $A$. 
\end{definition}

Next, we need the notion of a configuration hypergraph, which here we define slightly more generally to allow $E(G)\subseteq V(H)$ instead of $E(G)=V(H)$ as this will be more convenient for our proofs.

\begin{definition}[Configuration Hypergraph]
Let $G$ be a (multi)-hypergraph. We say a hypergraph $H$ is a \emph{configuration hypergraph} for $G$ if $E(G)\subseteq V(H)$ and every edge $e$ of $H$ has size at least two and $e\cap E(G)$ is a matching of $G$. We say a matching of $G$ is \emph{$H$-avoiding} if it spans no edge of $H$.
\end{definition}

Finally we recall some more degree and codegree notions as follows.

\begin{definition}
Let $G$ be a hypergraph and let $H$ be a configuration hypergraph of $G$. We define the \emph{$i$-codegree} of a vertex $v\in V(G)$ and $e\in E(G)\subseteq V(H)$ with $v\notin e$ as the number of edges of $H$ of size $i$ who contain $e$ and an edge incident with $v$. We then define the \emph{maximum $i$-codegree} of $G$ with $H$ as the maximum $i$-codegree over vertices $v\in V(G)$ and edges $e\in E(G)\subseteq V(H)$ with $v\notin e$. 
\end{definition}

\begin{definition}
Let $G$ be a hypergraph and let $H$ be a configuration hypergraph of $G$. We define the \emph{common $2$-degree} of distinct vertices $u, v\in V(H)$ as $|\{w\in V(H): uw, vw\in E(H)\}|$. Similarly, we define the \emph{maximum common $2$-degree} of $H$ with respect to $G$ as the maximum of the common $2$-degree of $u$ and $v$ over all distinct pairs of vertices $u,v$ of $H$ where $u$ and $v$ are vertex-disjoint in $G$.    
\end{definition}

Here we introduce one more definition which simplifies the statement of the necessary theorem and will be useful later when discussing the projections of omni-absorbers.

\begin{definition}[Treasury]
Let $r,g \ge 2$ be integers. An \emph{$(r,g)$-treasury} is a tuple $T=(G_1,G_2,H)$ where $G_1$ is an $r$-uniform hypergraph and $G_2=(A,B)$ is an $r$-bounded bipartite hypergraph such that $V(G_1)\cap V(G_2)=A$, and $H$ is a $g$-bounded configuration hypergraph of $G_1\cup G_2$.

A \emph{perfect matching} of $T$ is an $H$-avoiding $(V(G_1)\cap V(G_2))$-perfect matching of $G_1\cup G_2$.

Let $D, \sigma\ge 1$ and $\beta,\alpha \in (0,1)$ be reals. We say that the treasury $T$ is \emph{$(D,\sigma,\beta,\alpha)$-regular} if all of the following hold:
\begin{itemize}
\item (Quasi-Regular)  every vertex of $G_1$ has degree at most $D$ in $G_1$ and every vertex of $A$ has degree at least $D-\sigma$ in $G_1$,
\item (Reserve/Weighted Degrees) every vertex of $B$ has degree at most $D$ in $G_2$, every vertex of $A$ has degree at least $D^{1-\alpha}$ in $G_2$, and $\Delta_1\left(H^{(i)}\right) \le \alpha \cdot D^{i-1}\log D$ for all $2\le i\le g$,
\item (Codegrees) $G_1\cup G_2$ has codegrees at most $D^{1-\beta}$,  ~$\Delta_{t}\left(H^{(s)}\right) \le D^{s-t-\beta}$ for all $2\le t< s\le g$, and the maximum $2$-codegree of $G$ with $H$ and the maximum common $2$-degree of $H$ with respect to $G$ are both at most $D^{1-\beta}$.
\end{itemize}
\end{definition}
It might be helpful to the reader to note that if $T$ is $(D,\sigma,\beta,\alpha)$-regular, then $T$ is also $(D,\sigma',\beta',\alpha')$-regular for all $\sigma'\ge \sigma$, $\beta' \le \beta$ and $\alpha'\ge \alpha$.

We are now ready to state the main nibble theorem from~\cite{DP22} we need as follows.

\begin{thm}[Forbidden Submatchings with Reserves]\label{thm:ForbiddenSubmatchingReserves}
For all integers $r,g \ge 2$ and real $\beta \in (0,1)$, there exist an integer $D_{\beta}\ge 0$ and real $\alpha > 0$ such that following holds for all $D\ge D_{\beta}$: 
\vskip.05in
If $T$ is an $(r,g)$-treasury that is $(D,D^{1-\beta},\beta,\alpha)$-regular, then there exists a perfect matching of $T$. 
\end{thm}

\subsection{Configuration Hypergraphs for High Girth Designs}

We in part developed the ``forbidden submatching'' machinery in~\cite{DP22} to prove the approximate version of Conjecture~\ref{conj:HighGirthAllUniformities}. It is useful to recall how the two are related.

\begin{definition}[Erd\H{o}s configuration]
Let $q>r\ge 1$ and $i\ge 3$ be integers. An \emph{$(i(q-r)+r,i)$ Erd\H{o}s-configuration} of a $r$-graph $G$ is an $(i(q-r)+r,i)$-configuration that does not contain an $(i'(q-r)+r,i')$-configuration for any $2\le i' < i$.
\end{definition}

\begin{definition}[Girth Configuration Hypergraph]
Let $q> r\ge 1$ be integers and let $g\ge 3$ be an integer. Let $G$ be an $r$-uniform hypergraph and let $\mathcal{D}:={\rm Design}_{K_q^r}(G)$. 
Then the \emph{girth-$g$ $K_q^r$-configuration hypergraph} of $G$, denoted ${\rm Girth}_{K_q^r}^g(G)$ is the configuration hypergraph $H$ of $\mathcal{D}$ with $V(H):=E(\mathcal{D})$ and $E(H) := \{S\subseteq E(\mathcal{D}): S \text{ is an $(i(q-r)+r,i)$ Erd\H{o}s-configuration of } G \text{ for some $3\le i\le g$}\}$.       \end{definition}

Note here that definition of Erd\H{o}s-configuration is only for $i\ge 3$. Thus two $q$-cliques that intersect in at least one edge is a $(2(q-r)+r,2)$ but not an Erd\H{o}s configuration. In addition, since an Erd\H{o}s configuration contains no $(i'(q-r)+r,i')$- configuration for $2\le i' < i$, this means that an an Erd\H{o}s configuration is a matching of $\mc{D}$ and hence all the edges of $H$ are matchings of $\mc{D}$.

We note that a weaker version (namely just the almost perfect matching version) of Theorem~\ref{thm:ForbiddenSubmatchingReserves} applied to ${\rm Design}_{K_q^r}(K_n^r)$ and ${\rm Girth}_{K_q^r}^g(K_n^r)$ suffices to prove the existence of approximate girth-$g$ designs for all $q > r\ge 2$. 

We now use Lemma~\ref{lem:RandomX} to make a similar version about a regular enough treasury but first a definition.

\begin{definition}\label{def:GirthTreasury}
Let $q > r \ge 1$ be integers and let $g\ge 3$ be an integer. Let $G$ be an $r$-uniform hypergraph and let $X,G'\subseteq G$. The \emph{girth-$g$ $K_q^r$-design treasury} of $G$ on $G'$ with reserve $X$, denoted ${\rm Treasury}_q^g(G,G',X)$, is the treasury $T:=\left({\rm Design}_{K_q^r}(G'\setminus X),~{\rm Reserve}_{K_q^r}(G, G'\setminus X, X),~{\rm Girth}_{K_q^r}^g(G)\right)$.
\end{definition}

We also need the following easy proposition (which for example follows from Proposition 3.18 in~\cite{DPI}).

\begin{proposition}\label{prop:Turan}
For all integers $q > r\ge 1$ and real $\alpha > 0$, there exists a real $\varepsilon\in (0,1)$ and integer $n_0\ge 1$ such that the following holds for all $n\ge n_0$: if $G$ is an $r$-uniform hypergraph on $n$ vertices with $\delta(G)\ge (1-\varepsilon)n$, then for every edge $e$ of $G$, there exists at least $(1-\alpha)\cdot \binom{n}{q-r}$ $K_q^r$'s that contain $e$.
\end{proposition}

\begin{lem}\label{lem:RandomXTreasury}
For all integers $q> r\ge 1$ and $g \ge 3$ and real $\alpha > 0$, there exist a real $\varepsilon \in (0,1)$ and integer $n_0 \ge 1$ such that the following holds: Let $G$ be an $r$-uniform hypergraph on $n\ge n_0$ vertices with $\delta(G)\ge (1-\varepsilon)\cdot n$. Then there exists a spanning subhypergraph  $X\subseteq G$ such that $\Delta(X)\le 2n^{1-\varepsilon}$, ${\rm Treasury}^g_q(G,G,X)$ is $\bigg(\binom{n}{q-r},~\alpha\cdot \binom{n}{q-r},~\frac{1}{2g(q-r)},~\alpha\bigg)$-regular. 
\end{lem}
\begin{proof}
Let $\varepsilon := \min \left\{\frac{\alpha}{2\binom{q}{r}},~\varepsilon'\right\}$ where $\varepsilon'$ is the value of Lemma~\ref{lem:RandomX} for $q$ and $r$. Similarly, we choose $n_0$ large enough so that $n$ satisfies various inequalities throughout the proof. Let $p:= n^{-\varepsilon}$. By Lemma~\ref{lem:RandomX}, there exists a spanning subhypergraph $X\subseteq G$ such that $\Delta(X)\le 2pn = 2n^{1-\varepsilon}$ and for all $e\in G\setminus X$, there exist at least $\varepsilon\cdot p^{\binom{q}{r}-1}\cdot v(G)^{q-r}$ $K_q^r$'s in $X\cup \{e\}$ containing $e$. 

Let ${\rm Treasury}_q^g(G,G,X) = (G_1,G_2,H)$. Let $D:= \binom{n}{q-r}$. Since every edge of $K_n^r$ is in at most $\binom{n}{q-r}$ copies of $K_q^r$ in $K_n^r$, we find that $\Delta(G_1),\Delta(G_2) \le D$. By choosing $\varepsilon$ small enough with respect to $\alpha$, it also follows from Proposition~\ref{prop:Turan} that $\delta(G_1)\ge (1-\alpha)D$ as desired. 

Meanwhile, every edge $e$ of $G\setminus X$ has degree in $G_2$ at least $\varepsilon \cdot p^{\binom{q}{r}-1} n^{q-r} \ge \varepsilon \cdot D^{1-\varepsilon\cdot \binom{q}{r}} \ge D^{1-\alpha}$ since $D$ is large enough and $\varepsilon \le \frac{\alpha}{2\binom{q}{r}}$. Furthermore, $$\Delta\left(H^{(i)}\right) \le 2^{\binom{(i-1)(q-r)}{2}} \cdot \binom{n}{(i-1)(q-r)} \le 2^{\binom{(i-1)(q-r)}{2}} \cdot D^{i-1} \le \alpha\cdot D^{i-1}\cdot \log D$$ 
since $D$ is large enough as $n$ is large enough.

As for codegrees, let $\beta := \frac{1}{2g(q-r)}$. Thus $G_1\cup G_2$ has codegrees at most $\binom{n}{q-r-1} \le D^{1-\frac{1}{2(q-r)}}$ for $D$ large enough, which is at most $D^{1-\beta}$ as desired. 

Meanwhile, for all $2\le t < s \le g$, 
$$\Delta_{t}\left(H^{(s)}\right) \le O\left(n^{(s-t)(q-r)-1}\right) = O\left(D^{s-t - \frac{1}{q-r}}\right) ,$$
since $t$ edges of $G$ contained in an edge of $H$ of size $s$ span at least $t(q-r)+r+1$ vertices of $G$ by construction. Since $\beta < \frac{1}{q-r}$ and $D$ is large enough, it follows  that  
$$ \Delta_{t}\left(H^{(s)}\right) \le D^{s-t-\beta}$$
as desired. 

On the other hand, since $H$ has no edges of size $2$, it follows that the maximum $2$-codegree of $G$ with $H$ and the maximum common $2$-degree of $H$ are both $0$ (and hence $H$ is $0$-uncommon) which is at most $D^{1-\beta}$ as desired.

Altogether then, we find that ${\rm Treasury}_q^g(G,G,X)$ is $\bigg(\binom{n}{q-r},~\alpha\cdot \binom{n}{q-r},~\frac{1}{2g(q-r)},~\alpha\bigg)$-regular as desired.
\end{proof}

\subsection{Girth-$g$ Projections for Omni-Absorbers}\label{ss:OmniProjection}

Now we turn toward defining the terms in the statement of Theorem~\ref{thm:HighGirthAbsorber}.

First we recall the proof strategy and resulting issues mentioned in the introduction as follows. Given the existence of boosters as in Theorem~\ref{thm:HighRootedGirthBooster}, then to build high girth omni-absorbers, we embed edge-disjoint rooted $K_q^r$-boosters of high rooted girth, one for each $S$ in the decomposition family $\mathcal{H}$ of our refined $K_q^r$-omni-absorber $A$ (where $S$ is the root $K_q^r$ of its booster). This will ensure that each $K_q^r$ possibly used to decompose $L\cup A$ has its girth `spread out' into the booster. 

However even this is not enough. The issue is that small girth configurations may arise both across the different boosters but also between the boosters and the $K_q^r$ decomposition of $K_n^r\setminus (X\cup A)$ that we find using nibble. To overcome the first issue, we ensure the girth boosters are embedded such that collectively they have high girth (as in Definitions~\ref{def:OmniAbsorberCollectiveGirth} and~\ref{def:OmniBooster}). 

To overcome the second issue, we will forbid any partial configurations in $K_n^r\setminus (X\cup A)$ that extend to full small girth configurations using $K_q^r$'s in the high girth boosters. To do both of these simultaneously, we will first randomly sparsify the set of girth boosters that could be used for each clique independently to a very small size (polylog in $n$ say). We will show with high probability under this sparsification that the resulting configuration degree of edges in $K_n^r\setminus (X\cup A)$ has not increased too much with the inclusion of these partial configurations. This is done via concentration inequalities. We also argue that with high probability there is now a choice inside each of these sets of a high girth absorber (that is edge-disjoint choices with overall high girth).

We now formalize the above discussion with following definitions. First we define a common projection so as to project out by all the possible absorbings of an omni-absorber. Note we define this in the setting of a treasury as we need to delete `singletons' (i.e. edges of size $1$) of the projected configuration hypergraph from the underlying graphs (i.e.~delete edges of $G_1$ and $G_2$ that are in nearly complete edges of $H$).

\begin{definition}
Let $T=(G_1,G_2,H)$ be a treasury. Let $\mathcal{M}=\{M_1,\ldots,M_k\}$ be such that for all $i\in [k]$, we have $M_i\subseteq V(H)$ and $M_i\cap (G_1\cup G_2)$ is a matching of $G_1\cup G_2$. The \emph{common projection of $T$ by $\mathcal{M}$}, denoted $T\perp \mathcal{M}$ is the treasury  $T'=(G_1',G_2',H')$ where for all $i\in \{1,2\}, $$V(G_i'):=V(G_i)$, and
$$V(H') := V(H)\setminus \{ v\in V(H): \exists S\in E(H),~i\in[k] \text{ such that } v\in S,~v\not\in M_i \text{ and } S\setminus \{v\} \subseteq M_i \}$$
and for all $i\in \{1,2\}$, $E(G_i') = E(G_i)\cap V(H')$ and
$$E(H'):= \{ T \subseteq V(H'): \exists S\in E(H),~i\in [k] \text{ such that } T\subseteq S,~T\cap M_i=\emptyset \text{ and } S\setminus T \subseteq M_i \}.$$
\end{definition}

Now given the above definition, we turn to defining the projection of an omni-absorber. 

\begin{definition}[Girth $g$ Projection of Omni-Absorber]\label{def:OmniAbsorberProj}
Let $G$ be an $r$-uniform hypergraph, let $X$ be a subgraph of $G$, and let $A\subseteq G$ be a $K_q^r$-omni-absorber of $X$ with decomposition function $\mathcal{Q}_A$. We define the \emph{girth-$g$ projection treasury} of $A$ on to $G$ and $X$ as:
$${\rm Proj}_g(A,G,X):= {\rm Treasury}_q^g(G,~G\setminus A,~X) \perp \mc{M}(A)$$
where $$\mathcal{M}(A) := \{ \mathcal{Q}_A(L) : L\subseteq X \text{ is $K_q^r$-divisible}\}.$$
\end{definition}

The following proposition shows how to use the above definition to find a high girth Steiner system.

\begin{proposition}\label{prop:FindSteiner}
Let $q>r\ge 1$ and $g\ge 2$ be integers. Let $G$ be a $K_q^r$-divisible $r$-uniform hypergraph, let $X$ be a subgraph of $G$ and let $A\subseteq G$ be a $K_q^r$-omni-absorber of $X$ of collective girth at least $g$. If there exists a perfect matching of ${\rm Proj}_g(A,G,X)$, then there exists a $K_q^r$ decomposition of $G$ of girth at least $g$.
\end{proposition}
\begin{proof}
Let $M$ be a perfect matching of ${\rm Proj}_g(A,G,X)$ that is an $H$-avoiding $(G\setminus (X\cup A))$-perfect matching of $G\setminus A$ where $H$ is as in the definition of girth-$g$ projection treasury. Note then that $M$ is equivalent to a $K_q^r$ packing of $G\setminus A$ that covers all edges of $G\setminus (A\cup X)$ with the additional property that $M\cup \mathcal{Q}_A(L)$ has girth at least $g$ for any $L\subseteq X$ such that $L$ is $K_q^r$-divisible. 

Now let $L := X\setminus \bigcup M$. Note that $G$ is $K_q^r$-divisible, $\bigcup M \subseteq G$ is $K_q^r$-divisible, and $A\subseteq G$ is also $K_q^r$-divisible and $A\cap \bigcup M = \emptyset$. Thus we have that $L:= G\setminus (A\cup \bigcup M)$ is $K_q^r$-divisible. But then $\mathcal{Q}_A(L)$ is a $K_q^r$-decomposition of $A\cup L$. Hence $M\cup \mathcal{Q}_A(L)$ is a $K_q^r$-decomposition of $G$ that as mentioned above has girth at least $g$ as desired.
\end{proof}

For our High Girth Omni-Absorber Theorem, it will be convenient for the proof to allow us to take a substructure of the treasury that is itself a treasury. This motivates the following definition.

\begin{definition}
Let $(G_1,G_2,H)$ be a treasury. We say a treasury $(G_1',G_2',H')$ is a \emph{subtreasury} of $(G_1,G_2,H)$ if $G_i'$ is a spanning subgraph of $G_i$ for each $i\in \{1,2\}$ and $H[E(G_1')\cup E(G_2')]$ is a subgraph of $H'$. 
\end{definition}

The next proposition shows that a perfect matching of a subtreasury is still a perfect matching of the treasury (and hence we may still apply Proposition~\ref{prop:FindSteiner} to find a $K_q^r$ decomposition).

\begin{proposition}\label{prop:SubTreasury}
Let $T'$ be a subtreasury of a treasury $T$. If $M$ is a perfect matching of $T'$, then $M$ is also a perfect matching of $T$.    
\end{proposition}
\begin{proof}
Let $T'=(G_1',G_2',H')$ and $T=(G_1,G_2,H)$. By definition $M$ is an $H'$-avoiding $(V(G_1')\cap V(G_2'))$-perfect matching of $G_1'\cup G_2'$. Since $G_i'$ is a spanning subgraph of $G_i$ for each $i\in \{1,2\}$, we have that $V(G_1')\cap V(G_2') = V(G_1)\cap V(G_2)$.  It follows then that $M$ is a $(V(G_1)\cap V(G_2))$-perfect matching of $G_1\cup G_2$. Furthermore, $M$ does not span an edge $e$ of $H$ as otherwise $e\in H[E(G_1')\cup E(G_2')] \subseteq H'$, contradicting that $M$ is $H'$-avoiding. Hence $M$ is $H$-avoiding. Thus $M$ is a perfect matching of $T$ as desired.
\end{proof}

We need the following proposition relating the projection treasuries 
of two sets of matchings where one is a subset of the other. We note the proof just follows directly from the various definitions and hence we omit it.

\begin{proposition}\label{prop:SubProjTreasury}
Let $T=(G_1,G_2,H)$ be a treasury. If $\mathcal{M}'\supseteq \mathcal{M}$, then $T\perp \mathcal{M}'$ is a subtreasury of $T\perp \mathcal{M}$.
\end{proposition} 

We are now prepared to state our high girth omni-absorber theorem as follows.

\begin{thm}[High Girth Omni-Absorber Theorem]\label{thm:HighGirthAbsorber}
For all integers $q>r\ge 1$ and $g\ge 3$ and real $\alpha \in \left(0,~\frac{1}{2(q-r)}\right)$, there exist integers $a,n_0\ge 1$ such that the following holds for all $n\ge n_0$: 

Suppose that Theorem~\ref{thm:HighCogirthBooster} holds for $r':=r-1$. If $X$ is a spanning subgraph of $K_n^r$ with $\Delta(X) \le \frac{n}{\log^{ga} n}$ such that ${\rm Treasury}^g_q(K_n^r,K_n^r,X)$ is $\bigg(\binom{n}{q-r},~\alpha\cdot \binom{n}{q-r},~\frac{1}{2g(q-r)},~\alpha\bigg)$-regular and we let $\Delta:= \max\left\{ \Delta(X),~n^{1-\frac{1}{r}}\cdot \log n\right\}$, then there exists a $K_q^r$-omni-absorber $A$ for $X$ with $\Delta(A)\le a\Delta\cdot \log^{a} \Delta$ and collective girth at least $g$, together with a subtreasury $T$ of ${\rm Proj}_g(A, K_n^r, X)$ such that $T$ is 
$$\Bigg(\binom{n}{q-r},~3\alpha \cdot \binom{n}{q-r},~\frac{1}{4g(q-r)},~2\alpha\Bigg){\rm - regular}.$$
\end{thm}

\subsection{Proof of the High Girth Existence Conjecture}

Here then is our proof of the High Girth Existence Conjecture.

\begin{lateproof}{thm:HighGirthSteiner}
Let $r':= \binom{q}{r}$. Let $\beta' := \frac{1}{8g(q-r)}$ and let $\alpha'$ be as in Theorem~\ref{thm:ForbiddenSubmatchingReserves} for $r'$, $g$, and $\beta'$. Let $\alpha$ be chosen small enough to satisfy certain inequalities throughout the proof, in particular $\alpha\le  \frac{\alpha'}{2^{g+1}}$ and $\alpha < \frac{1}{2(q-r)}$. 

Let $\varepsilon$ be as in Lemma~\ref{lem:RandomXTreasury} for $q,r,g$ and $\alpha$. By Lemma~\ref{lem:RandomXTreasury}, there exists a spanning subhypergraph  $X\subseteq K_n^r$ such that $\Delta(X)\le 2n^{1-\varepsilon}$ and ${\rm Treasury}^g_q(K_n^r,K_n^r,X)$ is $\bigg(\binom{n}{q-r},~\alpha\cdot \binom{n}{q-r},~\frac{1}{2g(q-r)},~\alpha\bigg)$-regular. In addition since $\Delta(X) \le 2n^{1-\varepsilon}$, we note that every edge of $X$ has degree in ${\rm Reserve}_{K_q^r}(K_n^r,K_n^r\setminus X, X)$ at most $\binom{n}{q-r-1}\cdot (2n^{1-\varepsilon}) \le \frac{1}{2}\cdot \binom{n}{q-r}$ where we used that $n$ is large enough.

Let $a$ is as in Theorem~\ref{thm:HighGirthAbsorber} for $q$, $r$, $g$ and $\alpha$. Note $\alpha < \frac{1}{2(q-r)}$. Since $n$ is large enough, we have that 
$$\Delta(X) \le 2n^{1-\varepsilon} \le \frac{n}{\log^{ga} n}.$$ 

Hence by Theorem~\ref{thm:HighGirthAbsorber}, there exists a $K_q^r$-omni-absorber $A$ for $X$ with $\Delta(A)\le a\Delta\cdot \log^{a} \Delta$ and collective girth at least $g$, together with a subtreasury $T=(G_1,G_2,H)$ of ${\rm Proj}_g(A, K_n^r, X)$ that is 
$$\Bigg(\binom{n}{q-r},~3\alpha \cdot \binom{n}{q-r},~\frac{1}{4g(q-r)},~2\alpha\Bigg){\rm - regular}.$$

Let $d:=\frac{1}{(q-r)!}$, $\xi:=\frac{1}{2(q-r)!}$, and $\varepsilon' :=4\alpha$. Let $J:=K_n^r\setminus (X\cup A)$ and $\mathcal{J}:=G_1$. Note that $\mathcal{J}$ is $(d\pm \varepsilon')\cdot n^{q-r}$-regular. Moreover, since $\alpha$ is small enough, it follows from Proposition~\ref{prop:Turan} that 
$$\left|\left\{ S\subseteq \binom{V(J)\setminus V(e)}{q}: \binom{S\cup V(e)}{q}\subseteq \mathcal{J} \right\}\right| \ge \xi\cdot n^q,$$
for all $e\in E(J)$. Thus Lemma~\ref{lem:RegBoost}(a) and (b) hold for $J$ and $\mathcal{J}$ with parameters: $q$, $r$, $d$, $\xi$ and $\varepsilon'$. Moreover, $d\ge \xi \ge 2(2\sqrt{e})^r\cdot \varepsilon'$ where the last inequality follows since $\alpha$ is small enough. Hence by Lemma~\ref{lem:RegBoost} since $n$ is large enough, there exists a subhypergraph $G_1'$ of $\mathcal{J}$ such that $d_{G_1'}(v) = \left(1 \pm n^{-(q-r)/3}\right) \cdot \frac{d}{2} \cdot n^{q-r}$ for all $v\in V(G_1')$.

Let $T'=(G_1',G_2, H[G_1'\cup G_2])$. Since $G_1'$ is a spanning subhypergraph of $G_1$, we have by definition that $T'$ is a subtreasury of $T$. Let
$$D':=\left(1+n^{-(q-r)/3}\right)\frac{d}{2}\cdot n^{q-r},~~\sigma' := 2\cdot n^{-(q-r)/3} \cdot \frac{d}{2} \cdot n^{q-r}.$$

\begin{claim}
$T'$ is $(D',\sigma',\beta',\alpha')$-regular. 
\end{claim}
\begin{proofclaim}
First we check the degrees of $G_1'$ as follows. First, we have from above that every vertex of $G_1$ has degree at most $D'$ in $G_1'$ and degree at least $D'-\sigma'$ in $G_1'$. 

Next we check the degrees of $G_2$. Let $D:= \binom{n}{q-r}$. Note that $D' \ge \frac{1}{2} \binom{n}{q-r}=\frac{D}{2}$. From before, we have that every vertex of ${\rm Reserve}_{K_q^r}(G,G\setminus X, X)$ has degree at most $\frac{1}{2} \binom{n}{q-r} \le D'$. Hence since $G_2\subseteq {\rm Reserve}_{K_q^r}(G,G\setminus X, X)$, we find that every vertex of $G_2$ has degree at most $D'$ in $G_2$.  

Since $T$ is regular with parameters as above, we find the following. Every vertex of $G_2$ has degree at least $D^{1-2\alpha}$ in $G_2$. Since $D\ge D'\ge \frac{D}{2}$, we find that
 $$D^{1-2\alpha} \ge \frac{D}{D^{2\alpha}} \ge \frac{D'}{(2D')^{2\alpha}} \ge (D')^{1-2\alpha}\cdot \frac{1}{2^{2\alpha}} \ge (D')^{1-3\alpha} \ge (D')^{1-\alpha'}$$ 
as desired since $D$ is large enough and $\alpha' \ge 2^{g+1} \cdot \alpha \ge 3\alpha$ (as $g\ge 3$).  

Let $H':= H[G_1'\cup G_2]$. Next we check $\Delta_i(H')$ as follows. Note that $H'\subseteq H$. Similarly by the definition of regular, we have that $\Delta_i(H) \le (2\alpha)\cdot D^{i-1}\cdot \log D$ for all $2\le i \le g$. Hence for all $2\le i\le g$, we find since $D' \ge \frac{D}{2}$ that
$$\Delta\left(H'^{(i)}\right) \le \Delta\left(H^{(i)}\right) \le 2\alpha \cdot (2D')^{i-1} \log (2D') \le 2^{g+1}\cdot \alpha \cdot (D')^{i-1} \log D' \le \alpha' \cdot (D')^{i-1} \cdot \log D'$$
as desired since $\alpha' \ge 2^{g+1}\cdot \alpha$.

Finally we check the various codegrees as follows. Note that $\frac{1}{4g(q-r)} = 2\beta'$. By definition of regular we have that: $G_1\cup G_2$ has codegrees at most $D^{1-2\beta'}$, $\Delta_{t}(H^{(s)}) \le D^{s-t-2\beta'}$ for all $2\le t< s\le g$, and the maximum $2$-codegree of $G_1\cup G_2$ with $H$ and the maximum common $2$-degree of $H$ are both at most $D^{1-2\beta'}$. 

Since $G_1'\subseteq G_1$ and $H'\subseteq H$, we have that the same holds for $G_1'$, $G_2$, and $H'$. Moreover since $D\le 2D'$, we have for $m\le g$, that 
$$D^{m-2\beta'} \le (2D')^{m-2\beta'} \le 2^g \cdot (D')^{m-2\beta'} \le (D')^{m-\beta'}$$
since $D$ is large enough. Hence $D^{1-2\beta'}\le (D')^{1-\beta'}$ and for all $2\le t< s\le g$, $D^{s-t-2\beta'}\le (D')^{s-t-\beta'}$. 

Thus we find that: $G_1'\cup G_2$ has codegrees at most $(D')^{1-\beta'}$, $\Delta_{t}(H'^{(s)}) \le (D')^{s-t-\beta'}$ for all $2\le t< s\le g$, and the maximum $2$-codegree of $G_1'\cup G_2$ with $H'$ and the maximum common $2$-degree of $H'$ are both at most $(D')^{1-\beta'}$.
\end{proofclaim}
 
Since $\beta' < \frac{1}{3}$ and $n$ is large enough, it follows that $\sigma' \le (D')^{1-\beta'}$. Furthermore, $D'\ge D_{\beta'}$ where $D_{\beta'}$ is as in Theorem~\ref{thm:ForbiddenSubmatchingReserves} for $r',g,$ and $\beta'$.

Thus by Theorem~\ref{thm:ForbiddenSubmatchingReserves}, there exists a perfect matching $M$ of $T'$. By Proposition~\ref{prop:SubTreasury}, $M$ is also a perfect matching of $T$. Similarly then by Proposition~\ref{prop:SubTreasury}, $M$ is also a perfect matching of ${\rm Proj}_g(A, K_n^r, X)$. Thus by Proposition~\ref{prop:FindSteiner}, we have that there exists a $K_q^r$-decomposition of $K_n^r$ of girth at least $g$ as desired.
\end{lateproof}

\section{High Girth Boosters}\label{s:GirthBoosters}

In this section, we prove  Theorem~\ref{thm:HighRootedGirthBooster} that there exist $K_r$-boosters of arbitrarily high rooted girth. 

\subsection{The Necessity of High Cogirth Boosters}

To prove Theorem~\ref{thm:HighGirthAbsorber} and hence Theorem~\ref{thm:HighGirthSteiner}, we will prove Theorem~\ref{thm:HighRootedGirthBooster} that there exist rooted $K_q^r$-boosters with rooted girth at least $g$. To do this, we proceed by induction on uniformity and use the existence of $K_{q-1}^{r-1}$-boosters to build $K_q^r$-boosters. However to ensure the resulting $K_q^r$-boosters have high rooted girth, we will need the existence of high girth boosters of lower uniformity with an even stronger property. This motivates the following definition.

\begin{definition}[Cogirth]\label{def:Cogirth}
Let $q > r\ge 1$ be integers. Let $\mathcal{B}_1$ and $\mathcal{B}_2$ be two $K_q^r$-packings of a graph $G$. We define the \emph{cogirth} of $\mathcal{B}_1$ and $\mathcal{B}_2$ as the smallest integer $g\ge 2$ for which $\mathcal{B}_1\cup \mathcal{B}_2$ contains an $(g(q-r)+r-1,g)$-configuration.

Similarly, if $B$ is a $K_q^r$-booster with $K_q^r$ decomposition $\mathcal{B}_1,\mathcal{B}_2$, we define the \emph{cogirth} of $B$ as the minimum of the girth of $\mathcal{B}_1$, the girth of $\mathcal{B}_2$, and the cogirth of $\mathcal{B}_1$ and $\mathcal{B}_2$.
\end{definition}

Having high cogirth is a stronger property than having high rooted girth since having high rooted girth follows from high cogirth by only considering subsets of $S\cup \mathcal{B}_{{\rm on}}$, and then only those that contain $S$.


Thus to inductively prove the existence of high rooted girth boosters (Theorem~\ref{thm:HighRootedGirthBooster}), we desire to prove the following theorem.

\begin{thm}[High Cogirth Booster]\label{thm:HighCogirthBooster}
Let $q > r\ge 1$ and $g\ge 2$ be integers. There exists a $K_q^r$-booster $B$ with cogirth at least $g$.
\end{thm}

Note that every $K_q^1$-packing has infinite girth. Thus for the $r=1$ case of Theorem~\ref{thm:HighCogirthBooster}, only the cogirth condition is restrictive. Still, even then it is quite interesting. Namely, the $r=1$ case of Theorem~\ref{thm:HighCogirthBooster} is equivalent to the existence of a $q$-regular bipartite graph of girth at least $g$ (by letting $E(G)=V(B)$ and $V(G)=\mathcal{B}_1\cup \mathcal{B}_2$ and vice versa). It is a well-known fact that $q$-regular bipartite graphs of girth at least $g$ exist for all integers $q,g\ge 2$.

For $r\ge 2$ though, the only way we were able to construct $B$ as desired in Theorem~\ref{thm:HighCogirthBooster} is by letting $B=K_n^r$ for $n$ admissible and large enough. That is, we will prove a stronger form of Theorem~\ref{thm:HighGirthSteiner} as follows.

\begin{thm}[Existence of High Cogirth Designs]\label{thm:HighCogirthSteiner}
For all integers $q > r \geq 1$ and every integer $g\ge 2$, there exists $n_0 \ge 1$ such that for all admissible $n\ge n_0$ the following holds: 

There exist two $(n,q,r)$-Steiner systems both with girth at least $g$ and cogirth at least $g$. 
\end{thm}

Indeed, the necessity of proving that a pair of high girth designs with high cogirth exists is perhaps to be expected. In the proof of the Existence Conjecture by Glock, K\"{u}hn, Lo, and Osthus~\cite{GKLO16}, in order to prove the existence of absorbers, it is necessary to prove the existence of a $K_q^r$-booster (note decompositions have by definition girth at least 3) such that no two $K_q^r$'s intersect in more that $r$ vertices (that is, no $(2(q-r)+r-1,2)$-configuration exists or equivalently the two decompositions have cogirth at least $3$). The way Glock, K\"{u}hn, Lo, and Osthus~\cite{GKLO16} found to do this for $r\ge 2$ was by proving the existence of two $(n,q-1,r-1)$-Steiner systems with cogirth at least $3$. Similarly, in order to prove the existence of high girth absorbers with a purely combinatorial approach and lacking other constructions, it is natural that we need to prove the existence of a high cogirth pair of high girth designs as in Theorem~\ref{thm:HighCogirthSteiner}.

\subsection{The First Construction}

To construct our high rooted girth boosters, we first show that we can almost construct them except for one special clique $S'$ of $\mathcal{B}_2$ as follows.

\begin{lem}\label{lem:HighGirthBoost}
Let $q\ge r\ge 2$ and $g\ge 3$ be integers. If Theorem~\ref{thm:HighCogirthBooster} holds for $r:'=r-1$, then there exists a $K_q^r$-booster $B$ with $K_q^r$ decompositions $\mathcal{B}_1,\mathcal{B}_2$ with $\mathcal{B}_1\cap \mathcal{B}_2 = \emptyset$ and $S\in \mathcal{B}_1$, $S'\in \mathcal{B}_2$ with $|S\cap S'|=q-1$ such that 
\begin{enumerate}
    \item[(1)] $\mathcal{B}_1$ and $\mathcal{B}_2$ have girth at least $g$, and
    \item[(2)] $\mathcal{B}_2\setminus \{S'\}$ has rooted girth at $V(S)\cup V(S')$ at least $g$.
\end{enumerate}
\end{lem}
\begin{nonlateproof}{lem:HighGirthBoost}
Since Theorem~\ref{thm:HighCogirthBooster} holds for $r':=r-1$ by assumption, we have that there exists a $K_{q-1}^{r-1}$-booster $B'$ with disjoint $K_q^r$-decompositions $\mathcal{B}_1'$, $\mathcal{B}_2'$ of $B'$ both of girth at least $g+1$ and cogirth at least $g+1$. Let $B$ be the $r$-uniform hypergraph with
$$V(B):= V(B')\cup \{v_1,v_2\},$$
and
$$E(B):= \bigg\{ \{v_i\}\cup e: e\in E(B'),~i\in [2]\bigg\}~\cup \bigcup_{Q\in \mathcal{B}_1'\cup\mathcal{B}_2'} \binom{Q}{r}.$$
Note that since $\mathcal{B}_1'$ and $\mathcal{B}_2'$ have cogirth at least $g$, we have that $\bigcup_{Q\in\mathcal{B}_1'} \binom{Q}{r}$ is disjoint from $\bigcup_{Q\in \mathcal{B}_2'} \binom{Q}{r}$.
For $i\in [2]$ and $Q\in \mathcal{B}_1'\cup \mathcal{B}_2'$, let
$$\phi_i(Q):= \bigg\{ \{v_i\}\cup e: e\in Q\bigg\} \cup \binom {Q}{r}.$$
Then let
$$\mathcal{B}_1:= \{ \phi_1(Q): \mathcal{Q}\in \mathcal{B}_1'\} \cup \{ \phi_2(Q): \mathcal{Q}\in \mathcal{B}_2'\},$$
and
$$\mathcal{B}_2:= \{ \phi_1(Q): \mathcal{Q}\in \mathcal{B}_2'\} \cup \{ \phi_2(Q): \mathcal{Q}\in \mathcal{B}_1'\}.$$

Note that $\mathcal{B}_1$ and $\mathcal{B}_2$ are $K_q^r$ decompositions of $B$.

\begin{claim}\label{claim:girthg}
$\mathcal{B}_1$ and $\mathcal{B}_2$ have girth at least $g+1$. 
\end{claim}
\begin{proofclaim}
By symmetry, it suffices to prove the claim for $\mathcal{B}_1$. To that end, let $\mathcal{T} \subseteq \mathcal{B}_1$ with $|\mathcal{T}|=i$ for some $2\le i \le g$. Let $T:= \bigcup \mathcal{T}$. It suffices to prove that $|V(T)|\ge (q-r)\cdot i + r + 1$.

First suppose $|V(T)\cap \{v_1,v_2\}|=1$. We assume without loss of generality that $v_1\in V(T)$. Hence by construction, we have for every $T\in \mathcal{T}$ that $\phi_1^{-1}(T) \in \mathcal{B}_1'$. Let $\mathcal{T}' := \{\phi_1^{-1}(T):~T\in \mathcal{T}\}$ and let $T' := \bigcup \mathcal{T}'$. Since $\mathcal{B}_1'$ has girth at least $g+1$, we have that
$$|V(T')| \ge ((q-1)-(r-1))\cdot i + (r-1)+1.$$
Since $|V(T)|=|V(T')|+1$, it follows that $|V(T)|\ge (q-r)\cdot i + r+1$ as desired.

So we assume that $|V(T)\cap \{v_1,v_2\}|=2$. Let $\mathcal{T}':= \bigcup_{i\in[2]} \{\phi_i^{-1}(T):~T\in \mathcal{T},~v_i\in V(T)\}$ and let $T':= \bigcup \mathcal{T}'$. Since $\mathcal{B}_1'$ and $\mathcal{B}_2'$ have cogirth at least $g+1$, we have that
$$|V(T')|\ge ((q-1)-(r-1))\cdot i + (r-1).$$
Since $|V(T)|=|V(T')|+2$, it follows that $|V(T)|\ge (q-r)\cdot i + r+1$ as desired.
\end{proofclaim}

Let $S\in \mathcal{B}_1$. We may assume without loss of generality that $S=\phi_1(Q)$ for some $Q\in\mathcal{B}_1'$. Let $S':=\phi_2(Q)$. Note by definition that $|S\cap S'|=|Q|=q-1$.

\begin{claim}
$\mathcal{B}_2\setminus \{S'\}$ has rooted girth at $V(S)\cup V(S')$ at least $g$.    
\end{claim}
\begin{proofclaim}
Let $\mathcal{T} \subseteq \mathcal{B}_2\setminus \{S'\}$ with $|\mathcal{T}|=i$ for some $2\le i \le g-1$. Let $T:= \bigcup \mathcal{T}$. It suffices to prove that $|V(T)\setminus (V(S)\cup V(S'))|\ge (q-r)\cdot i$.

Let $\mathcal{T}':= \mathcal{T}\cup \{S'\}$ and let $T':= \bigcup \mathcal{T}'$. By Claim~\ref{claim:girthg}, we have that $\mathcal{B}_2$ has girth at least $g+1$. Since $|\mathcal{T}'|=i+1\le g$, it follows that 
$$|V(T')|\ge (q-r)\cdot (i+1) + r + 1 = (q-r)\cdot i + q + 1.$$
Thus we find that
\begin{align*}|V(T)\setminus (V(S)\cup V(S'))| &\ge |V(T)| - |V(S)\cup V(S')| \\
&\ge (q-r)\cdot i + q + 1 - (q+1) = (q-r)\cdot i,
\end{align*}
as desired
\end{proofclaim}

\noindent Hence (1) and (2) hold as desired.
\end{nonlateproof}

\subsection{Iterating the Construction}

We are now prepared to prove our high rooted girth booster theorem as follows.

\begin{lateproof}{thm:HighRootedGirthBooster}
Let $B$ be a $K_q^r$-booster with $K_q^r$ decompositions $\mathcal{B}_1,\mathcal{B}_2$ and $S\in \mathcal{B}_1$, $S'\in \mathcal{B}_2$ with $|S\cap S'|=i$ for some $i\in [r-1]_0$ such that 
\begin{enumerate}
    \item[(1)] $\mathcal{B}_1$ and $\mathcal{B}_2$ have girth at least $g$, and
    \item[(2)] $\mathcal{B}_2\setminus \{S'\}$ has rooted girth at $V(S)\cup V(S')$ at least $g$,
\end{enumerate}
and subject to those conditions $i$ is minimized. Note such a choice of booster exists by Lemma~\ref{lem:HighGirthBoost} (where $i=r-1$).

First suppose that $i=0$. By (1), $\mathcal{B}_1$ and $\mathcal{B}_2$ have girth at least $g$. We claim that $\mathcal{B}_2$ has rooted girth at $V(S)$ at least $g$. To see this let, $\mathcal{T}\subseteq \mathcal{B}_2$ with $\mathcal{T}=j$ for some $2\le j \le g-1$. Let $T:= \bigcup \mathcal{T}'$. It suffices to prove that $|V(T)\setminus V(S)|\ge (q-r)\cdot i$. If $S'\not\in \mathcal{T}$, then this follows from (2). So we assume $S'\in \mathcal{T}$. Since $i=0$, we have that $|V(S')\setminus V(S)|=q$. By (2), we have that $|V(T)\setminus (V(S)\cup V(S'))|\ge (q-r)\cdot (i-1)$. Hence it follows that $|V(T)\setminus V(S)|\ge q + (q-r)\cdot (i-1) \ge (q-r)\cdot i$ as desired.

So we assume $i > 0$. By Lemma~\ref{lem:HighGirthBoost}. There exists a $K_q^r$ booster $B'$ with $K_q^r$ decompositions $\mathcal{B}_1', \mathcal{B}_2'$ where $S'\in \mathcal{B}_1'$, $S''\in \mathcal{B}_2$ with $|S'\cap S''|=r-1$ such that Lemma~\ref{lem:HighGirthBoost}(1)-(2) hold. Crucially we assume here that $V(B')\cap V(B) = V(S')$, $E(B)\cap E(B') = E(S')$ and that $V(S')\setminus V(S'') \subseteq V(S)\cap V(S')$ (which is possible since $V(S)\cap V(S')\ne \emptyset$). It follows that $|S\cap S''| < |S\cap S'| = i$. 

Let $B'' := B\cup B'$ (where $E(S')$ appears only once). Let $\mathcal{B}_1'' := \mathcal{B}_1\cup (\mathcal{B}_1'\setminus \{S'\})$ and let $\mathcal{B}_2'' := (\mathcal{B}_2\setminus \{S'\})\cup\mathcal{B}_2'$. Note that $\mathcal{B}_1''$ and $\mathcal{B}_2''$ are $K_q^r$ decompositions of $B''$. Moreover, $S\in \mathcal{B}_1''$ and $S''\in \mathcal{B}_2''$.

We claim that (1) and (2) hold for $B''$, $\mathcal{B}_1''$, $\mathcal{B}_2''$, $S$ and $S''$, contradicting the choice of $B$, $\mathcal{B}_1$, $\mathcal{B}_2$, $S$ and $S'$ since $|S\cap S''| < i$. 

\begin{claim}\label{claim:girthg2}
$\mathcal{B}_1''$ and $\mathcal{B}_2''$ have girth at least $g$.
\end{claim}
\begin{proofclaim}
First we prove the claim for $\mathcal{B}_1''$. To that end, let $\mathcal{T}_1'' \subseteq \mathcal{B}_1''$ with $2\le |\mathcal{T}''| \le g-1$. Let $T_1'':= \bigcup \mathcal{T}_1''$. It suffices to prove that $|V(T_1'')|\ge (q-r)\cdot |\mathcal{T}_1''| + r + 1$.

Let $\mathcal{T}_1 := \mathcal{T}_1''\cap \mathcal{B}_1$ and let $V(T_1) = \bigcup \mathcal{T}_1$. Similarly let $\mathcal{T}_1' := \mathcal{T}_1''\cap (\mathcal{B}_1'\setminus \{S'\})$ and let $T_1' := \bigcup \mathcal{T}_1'$. Since $\mathcal{B}_1''$ is the disjoint union of $\mathcal{B}_1$ and $\mathcal{B}_1'\setminus \{S'\}$, we have that $\mathcal{T}''$ is the disjoint union of $\mathcal{T}$ and $\mathcal{T}'$.

Since $\mathcal{B}_1$ and $\mathcal{B}_1'$ have girth at least $g$ by (1) and Lemma~\ref{lem:HighGirthBoost}(1) respectively, we may assume that $\mathcal{T}_1,\mathcal{T}_1'\ne\emptyset$. Since $\mathcal{B}_1$ has girth at least $g$, it follows that $$|V(T_1)|\ge (q-r)\cdot |\mathcal{T}_1| + r$$ 
(note there is no $+1$ term here since $|\mathcal{T}_1|$ may equal $1$). 

Further note that $|\mathcal{T}_1'| \le g-2$. Thus since $\mathcal{B}_1'$ has girth at least $g$, we find that 
$$|V(T_1'\cup S')| \ge (q-r)\cdot (|\mathcal{T}_1'|+1)+r+1=(q-r)\cdot |\mathcal{T}_1'| + q + 1.$$
Thus 
$$|V(T_1')\setminus V(B)| =|V(T_1'\cup S')|-q\ge (q-r)\cdot |\mathcal{T}_1'|+1.$$
Combining we find that
$$|V(T_1'')| = |V(T_1)|+|V(T_1')\setminus V(B')| \ge (q-r)\cdot |\mathcal{T}_1| + r + (q-r)\cdot |\mathcal{T}_1'| + 1 = (q-r)\cdot |\mathcal{T}_1''| + r + 1,$$
as desired.

Finally we prove the claim for $\mathcal{B}_2''$. To that end, let $\mathcal{T}_2'' \subseteq \mathcal{B}_2''$ with $2\le |\mathcal{T}_2''| \le g-1$. Let $T_2'':= \bigcup \mathcal{T}_2''$. It suffices to prove that $|V(T_2)|\ge (q-r)\cdot |\mathcal{T}_2''| + r + 1$.

Let $\mathcal{T}_2 := \mathcal{T}_2''\cap (\mathcal{B}_2\setminus \{S'\})$ and let $V(T_2) = \bigcup \mathcal{T}_2$. Similarly let $\mathcal{T}_2' := \mathcal{T}_2''\cap \mathcal{B}_2'$ and let $T_2' := \bigcup \mathcal{T}_2'$. Since $\mathcal{B}_2''$ is the disjoint union of $\mathcal{B}_2\setminus \{S'\}$ and $\mathcal{B}_2'$, we have that $\mathcal{T}_2''$ is the disjoint union of $\mathcal{T}_2$ and $\mathcal{T}_2'$.

Since $\mathcal{B}_2$ and $\mathcal{B}_2'$ have girth at least $g$ by (1) and Lemma~\ref{lem:HighGirthBoost}(1) respectively, we may assume that $\mathcal{T}_2,\mathcal{T}_2'\ne\emptyset$. Since $\mathcal{B}_2'$ has girth at least $g$, it follows that $$|V(T_2')|\ge (q-r)\cdot |\mathcal{T}_2'| + r$$ 
(note there is no $+1$ term here since $|\mathcal{T}_2'|$ may equal $1$). 

Further note that $|\mathcal{T}_2| \le g-2$. Thus since $\mathcal{B}_2$ has girth at least $g$, we find that 
$$|V(T_2\cup S')| \ge (q-r)\cdot (|\mathcal{T}_2|+1)+r+1=(q-r)\cdot |\mathcal{T}_2| + q + 1.$$
Thus 
$$|V(T_2)\setminus V(B')| =|V(T_2\cup S')|-q\ge (q-r)\cdot |\mathcal{T}_2|+1.$$
Combining we find that
$$|V(T_2'')| = |V(T_2')|+|V(T_2)\setminus V(B')| \ge (q-r)\cdot |\mathcal{T}_2'| + r + (q-r)\cdot |\mathcal{T}_2| + 1 = (q-r)\cdot |\mathcal{T}_2''| + r + 1,$$
as desired.
\end{proofclaim}

\begin{claim}
$\mathcal{B}_2''\setminus \{S''\}$ has rooted girth at $V(S)\cup V(S'')$ at least $g$.
\end{claim}
\begin{proofclaim}
Let $\mathcal{T}'' \subseteq \mathcal{B}_2''\setminus \{S''\}$ with $|\mathcal{T}''|$ where $1\le |\mathcal{T}''| \le g-1$. Let $T'':= \bigcup \mathcal{T}''$. It suffices to prove that $|V(T'')\setminus (V(S)\cup V(S''))|\ge (q-r)\cdot |\mathcal{T}''|+r+1$.

Let $\mathcal{T} := \mathcal{T}''\cap (\mathcal{B}_2\setminus \{S'\})$ and let $V(T) = \bigcup \mathcal{T}$. Similarly let $\mathcal{T}' := \mathcal{T}''\cap \mathcal{B}_2'$ and let $T' := \bigcup \mathcal{T}'$. Since $\mathcal{B}_2''$ is the disjoint union of $\mathcal{B}_2\setminus \{S'\}$ and $\mathcal{B}_2'$, we have that $\mathcal{T}''$ is the disjoint union of $\mathcal{T}$ and $\mathcal{T}'$.

Note that $|\mathcal{T}|, |\mathcal{T}'| \le |\mathcal{T}''| \le g-1$. By (2), we have that $\mathcal{B}_2\setminus \{S'\}$ has rooted girth at $V(S)\cup V(S')$ at least $g$. Hence 
$$|V(T)\setminus (V(S)\cup V(S')| \ge (q-r)\cdot |\mathcal{T}|.$$
By Lemma~\ref{lem:HighGirthBoost}, we have that $\mathcal{B}_2'$ has rooted girth at $V(S')\cup V(S'')$ at least $g$. 
Hence 
$$|V(T')\setminus (V(S')\cup V(S'')| \ge (q-r)\cdot |\mathcal{T}'|.$$
Note that $V(T)\cap (V(S'')\setminus V(S')) = \emptyset$ since $V(T)\subseteq V(B)$. Similarly $V(T')\cap (V(S)\setminus V(S')) = \emptyset$ since $V(T')\subseteq V(B')$. Combining we find that
\begin{align*}
|V(T'')\setminus (V(S)\cup V(S''))| &= |V(T)\setminus (V(S)\cup V(S'))| + |V(T')\setminus (V(S')\cup V(S''))| \\
&\ge (q-r)\cdot |\mathcal{T}| + (q-r)\cdot |\mathcal{T}'| = (q-r)\cdot |\mathcal{T}''|,
\end{align*}
as desired.
\end{proofclaim}

\end{lateproof}

\section{High Girth Omni-Boosters}\label{s:GirthOmniBoosters}

In this section, we state our high girth omni-booster theorem (Theorem~\ref{thm:HighGirthOmniBooster}) and show how it implies Theorem~\ref{thm:HighGirthAbsorber}. As mentioned in the introduction to prove our high girth omni-booster theorem, we will sparsify a set of possible omni-boosters and show that inside of these sparsified sets of omni-boosters lies a choice of high girth omni-booster as desired. To that end, we define a \emph{quantum omni-booster} (see Definition~\ref{def:QuantumOmniBooster}) which is like an omni-booster except that every element of $\mc{H}$ is permitted a set of boosters. We then state our high girth quantum omni-booster theorem (Theorem~\ref{thm:HighGirthQuantumOmniBooster}) and show it implies Theorem~\ref{thm:HighGirthOmniBooster}. Finally, we will show how Theorem~\ref{thm:HighGirthQuantumOmniBooster} follows by taking a \emph{random $p$-sparsification} of a \emph{full} omni-booster with the appropriate choice of $p$; this is done via two lemmas, Lemma~\ref{lem:RandomQuantumIntrinsic} which shows that certain intrinsic properties of the quantum booster (small maximum degree, large collective girth) hold with high probability and Lemma~\ref{lem:RandomQuantumExtrinsic} which shows that certain extrinsic properties of the quantum booster (namely its projection treasury is regular enough) hold with high probability. Lemma~\ref{lem:RandomQuantumIntrinsic} is proved in Section~\ref{s:Intrinsic} and Lemma~\ref{lem:RandomQuantumExtrinsic} is proved in Section~\ref{s:Extrinsic}.

For the purposes of brevity in the following theorem and lemma statements, it is useful to codify the setup via the following definition.

\begin{definition}
We say a triple $\mc{S}=(G,A,X)$ is a \emph{$K_q^r$-sponge} if $G$ is an $r$-uniform hypergraph, $X\subseteq G$ and $A\subseteq G\setminus X$ is a $K_q^r$-omni-absorber for $X$. We will also write $\mc{S}=(G,A,X,\mc{H})$ where $\mc{H}$ is the decomposition family of $A$ if we desire to specify the decomposition family.

We say $\mc{S}$ is 
\begin{itemize}
    \item \emph{$(\alpha,g)$-regular} if ${\rm Treasury}_q^g(G,~G\setminus A,~X)$ is $\bigg(\binom{v(G)}{q-r},~\alpha\cdot \binom{v(G)}{q-r},~\frac{1}{2g(q-r)},~\alpha\bigg)$-regular, 
    \item \emph{$a$-bounded} if $\Delta(X)\le \frac{v(G)}{\log^a v(G)}$, 
    \item \emph{$C$-refined} if $A$ is $C$-refined and $\Delta(A)\le C\cdot \max\left\{ \Delta(X),~v(G)^{1-\frac{1}{r}}\cdot \log v(G)\right\}$.
\end{itemize}
\end{definition}

\subsection{Girth-$g$ Projections of Omni-Boosters}

Here is the canonical omni-absorber decomposition function that an omni-booster yields.

\begin{proposition}\label{prop:CanonicalBoost}
Let $q > r\ge 1$ be integers. Let $(G,A,X,\mc{H})$ be a $K_q^r$-sponge and let $\mc{Q}_A$ denote the decomposition function of $A$. If $B$ is a $K_q^r$-omni-booster for $A$ and $X$ in $G$ with booster family $\mathcal{B}=(B_H: H\in\mathcal{H})$, then $A\cup B$ is a $K_q^r$-omni-absorber for $X$ with decomposition family $\bigcup_{H\in \mathcal{H}} (\mathcal{B}_H)_{{\rm on}}\cup (\mathcal{B}_H)_{{\rm off}}$ and decomposition function
$$\mathcal{Q}_{A\cup B}(L):= \bigcup_{H\in \mathcal{Q}_A(L)} (\mathcal{B}_H)_{{\rm on}}~\cup \bigcup_{H \in \mathcal{H}\setminus \mathcal{Q}_A(L)} (\mathcal{B}_H)_{{\rm off}}.$$    
We refer to $A\cup B$ with decomposition function $\mathcal{Q}_{A\cup B}$ defined above as the \emph{canonical omni-absorber} for $B$ and $\mathcal{B}$.
\end{proposition}

For the purposes of our proof, instead of showing that $\mathcal{Q}_{A\cup B}$ has high collective girth, we will show that all $K_q^r$ packings that are the union of `on'/`off' decompositions are high girth. This will be easier to analyze instead of considering what the girth of $\mathcal{Q}_{A\cup B}(L)$ could be as this would be dependent on the structure of $A$ and $L$. 
To that end, we make the following definition.

\begin{definition}[Girth $g$ Projection of Omni-Booster]\label{def:OmniBoosterProj}
Let $(G,A,X,\mc{H})$ be a $K_q^r$-sponge. Let $B$ a $K_q^r$-omni-booster for $A$ with booster family $\mathcal{B}=(B_H:H\in\mathcal{H})$. 

We define the \emph{girth-$g$ projection treasury} of $B$ on to $G$ and $X$ as:
$${\rm Proj}_g(B,A,G,X):= {\rm Treasury}_q^g(G,~G\setminus (A\cup B),~X) \perp \mathcal{M}(\mc{B}).$$
\end{definition}

The following proposition relates the treasury of an omni-booster and its canonical omni-absorber.

\begin{proposition}\label{prop:BoosterSubTreasury}
${\rm Proj}_g(B,A,G,X)$ is a subtreasury of ${\rm Proj}_g(A\cup B,G,X)$.
\end{proposition}
\begin{proof}
Let $\mathcal{M}(\mc{B})$ be as in Definition~\ref{def:OmniBoosterProj} for ${\rm Proj}_g(B,A,G,X)$ and let $\mathcal{M}(A\cup B)$ be as in Definition~\ref{def:OmniAbsorberProj} for the canonical omni-absorber $A\cup B$ as in Proposition~\ref{prop:CanonicalBoost}. We note that $\mathcal{M}(A\cup B)\subseteq \mathcal{M}(\mc{B})$. Hence by Proposition~\ref{prop:SubTreasury}, we find that ${\rm Proj}_g(B,A,G,X)$ is a subtreasury of ${\rm Proj}_g(A\cup B,G,X)$ as desired. 
\end{proof}

We are now prepared to state our high girth omni-booster theorem as follows.

\begin{thm}[High Girth Omni-Booster Theorem]\label{thm:HighGirthOmniBooster}
For all integers $q>r\ge 1$, $C\ge 1$ and $g\ge 3$ and real $\alpha \in \left(0,~\frac{1}{2(q-r)}\right)$, there exist integers $a,n_0\ge 1$ such that the following holds for all $n\ge n_0$: 

Suppose that Theorem~\ref{thm:HighCogirthBooster} holds for $r':=r-1$. If $(K_n^r, A, X)$ is an $(\alpha,g)$-regular $ga$-bounded $C$-refined $K_q^r$-sponge, then there exists a $K_q^r$-omni-booster $B$ for $A$ and $X$ with $\Delta(B)\le a\Delta\cdot \log^{a} \Delta$ and collective girth at least $g$, and a subtreasury $T$ of ${\rm Proj}_g(B,A,K_n^r,X)$ such that $T$ is 
$$\Bigg(\binom{n}{q-r},~3\alpha \cdot \binom{n}{q-r},~\frac{1}{4g(q-r)},~2\alpha\Bigg)-{\rm regular}.$$
\end{thm}

Assuming Theorem~\ref{thm:HighGirthOmniBooster}, we are now prepared to prove Theorem~\ref{thm:HighGirthAbsorber} as follows.

\begin{proof}[Proof of Theorem~\ref{thm:HighGirthAbsorber}]
Let $C_0$ be as in Theorem~\ref{thm:Omni} for $q$ and $r$ and let $C'$ be as in Theorem~\ref{thm:HighGirthOmniBooster} for $q$, $r$, $C_0$, $g$ and $\alpha$. Let $C:= \max\{C_0,C'\}$. By Theorem~\ref{thm:Omni}, there exists a $C_0$-refined $K_q^r$-omni-absorber $A_0$ for $X$ with $\Delta(A_0)\le C_0 \cdot \Delta$. By Theorem~\ref{thm:HighGirthOmniBooster}, there exists a $K_q^r$-omni-booster $B$ for $A_0$ and $X$ as in Theorem~\ref{thm:HighGirthOmniBooster}. Let $A:= A_0\cup B$. By Proposition~\ref{prop:CanonicalBoost}, we have that $A$ is a $K_q^r$-omni-absorber for $X$. Moreover, there exists a subtreasury $T$ of ${\rm Proj}_g(B,A_0,K_n^r,X)$ that is $\bigg(\binom{n}{q-r},~2\alpha \cdot \binom{n}{q-r},~\frac{1}{4g(q-r)},~2\alpha\bigg)-{\rm regular}$. By Proposition~\ref{prop:BoosterSubTreasury}, ${\rm Proj}_g(B,A_0,K_n^r,X)$ is a subtreasury of ${\rm Proj}_g(A,K_n^r,X)$. Thus it follows from the definition of subtreasury that $T$ is also a subtreasury of ${\rm Proj}_g(A,K_n^r,X)$. Hence $A$ is as desired.
\end{proof}

\subsection{Quantum Boosters}

As mentioned in the introduction, it behooves us for probabilistic purposes to choose each possible booster $B_H$ for each $H\in\mathcal{H}$ in the decomposition family of $A$ independently at random with some small probability; that is, instead of choosing one booster for each $H$, we will choose $p$ so that the set of chosen boosters for an $H$ is of size ${\rm polylog} n$. To that end, we define a more general object than an omni-booster, a \emph{quantum omni-booster}, where each booster is instead a set of boosters (so the booster is in a state of flux, a superposition of possible boosters; hence the name quantum).

\begin{definition}[Quantum Omni-Booster]\label{def:QuantumOmniBooster}
Let $q > r\ge 1$ be integers. Let $(G,A,X,\mc{H})$ be a $K_q^r$-sponge. A \emph{quantum $K_q^r$-omni-booster} for $A$ and $X$ in $G$ is an $r$-uniform hypergraph $B$ together with a family $\mc{B}=(\mathcal{B}_H: H\in \mathcal{H})$ (called the \emph{quantum booster collection}) where for each $H\in \mathcal{H}$, $\mathcal{B}_H$ is a family of (not necessarily-edge-disjoint) $K_q^r$-boosters rooted at $H$ in $G\setminus (X\cup A)$. 

The \emph{matching set} of $\mc{B}$ is 
$$\mc{M}(\mc{B}):= \bigg\{ M \in \prod_{H\in \mathcal{H}} \bigcup_{B_H\in \mathcal{B}_H} \big\{ (B_{H})_{{\rm on}},~ (B_{H})_{{\rm off}} \big\} : M \text{ is a matching of } {\rm Design}_{K_q^r}(G)\bigg\}.$$

For each $H\in \mathcal{H}$, we define 
$${\rm Disjoint}(\mathcal{B}_{H}) := \left\{B_H \in \mathcal{B}_{H}:~ B_H \text{ is (edge-)disjoint from all of }\bigcup_{H' \in \mathcal{H}\setminus \{H\}} \bigcup \mathcal{B}_{H'}\right\}.$$
For an integer $g\ge 3$, we define
\begin{align*}{\rm HighGirth}_g(\mathcal{B}_{H}) := \bigg\{B_H \in \mathcal{B}_{H}: \nexists &R\in {\rm Girth}_{K_q^r}^g(G) \text{ where } R\subseteq M \text{ for some } M \in \mc{M}(\mc{B}) \\
&\text{ and } R\cap \big( (B_H)_{\rm on}\cup (B_H)_{\rm off}\big)\ne \emptyset. \bigg\}.
\end{align*}
\end{definition}

Now we define a projection treasury for a quantum omni-booster where we project out the matching set as follows. Note that here we take the treasury ${\rm Treasury}_q^g(G,~G\setminus A,~X)$ instead of ${\rm Treasury}_q^g(G,~G\setminus (A\cup B),~X)$ since we do not yet know which boosters  of the quantum booster will be taken to form our actual omni-booster. This means though that when we do actualize $B$, we also have to delete $B$ from ${\rm Design}_{K_q^r}(G\setminus (X\cup A))$ which will be acceptable regularity-wise provided $B$ has small maximum degree.

\begin{definition}[Girth $g$ Projection of Quantum Omni-Booster]\label{def:QuantumOmniBoosterProj}
Let $(G,A,X,\mc{H})$ be a $K_q^r$-sponge. Let $B$ a quantum $K_q$-omni-booster for $A$ with quantum booster collection $\mathcal{B} = (\mathcal{B}_H: H\in\mathcal{H})$.

We define the \emph{girth-$g$ projection treasury} of $B$ on to $G$ and $X$ as: 
$${\rm Proj}_g(B,A,G,X):= {\rm Treasury}_q^g(G,~G\setminus A,~X) \perp \mathcal{M}(\mc{B}).$$
\end{definition}

The following proposition shows that choosing a booster $B_H$ for each $H\in \mathcal{H}$ from a quantum booster $B$ yields a booster $B'$ (provided the boosters are edge-disjoint) and furthermore, the quantum booster projection treasury for $B$ is a subtreasury of the projection treasury of $B'$. We omit the proof since it follows from the definitions.

\begin{proposition}\label{prop:QuantumSubTreasury}
Let $(G,A,X,\mc{H})$ be a $K_q^r$-sponge. Let $B$ a quantum $K_q$-omni-booster for $A$ with quantum booster collection $\mathcal{B} = (\mathcal{B}_H: H\in\mathcal{H})$. Suppose that $\mathcal{B}'=(B_H:H\in \mathcal{H})$ where $B_H\in \mathcal{B}_H$ for each $H\in \mathcal{H}$ and the $B_H$ are pairwise edge-disjoint. Then $B' := \bigcup_{H\in \mathcal{H}} B_H$ is a $K_q^r$-omni-booster for $A$ with booster family $\mathcal{B}'$ and ${\rm Proj}(B',A,G,X)$ is a subtreasury of ${\rm Proj}(B,A,G,X)$.    
\end{proposition}

Thus to prove Theorem~\ref{thm:HighGirthOmniBooster}, we will prove the following theorem about quantum omni-boosters.

\begin{thm}[High Girth Quantum Omni-Booster Theorem]\label{thm:HighGirthQuantumOmniBooster}
For all integers $q>r\ge 1$, $C\ge 1$ and $g\ge 3$ and real $\alpha \in \left(0,~\frac{1}{2(q-r)}\right)$, there exist integers $a, n_0\ge 1$ such that the following holds for all $n\ge n_0$: 

Suppose that Theorem~\ref{thm:HighCogirthBooster} holds for $r':=r-1$. If $(K_n^r, A, X,\mc{H})$ is an $(\alpha,g)$-regular $ga$-bounded $C$-refined $K_q^r$-sponge, then there exists a quantum $K_q^r$-omni-booster $B$ for $A$ and $X$ with quantum booster collection $(\mathcal{B}_H: H\in \mathcal{H})$ such that 
\begin{enumerate}
    \item[(1)] $\Delta(B)\le \Delta\cdot \log^{a} \Delta$ and for each $H\in \mathcal{H}$, ${\rm Disjoint}(\mc{B}_H) \cap {\rm HighGirth}_g(\mc{B}_H) \ne \emptyset$, and
    \item[(2)] ${\rm Proj}_g(B,A,K_n^r,X)$ is $\Bigg(\binom{n}{q-r},~2\alpha \cdot \binom{n}{q-r},~\frac{1}{4g(q-r)},~2\alpha\Bigg)-{\rm regular}$.
\end{enumerate} 
\end{thm}

We are now prepared to prove Theorem~\ref{thm:HighGirthOmniBooster}.

\begin{proof}[Proof of Theorem~\ref{thm:HighGirthOmniBooster}]
By Theorem~\ref{thm:HighGirthQuantumOmniBooster}, there exists quantum $K_q^r$-omni-booster $B$ for $A$ and $X$ with quantum booster collection $(\mathcal{B}_H: H\in \mathcal{H})$ such that  Theorem~\ref{thm:HighGirthQuantumOmniBooster}(1) and (2) hold. for each $H\in \mc{H}$, choose one $B_H\in {\rm Disjoint}(\mc{B}_H) \cap {\rm HighGirth}(\mc{B}_H)$. Then we let $B':=\bigcup_{H\in \mc{H}} B_H$ and $\mc{B}':=(B_H: H\in \mc{H})$. Since $B_H\in {\rm Disjoint}(\mc{B}_H)$ for all $H\in \mc{H}$, we find by Proposition~\ref{prop:QuantumSubTreasury} that $B'$ is  a $K_q^r$-omni-booster for $A$ and $X$ with booster family $\mc{B}'$. Furthermore since $B_H\in {\rm HighGirth}(\mc{B}_H)$ for all $H\in \mc{H}$, we find that $B'$ has collective girth at least $g$.

Let ${\rm Proj}_g(B,A,K_n^r,X)=(G_1,G_2,H)$ and let ${\rm Proj}_g(B',A,K_n^r,X)=(G_1',G_2',H')$. We let $G_1'':= G_1\cap {\rm Design}_{K_q^r}(K_n^r\setminus (X\cup A\cup B))$ and we let $G_2'':= G_2[K_n\setminus (X\cup A\cup B)]$, or equivalently $G_2''$ is obtained from deleting the edges of $B$ (which are also vertices of $G_2$). We let $T:=(G_1'',G_2'',H)$. Note that for $i\in \{1,2\}$, $G_i''$ is a spanning subgraph of $G_i'$. Similarly since $\mc{M}(\mc{B}') \subseteq \mc{M}(\mc{B})$, we find that $H'[E(G_1'')\cup E(G_2'')]\subseteq H$. Hence $T$ is a subtreasury of ${\rm Proj}_g(B',A,K_n^r,X)$.  

Since Theorem~\ref{thm:HighGirthQuantumOmniBooster}(2) holds, we have that 
${\rm Proj}_g(B,A,K_n^r,X)$ is $\Bigg(\binom{n}{q-r},~2\alpha \cdot \binom{n}{q-r},~\frac{1}{4g(q-r)},~2\alpha\Bigg)-{\rm regular}$. Given the definitions in order to show that $T$ is $\Bigg(\binom{n}{q-r},~3\alpha \cdot \binom{n}{q-r},~\frac{1}{4g(q-r)},~2\alpha\Bigg)-{\rm regular}$, it suffices to show that $d_{G_1}(v)-d_{G_1''}(v) \le \alpha\cdot \binom{n}{q-r}$ for all $v\in V(G_1'')$. However, this follows since 
$$\Delta(B)\le \Delta\cdot \log^a \Delta \le \frac{n}{\log^{ga} n}\cdot \log^a \Delta \le \frac{n}{\log^{(g-1)a}n}$$ 
where we used Theorem~\ref{thm:HighGirthQuantumOmniBooster}(1), $\Delta \le n$ and that the $K_q^r$-sponge $(K_n^r,A,X)$ is $ga$-bounded, and hence for all edges $e$ in $K_n^r\setminus B$, the number of $q$-cliques containing $e$ and edge of $B$ is at most 
$$q!\cdot \Delta(B)\cdot n^{q-r} \le \alpha\cdot \binom{n}{q-r}$$
as desired since $a\ge 1$, $g\ge 2$ and $n$ is large enough.
\end{proof}

\subsection{Random Sparsification of a Full $B_0$-type Quantum Booster}

As mentioned, in order to prove Theorem~\ref{thm:HighGirthQuantumOmniBooster}, we will fix a rooted $K_q^r$ booster $B_0$ of rooted girth at least $g$ and then randomly sparsify the set of possible copies of $B_0$ by choosing each independently with probability $p$. To that end, we formalize the definition of ``possible copies of $B_0$" as follows.

\begin{definition}[Full $B_0$-Type Omni-Booster]
Let $q > r\ge 1$ be integers. Let $B_0$ be a rooted $K_q^r$-booster. Let $(G,A,X,\mc{H})$ be a $K_q^r$-sponge. We say a quantum $K_q^r$-omni-booster $B$ for $A$ and $X$ in $G$ with quantum booster collection $(\mathcal{B}_H: H\in \mathcal{H})$ is \emph{$B_0$-type} if every $\mathcal{B}_H$ is a copy of $B_0$ with root $H$.

For each $H\in \mc{H}$, we let $\mc{S}(H)$ denote the set of subsets $S\subseteq K_{v(G)}^r\setminus (X\cup A)$ with $|S|=v(B_0)-q$ such that there exists a copy of $B_0$ in $G\setminus (X\cup A)$ rooted at $V(H)$ with $V(B_0)\subseteq V(H)\cup S$.

We say $B$ is \emph{full} if for every $H\in\mc{H}$, $\mc{B}_H$ consists of exactly one such copy of $B_0$ for $S\in \mc{H}$. We denote the set of full $B_0$-type $K_q^r$-omni-boosters for $A$ and $X$ in $G$ as ${\rm Full}_{B_0}(G,A,X)$.
\end{definition}

We note that we could have allowed full to mean all copies of $B_0$ in $G\setminus (X\cup A)$ instead of just one for each $S\in \mc{H}$, but it seemed simpler for the purposes of counting to only use one per $S$. Next we define the notion of a random $p$-sparsification of a quantum omni-booster as follows.

\begin{definition}
Let $p\in [0,1]$. The \emph{random $p$-sparsification} of a quantum booster $\mathcal{B} = (\mc{B}_H: H\in \mc{H})$ is the quantum booster $\mathcal{B}_p := ( (\mc{B}_H)_p: H\in \mc{H})$ where $(\mc{B}_H)_p$ denotes the subset of $\mc{B}_H$ obtained by choosing each element independently with probability $p$.   
\end{definition}

We thus prove Theorem~\ref{thm:HighGirthQuantumOmniBooster} by fixing a rooted $K_q^r$ booster $B_0$ of rooted girth at least $g$ and considering the random $p$-sparsification $\mc{B}_p$ of an arbitrary element $\mc{B}\in {\rm Full}_{B_0}(G,A,X)$. Namely, we will show with high probability that $\mc{B}_p$ satisfies outcomes (1) and (2) of Theorem~\ref{thm:HighGirthQuantumOmniBooster}.

Since the proof that (1) and (2) hold with high probability are long and very technical, we will split the proof into the following two lemmas, the first proves that (1) holds with high probability (the intrinsic properties) while the second proves that (2) holds with high probability (the extrinsic properties).

\begin{lem}\label{lem:RandomQuantumIntrinsic}
For all integers $q>r\ge 1$, $C\ge 1$ and $g\ge 3$, real $\alpha > 0$, and $K_q^r$ booster $B_0$ of rooted girth at least $g$, there exists an integer $a_1\ge 1$ such that for all $a\ge a_1$, there exists an integer $n_1\ge 1$ such that the following holds for all $n\ge n_1$: 

Let $(K_n^r, A, X,\mc{H})$ be an $(\alpha,g)$-regular $ga$-bounded $C$-refined $K_q^r$-sponge. If $\mc{B}\in {\rm Full}_{B_0}(K_n^r,A,X)$ and $p :=\frac{\log^{a} n}{\binom{n}{v(B_0)-q}}$, then with probability at least $0.9$, $\mc{B}_p$ satisfies that $\Delta(\bigcup \mc{B}_p)\le a\Delta\cdot \log^{a} \Delta$ and for each $H\in \mathcal{H}$, ${\rm Disjoint}(\mc{B}_H) \cap {\rm HighGirth}_g(\mc{B}_H) \ne \emptyset$. (i.e.~outcome Theorem~\ref{thm:HighGirthQuantumOmniBooster}(1) holds).
\end{lem}

\begin{lem}\label{lem:RandomQuantumExtrinsic}
For all integers $q>r\ge 1$, $C\ge 1$ and $g\ge 3$, real $\alpha \in \left(0,~\frac{1}{2(q-r)}\right)$, and $K_q^r$ booster $B_0$ of rooted girth at least $g$, there exists an integer $a_2\ge 1$ such that for all $a\ge a_2$, there exists an integer $n_2\ge 1$ such that the following holds for all $n\ge n_2$: 

Let $(K_n^r, A, X,\mc{H})$ be an $(\alpha,g)$-regular $ga$-bounded $C$-refined $K_q^r$-sponge. If $\mc{B}\in {\rm Full}_{B_0}(K_n^r,A,X)$ and $p :=\frac{\log^{a} n}{\binom{n}{v(B_0)-q}}$, then with probability at least $0.9$, $\mc{B}_p$ satisfies that ${\rm Proj}_g(B,A,K_n^r,X)$ is $\Bigg(\binom{n}{q-r},~2\alpha \cdot \binom{n}{q-r},~\frac{1}{4g(q-r)},~2\alpha\Bigg)-{\rm regular}$. (i.e.~outcome Theorem~\ref{thm:HighGirthQuantumOmniBooster}(2) holds).
\end{lem}

We note that the above lemmas do not require the condition ``Suppose that Theorem~\ref{thm:HighCogirthBooster} holds for $r':=r-1$" since the intended use is assumed with the existence of $B_0$. 

Assuming Lemmas~\ref{lem:RandomQuantumIntrinsic} and~\ref{lem:RandomQuantumExtrinsic}, we are now prepared to prove Theorem~\ref{thm:HighGirthQuantumOmniBooster} as follows.

\begin{proof}[Proof of Theorem~\ref{thm:HighGirthQuantumOmniBooster}]
We let $a := \max \{a_1,a_2\}$ where $a_1$ is as in Lemma~\ref{lem:RandomQuantumIntrinsic} and $a_2$ is as in Lemma~\ref{lem:RandomQuantumExtrinsic}. We let $n_1$ be as in Lemma~\ref{lem:RandomQuantumIntrinsic} for $a$ and $n_2$ be as in Lemma~\ref{lem:RandomQuantumExtrinsic} for  $a$. We set $n_0:= \max\{n_1,n_2\}$. We let $\mc{B}\in {\rm Full}_{B_0}(K_n^r,A,X)$ and we define $p :=\frac{\log^{a} n}{\binom{n}{v(B_0)-q}}$. By Lemma~\ref{lem:RandomQuantumIntrinsic}, $\mc{B}_p$ satisfies Theorem~\ref{thm:HighGirthQuantumOmniBooster}(1) with probability at least $.9$. Similarly by Lemma~\ref{lem:RandomQuantumExtrinsic}, $\mc{B}_p$ satisfies Theorem~\ref{thm:HighGirthQuantumOmniBooster}(2) with probability at least $.9$. Hence by the union bound, there exists some $\mc{B}_p$ satisfying both Theorem~\ref{thm:HighGirthQuantumOmniBooster}(1) and (2) as desired.
\end{proof}

\section{Intrinsic Properties of a Random Sparsification of a Full Quantum Omni-Booster}\label{s:Intrinsic}

In this section, we prove Lemma~\ref{lem:RandomQuantumIntrinsic}. First we recall a corollary of the Kim-Vu polynomial concentration inequality. We will use said corollary extensively in the proofs of Lemma~\ref{lem:RandomQuantumIntrinsic} and~\ref{lem:RandomQuantumExtrinsic}. Second we establish properties of a full quantum booster, namely upper bounds on the number of boosters containing a fixed set of cliques. Third we use these upper bounds to provide upper bounds on the degrees and codegrees of an auxiliary hypergraph that encodes low girth configurations among boosters. Finally we use those bounds and the corollary of Kim-Vu to prove Lemma~\ref{lem:RandomQuantumIntrinsic}.

\subsection{Probabilistic Tools: Kim-Vu and a Corollary}

Here we need a hypergraph corollary of the polynomial concentration theorem of Kim-Vu~\cite{KV00} to concentrate degrees and codegrees for our proof of Theorem~\ref{thm:HighGirthOmniBooster}. To that end, we need the following definition.

\begin{definition}[$p$-random subhypergraph]
Let $H$ be a hypergraph and $p\in [0,1]$ be real. The \emph{$p$-vertex-random subhypergraph of $H$}, which for this paper we will denote as $H_p$, is the subgraph of $H$ induced on the set $X$ where $X$ is obtained by choosing each vertex of $H$ independently with probability $p$.    
\end{definition}

Note since in this paper we will not consider random hypergraphs where the randomness is in choosing edges (such as is done in our other paper with Kelly~\cite{DKPIII}), there should be no confusion as to the use of $H_p$.

Here is our needed hypergraph corollary which we stated in a more general form in our paper with Kelly~\cite{DKPIII} (where we included a proof deriving it from Kim-Vu); that said, the below is a fairly standard application so we omit the proof.

\begin{cor}\label{cor:KimVu}
For all integers $k\ge 1$, there exists $n_0 \ge 1$ such that the following holds for all $p\in (0,1]$ and $n\ge n_0$: Let $H$ be a $k$-uniform multi-hypergraph on at most $n$ vertices. If $K>0$ satisfies
\begin{equation}
    p^{i}\ge \frac{\Delta_i(H)}{K}\cdot \log^{4k+2} n\tag{$*_i$}
\end{equation}
for all $i \in [k]$ and $e(H)\le K$, then with probability at least $1-n^{-\log n}$,
$$e(H_p)\le p^k \cdot 2K.$$
\end{cor}

\subsection{Properties of a Full Quantum Booster}

To prove Lemma~\ref{lem:RandomQuantumIntrinsic}, it is useful to upper bound the number of boosters containing an edge, a clique, or a set of cliques. To that end, we make the following definitions.

\begin{definition}
Let $(G,A,X,\mc{H})$ be a $K_q^r$-sponge. Let $\mc{B} \in {\rm Full}_{B_0}(G,A,X)$ for some rooted $K_q^r$-booster $B_0$. For $e\in G$, we let
$$\mc{B}(e) := \{B\in \mc{B}: e\in B\}.$$
For a ${\rm Girth}_{K_q^r}^g(G)$-avoiding matching $Z$ of ${\rm Design}_{K_q^r}(G)$ with $|Z|\le g-2$, we  let
$$\mc{B}(Z):= \{B\in \mc{B}: Z\subseteq B_{\rm off} \text{ or } Z\subseteq B_{\rm on}\},$$
$$\mc{B}_{\rm small}(Z):= \{B_H \in \mc{B}(Z): |V(H)\cap V(Z)|\le r-1\},~~\mc{B}_{\rm large}(Z):=\mc{B}(Z)\setminus \mc{B}_{\rm small}(Z).$$
For a partition of such a $Z$ into $m$ nonempty sets $Z_1,\ldots, Z_m$, we let
$$\mc{B}(Z_1,\ldots, Z_m):= \prod_{i\in [m]} \mc{B}(Z_i),$$
$$\mc{B}_{\rm small}(Z_1,\ldots,Z_m):= \prod_{i\in [m]} \mc{B}_{\rm small}(Z_i),~~\mc{B}_{\rm large}(Z_1,\ldots, Z_m):=\mc{B}(Z_1,\ldots,Z_m)\setminus \mc{B}_{\rm small}(Z_1,\ldots,Z_m).$$
\end{definition}

Here is the key lemma for upper bounding the number of boosters containing a fixed edge or a fixed set of cliques.

\begin{lem}\label{lem:BoostersContainingSetOfCliques}
For all integers $q>r\ge 1$, $C\ge 1$ and $g\ge 3$, and $K_q^r$ booster $B_0$ of rooted girth at least $g$, there exists an integer $C'\ge 1$ such that the following holds: Let $(G,A,X,\mc{H})$ be a $C$-refined $K_q^r$-sponge and let $\Delta := \max\left\{ \Delta(X),~v(G)^{1-\frac{1}{r}}\cdot \log v(G)\right\}$ and $b:=v(B_0)-q$. Let $\mc{B} \in {\rm Full}_{B_0}(G,A,X)$.
If $e\in G$, then
$$|\mc{B}(e)| \le C'\cdot \Delta\cdot n^{b-1}.$$
Furthermore, if $Z$ is a ${\rm Girth}_{K_q^r}^g(G)$-avoiding matching of ${\rm Design}_{K_q^r}(G)$ with $|Z|\le g-2$, then both of the following hold:
\begin{enumerate}
    \item[(1)] $|\mc{B}_{\rm small}(Z)| \le C'\cdot \frac{\Delta}{n}\cdot \binom{n}{b} \cdot n^{-(q-r)\cdot |Z| - {\bf 1}_{|Z|\ge 2}}.$
    \item[(2)] $|\mc{B}_{\rm large}(Z)|\le C'\cdot \binom{n}{b} \cdot n^{-(q-r)\cdot |Z|}.$    
\end{enumerate}
\end{lem}
\begin{proof}
We choose $C'$ large enough as needed throughout the proof.

First let $e\in G$. We assume $e\in G \setminus (X\cup A)$ as otherwise $\mc{B}(e)=\emptyset$ and there is nothing to show. Now let $F\subseteq V(e)$ with $|F|=i$ where $i\in [r-1]$. We note that since $A$ is $C$-refined there are at most $n^{r-1-i}\cdot \Delta(A) \cdot C$ ($\le C^2 \cdot \Delta \cdot n^{r-1-i}$) elements $H$ of $\mathcal{H}$ such that $F\subseteq V(H)$, and for each such $H$ there are at most $n^{b-(r-i)}$ subsets $S\in \mathcal{S}(H)$ such that $e\in B_H(S)$. Hence 
$$|\mc{B}(e)|\le \sum_{i=1}^{r-1} C^2\cdot \Delta \cdot n^{r-1-i} \cdot n^{b-(r-i)} \le r\cdot C^2 \cdot \Delta\cdot n^{b-1} \le C'\cdot \Delta\cdot n^{b-1},$$
as desired where we used that $C'$ is large enough.

Now we prove that (1) and (2) hold for a fixed $Z$. To that end, fix $F\subseteq V(Z)$ with $|F|\in \{0,1\ldots,q\}$. (Here $F$ will be the possible intersection of $V(Z)$ with $V(H)$ for $H\in \mc{H}$.) 

Consider $H\in \mc{H}$ with $F\subseteq V(H)$. If $|F| \le r-1$, then
$$|\{H\in \mc{H}: F\subseteq V(H)\}| \le n^{r-1-|F|}\cdot \Delta(A)\cdot C \le C^2\cdot \Delta \cdot n^{r-1-|F|}.$$
For each such $H$, there are at most $n^{b - (v(Z)-|F|)}$ elements $B$ of $\mc{B}_H$ with $V(Z) \subseteq V(B)$.
Hence
\begin{align*}|\mathcal{B}_{\rm small}(Z)| &\le \sum_{i=0}^{r-1} \binom{v(Z)}{i}\cdot  C^2\cdot \Delta \cdot n^{r-1-i} \cdot n^{b-(v(Z)-i)} \le v(B_0)^r\cdot r\cdot C^2 \cdot \Delta\cdot n^{b-v(Z)+r-1}\\
&\le C'\cdot \frac{\Delta}{n}\cdot \binom{n}{b} \cdot n^{-v(Z)+r},
\end{align*}
where the last inequality follows since $C'$ is large enough.
Since $Z\subseteq B_{\rm on}$ or $Z\subseteq B_{\rm off}$ and $B_0$ has rooted girth at least $g$, this implies that 
$$v(Z) \ge (q-r)|Z|+r+{\bf 1}_{|Z|\ge 2}$$ 
where we used that $B_{\rm on}$ and $B_{\rm off}\cup \{H\}$ both have girth at least $g$ which follows since $B_0$ has rooted girth at least $g$. Hence 
$$|\mathcal{B}_{\rm small}(Z)|\le C'\cdot \frac{\Delta}{n}\cdot \binom{n}{b} \cdot n^{-(q-r)\cdot |Z| - {\bf 1}_{|Z|\ge 2}},$$
and (1) holds as desired.

So suppose $|F| \ge r$. Then
$$|\{H\in \mc{H}: F\subseteq V(H)\}| \le C.$$
For each such $H$, there are at most $n^{b - (v(Z)-|F|)}$ elements $B$ of $\mc{B}_H$ with $V(Z) \subseteq V(B)$. Moreover, for those $B$ for which $Z\subseteq B_{\rm on}$ or $Z\subseteq B_{\rm off}$, it follows that since $B_0$ has rooted girth at least $g$, that 
$$v(Z)-|F| \ge (q-r)\cdot |Z|.$$
where we used that $B_{\rm on}$ has rooted girth at least $g$ at $V(H)$ and $B_{\rm off}\cup \{H\}$ has girth at least $g$ which follow since $B_0$ has rooted girth at least $g$. Hence
\begin{align*}
|\mathcal{B}_{\rm large}(Z)| &\le \sum_{i=r}^{q} \binom{v(Z)}{i}\cdot  C \cdot n^{b-(v(Z)-i)} \le v(B_0)^q\cdot q\cdot C \cdot n^{b-(q-r)\cdot |Z|}\\
&\le C'\cdot \binom{n}{b}\cdot n^{-(q-r)\cdot |Z|},
\end{align*}
and (2) holds as desired where the last inequality follows since $C'$ is large enough.
\end{proof}

\subsection{An Auxiliary Hypergraph for Intrinsic Girth}

To prove Lemma~\ref{lem:RandomQuantumIntrinsic}, it is useful to construct an auxiliary hypergraph that encodes the low girth configurations among boosters. First we define a configuration among boosters as follows.

\begin{definition}\label{def:GirthBoosterConfig}
Let $q>r\ge 2$ and $g\ge 3$ be integers. Let $(G,A,X,\mc{H})$ be a $K_q^r$-sponge. Let $\mc{B} \in {\rm Full}_{B_0}(G,A,X)$ for some rooted $K_q^r$-booster $B_0$.

An \emph{$g$-booster-configuration} of $\mc{B}$ is a subset $U =\{U_1,\ldots,U_m\}\subseteq \mc{B}$ with $m\ge 2$ such that $|U\cap \mc{B}_H|\le 1$ for all $H\in \mc{H}$ and there exists $F\in {\rm Girth}_{K_q^r}^g(G)$ and a partition of $F$ into $m$ nonempty sets, $F_1,\ldots, F_m$, such that for each $i\in [m]$, we have $F_i\subseteq (U_i)_{\rm on}$ or $F_i\subseteq (U_i)_{\rm off}$.
\end{definition}

We note that we could further require the boosters in a configuration above to be edge-disjoint (since any realization of a quantum booster into edge-disjoint boosters has this property); however such a requirement can only decrease the degrees and codegrees of the associated hypergraph below. Since we are able to upper bound as is, we opted for the above definition for simplicity. 

Here is the required auxiliary hypergraph.

\begin{definition}\label{def:GirthConfig}
Let $q>r\ge 2$ and $g\ge 3$ be integers. Let $(G,A,X,\mc{H})$ be a $K_q^r$-sponge. Let $\mc{B} \in {\rm Full}_{B_0}(G,A,X)$ for some rooted $K_q^r$-booster $B_0$.
Define the \emph{girth-$g$ configuration hypergraph} of $\mc{B}$, denoted ${\rm Girth}_g(\mc{B})$, as
$$V({\rm Girth}_g(\mc{B})) := \bigcup_{H\in \mc{H}} \mc{B}_H,$$
$$E({\rm Girth}_g(\mc{B})):= \{ \mc{Z}\subseteq V({\rm Girth}_g(\mc{B})): \mc{Z} \text{ is a $g$-booster-configuration of } \mc{B}\}.$$
\end{definition}

We need the following useful definition and proposition.

\begin{definition}
Let $q > r \ge 2$ and $g\ge 3$ be integers. Let $Q\in V({\rm Girth}_{K_q^r}^g(K_n^r))$ and $J\subseteq K_n^r$. We say $F\in {\rm Girth}_{K_q^r}^g(K_n^r)$ with $Q\subseteq F$ is \emph{$(Q,J)$-large} if $\left(\binom{V(F)}{r}\setminus \binom{Q}{r}\right)\cap J\ne \emptyset$. We let ${\rm Large}_J(Q)$ denote the set of $(Q,J)$-large $F$.
\end{definition}
\begin{proposition}\label{prop:LargeIntersect}
Let $q > r \ge 2$ and $g\ge 3$ be integers. If $Q\in V({\rm Girth}_{K_q^r}^g(K_n^r))$ and $J\subseteq K_n^r$, then     
$$\left|{\rm Large}_J(Q)^{(s)} \right| \le (sq)\cdot \binom{sq}{r-1}\cdot \binom{sq}{q}^s\cdot n^{(q-r)(s-1)}\cdot \frac{\Delta(J)}{n}.$$
\end{proposition}
\begin{proof}
For every $F\in {\rm Large}_J(Q)^{(s)}$, it follows from the definition of ${\rm Large}_J$ that there exists an ordering $v_1,\ldots, v_{v(F)-q}$ of $V(F)\setminus V(Q)$ such that there exists $e\in J$ where $v_{v(F)-q} \in V(e) \subseteq V(F)$. Also note that since $F\in {\rm Girth}_{K_q^r}^g(K_n^r)$ and $|F|=s$, we have that $v(F)\le (q-r)s+r$ and hence $v(F)-q\le (q-r)(s-1)$. 

Thus $|F|$ is determined by choosing $m:= v(F)-q$ (at most $sq$ such choices), then $v_1,\ldots v_{m-1}$ (at most $n^{m-1}$ such choices), then $e\cap (V(Q)\cup \{v_1,\ldots,v_{m-1}\}$ (at most $\binom{sq}{r-1}$ such choices), then $v_m$ (at most $\Delta(J)$ such choices), and finally $s$ of the $q$-subsets of $V(Q)\cup \{v_1,\ldots,v_m\}$ (at most $\binom{sq}{q}^s$ such choices). 
\end{proof}

We also need the following useful definition and lemma.

\begin{definition}
Let $(G,A,X,\mc{H})$ be a $K_q^r$-sponge. Let $\mc{B} \in {\rm Full}_{B_0}(G,A,X)$ for some rooted $K_q^r$-booster $B_0$. If $Z$ is a ${\rm Girth}_{K_q^r}^g(G)$-avoiding matching of ${\rm Design}_{K_q^r}(G)$ and $s, s', m$ are integers such that $s \ge |Z|\ge 1$ and $1\le m\le s'-s$, then we define
$$|\mc{B}^{*}(Z,s',s,m)|:= \sum_{F\in \mc{G}(Z)^{(s')}}~\sum_{F': Z\subseteq F' \subseteq F, |F'|=s}~\sum_{F_1,\ldots, F_{m}\in \mc{P}_{m}(F\setminus F')} |\mc{B}(F_1,\ldots, F_{m})|,$$
where $\mc{G}:= {\rm Girth}_{K_q^r}^g(G)$ and $\mc{P}_m(F\setminus F')$ denotes the set of partitions of $F\setminus F'$ into $m$ nonempty sets.
\end{definition}

The above is quite useful in the next section as it counts the number of ways to extend a matching $Z$ to an Erd\H{o}s configuration of size $s'$ and then partition and replace $s'-s$ of the new cliques into boosters containing them. Such an object could create a new configuration of size $s$ containing $Z$ if all of the boosters in it are picked by the quantum booster. That said, we use it in this section with $s=|Z|$ for counting the degrees and codegrees of the booster configuration hypergraph (since we may view $Z$ as some matching inside a fixed set of boosters $U$).

We now upper bound this number with the following lemma; note that in most cases we obtain an `extra' $\frac{1}{n}$ except for some special cases where we only obtain an `extra' $\frac{\Delta}{n}$ factor.

\begin{lem}\label{lem:GirthConfigPrelim}
For all integers $q>r\ge 1$, $C\ge 1$ and $g\ge 3$, and $K_q^r$ booster $B_0$ of rooted girth at least $g$, there exists an integer $C'\ge 1$ such that the following holds: Let $(G,A,X,\mc{H})$ be a $C$-refined $K_q^r$-sponge and let $\Delta := \max\left\{ \Delta(X),~v(G)^{1-\frac{1}{r}}\cdot \log v(G)\right\}$ and let $b:=v(B_0)-q$. If $\mc{B} \in {\rm Full}_{B_0}(G,A,X)$ and $Z$ is a ${\rm Girth}_{K_q^r}^g(G)$-avoiding matching of ${\rm Design}_{K_q^r}(G)$ and $s, s', m$ are integers such that $s \ge |Z|\ge 1$ and $1\le m\le s'-s$, then
both of the following hold:
\begin{enumerate}
    \item[(1)] $|\mc{B}^{*}(Z,s',s,m)|\le C'\cdot \binom{n}{q-r}^{s-|Z|}\cdot \binom{n}{b}^{m}\cdot \frac{1}{n}$ if $|Z|\ge 2$, and 
    \item[(2)] $|\mc{B}^{*}(Z,s',s,m)|\le C'\cdot \binom{n}{q-r}^{s-|Z|}\cdot \binom{n}{b}^{m}\cdot \frac{\Delta}{n}$ otherwise. 
\end{enumerate}    
\end{lem}
\begin{proof}
We choose $C'$ large enough as needed throughout the proof. For brevity, let $\mc{D}:={\rm Design}_{K_q^r}(G)$ and $\mc{G}:= {\rm Girth}_{K_q^r}^g(G)$. Note that $|\mc{P}_{m}(F\setminus F')|\le |F|^{m}$ and for $(F_1,\ldots, F_m)\in \mc{P}_{m}(F\setminus F')$, we have $\sum_{i\in [m]} |F_i| = |F|-|F'|$. 

Let $t':= |Z|$. Now we consider two cases depending on whether $t'\ge 2$. First suppose $t'\ge 2$. Then since $Z$ is a $\mc{G}$-avoiding matching of $\mc{D}$ and $|Z|\ge 2$, we find that
$$|\mc{G}(Z)^{(s')}| \le \Delta_{s',t'}(\mc{G}) \le (qs')^{qs'}\cdot n^{(q-r)(s'-t') - 1}.$$ 
Let $C_0$ be as in Lemma~\ref{lem:BoostersContainingSetOfCliques} for $q,r,C,g$ and $B_0$. It follows from Lemma~\ref{lem:BoostersContainingSetOfCliques}(1)-(2) that for all $i\in [m]$, we have
$$|\mc{B}(F_i)|\le 2C_0\cdot n^{b-(q-r)|F_i|},$$
and hence using that $\sum_{i\in [m]} |F_i| = s'-s$, we find that
$$|\mc{B}(F_1,\ldots, F_{m})| \le (2C_0)^{m}\cdot n^{b\cdot m-(q-r)(s'-s)}.$$ 
Substituting the above bounds, we calculate that
\begin{align*}
|\mc{B}^{*}(Z,s',s,m)| &= \sum_{F\in \mc{G}(Z)^{(s')}}~\sum_{F': Z\subseteq F' \subseteq F, |F'|=s}~\sum_{(F_1,\ldots, F_{m})\in \mc{P}_{m}(F\setminus F')} |\mc{B}(F_1,\ldots, F_{m})| \\
&\le (qs')^{qs'}\cdot  n^{(q-r)(s'-t') - 1} \cdot (s')^{m}\cdot  (2C_0)^{m}\cdot n^{b\cdot m-(q-r)(s'-s)} \\  
&\le C'\cdot \binom{n}{q-r}^{s-t'}\cdot \binom{n}{b}^{m}\cdot \frac{1}{n},
\end{align*}
where for the last inequality we used that $C'$ is large enough.  

So we assume $t'=1$. For a partition $F_1,\ldots, F_{m}$ of $F\setminus F'$, recall that 
$$|\mc{B}(F_1,\ldots, F_{m})| = |\mc{B}_{\rm small}(F_1,\ldots, F_{m})|+|\mc{B}_{\rm large}(F_1,\ldots, F_{m})|.$$
We now bound the sizes of these two sets separately.  By Lemma~\ref{lem:BoostersContainingSetOfCliques}(1), we have that for all $i\in [m]$
$$|\mc{B}_{\rm small}(F_i)| \le C'\cdot \Delta\cdot n^{b-(q-r)|F_i|-1-{\bf 1}_{|F_i|\ge 2}} \le  C'\cdot \Delta\cdot n^{b-(q-r)|F_i|-1} .$$
Hence 
$$|\mc{B}_{\rm small}(F_1,\ldots, F_{m})| \le \prod_{i\in [m]} C_0\cdot \Delta\cdot n^{b-(q-r)|F_i|-1} = \left(C_0\right)^{m}\cdot \left(\frac{\Delta}{n}\right)^{m} \cdot n^{b\cdot m-(q-r)(s'-s)} .$$
On the other hand,
$$|\mc{B}_{\rm large}(F_1,\ldots, F_{m})| \le |\mc{B}(F_1,\ldots,F_m)| \le (2C_0)^{m} \cdot n^{b\cdot m-(q-r)(s'-s)}.$$
Recall that in this case since $t'=1$, we have that $|Z|=1$. Let $Q$ be such that $Z=\{Q\}$. Then since $Z$ is a $\mc{G}$-avoiding matching of $\mc{D}$, we find that
$$|\mc{G}(Z)^{(s')}| \le \Delta_{s',1}(\mc{G}) \le (qs')^{qs'}\cdot n^{(q-r)(s'-1)}.$$ 
Yet by Proposition~\ref{prop:LargeIntersect}, we have that
\begin{align*}
|{\rm Large}_{X\cup A}(Q)^{(s')}| &\le (s'q)\cdot \binom{s'q}{r-1}\cdot \binom{s'q}{q}^{s'}\cdot n^{(q-r)(s'-1)}\cdot \frac{\Delta(X\cup A)}{n} \\
&\le (s'q)\cdot \binom{s'q}{r-1}\cdot \binom{s'q}{q}^{s'}\cdot n^{(q-r)(s'-1)}\cdot \frac{(C+1)\Delta}{n}.
\end{align*}
Note that for $F\in \mc{G}(Z)^{(s'})$, we have that $|{\rm B}_{\rm large}(F_1,\ldots, F_{m})|\ne 0$ only if $F\in {\rm Large}(X\cup A)(Q)^{(s')}$. Combining these calculations yields that
\begin{align*}
|\mc{B}^{*}(Z,s',s,m)| &= \sum_{F\in \mc{G}(Z)^{(s')}}~\sum_{F': Z\subseteq F' \subseteq F, |F'|=s}~\sum_{(F_1,\ldots, F_{m})\in \mc{P}_{m}(F\setminus F')} |\mc{B}(F_1,\ldots, F_{m})| \\
&= \sum_{F\in \mc{G}(Z)^{(s')}}~\sum_{F': Z\subseteq F' \subseteq F, |F'|=s}~\sum_{(F_1,\ldots, F_{m})\in \mc{P}_{m}(F\setminus Z)} |\mc{B}_{\rm small}(F_1,\ldots, F_{m})| \\
&+ \sum_{F\in {\rm Large}_{X\cup A}(Q)^{(s')}}~\sum_{F': Z\subseteq F' \subseteq F, |F'|=s}~\sum_{(F_1,\ldots, F_{m})\in \mc{P}_{m}(F\setminus Z)} |\mc{B}_{\rm large}(F_1,\ldots, F_{m})|\\
&\le (qs')^{qs'}\cdot  n^{(q-r)(s'-1)} \cdot \binom{s'-1}{s-1}\cdot (s')^{m}\cdot \left(C_0\right)^{m}\cdot \left(\frac{\Delta}{n}\right)^{m} \cdot n^{b\cdot m-(q-r)(s'-s)} \\
&~~~+ (s'q)\cdot \binom{s'q}{r-1}\cdot \binom{s'q}{q}^{s'}\cdot n^{(q-r)(s'-1)}\cdot \frac{(C+1)\Delta}{n} \cdot \binom{s'-1}{s-1}\cdot (s')^{m} \cdot (2C_0)^{m}\cdot n^{b\cdot m-(q-r)(s'-s)} \\  
&\le (qg)^{4qg}\cdot (2C_0)^{m}\cdot n^{(q-r)(s-1)} \cdot n^{b\cdot m}\cdot \frac{\Delta}{n}, \\
&\le C'\cdot \binom{n}{q-r}^{s-1}\cdot \binom{n}{b}^n \cdot \frac{\Delta}{n}
\end{align*}
where we used that $\Delta \le n$ and $C'$ is large enough. Hence (2) holds as desired.
\end{proof}

We now use Lemma~\ref{lem:GirthConfigPrelim} to upper bound the degrees and codegrees of the girth-$g$ configuration hypergraph as follows.

\begin{lem}\label{lem:GirthConfigBoosterDegrees}
For all integers $q>r\ge 1$, $C\ge 1$ and $g\ge 3$, and $K_q^r$ booster $B_0$ of rooted girth at least $g$, there exists an integer $C'\ge 1$ such that the following holds: Let $(G,A,X,\mc{H})$ be a $C$-refined $K_q^r$-sponge and let $\Delta := \max\left\{ \Delta(X),~v(G)^{1-\frac{1}{r}}\cdot \log v(G)\right\}$ and let $b:=v(B_0)-q$. If $\mc{B} \in {\rm Full}_{B_0}(G,A,X)$, then
both of the following hold:
\begin{enumerate}
    \item[(1)] $\Delta_{s,t}({\rm Girth}_g(\mc{B}))\le C'\cdot \binom{n}{b}^{s-t}\cdot \frac{1}{n}$ for all $g\ge s > t \ge 2$. 
    \item[(2)] $\Delta_{s,1}({\rm Girth}_g(\mc{B}))\le C'\cdot \binom{n}{b}^{s-1}\cdot \frac{\Delta}{n}$ for all $g\ge s > 1$. 
\end{enumerate}    
\end{lem}
\begin{proof}
We choose $C'$ large enough as needed throughout the proof. For brevity, let $\mc{D}:={\rm Design}_{K_q^r}(G)$ and $\mc{G}:= {\rm Girth}_{K_q^r}^g(G)$. 

Fix $U\subseteq V({\rm Girth}_g(\mc{B}))$ with $|U|=t$. We define
\begin{align*}
\mc{I}(U) &:= \left\{Z\subseteq \bigcup_{B\in U} (B_{\rm off}\cup B_{\rm on}) \text{ with } Z\cap ( B_{\rm on}\cup B_{\rm off}) \ne \emptyset \text{ for all } B\in U: Z\text{ is a } \mc{G}\text{-avoiding matching of } \mc{D}\right\},\\
\mc{I}_{t'}(U) &:= \{Z\in \mc{I}(U): |Z|=t'\}.    
\end{align*}
We note that $|\mc{I}_{t'}(U)|\le \left(\binom{b+q}{q}\cdot t\right)^{t'}$. Thus using this notation, we find that
$$|{\rm Girth}_g(\mc{B})(U)| \le \sum_{t': g-2\ge t' \ge t} ~\sum_{Z\in \mc{I}_{t'}(U)}~\sum_{s': g-1\ge s' > t'}~|\mc{B}^{*}(Z,s',t',s-t)|.$$

Let $C_1$ be as in Lemma~\ref{lem:GirthConfigPrelim}. We now partition this sum according to whether $t'\ge 2$. First suppose $t'\ge 2$. Then by Lemma~\ref{lem:GirthConfigPrelim}(1), we find that
$$|\mc{B}^{*}(Z,s',t',s-t)|\le C_1\cdot \binom{n}{b}^{s-t}\cdot \frac{1}{n}.$$
Substituting the above bounds, we calculate that
\begin{align*}
\sum_{t'\ge 2:~g-2\ge t' \ge t} &~\sum_{Z\in \mc{I}_{t'}(U)}~\sum_{s': g-1\ge s' > t'}~|\mc{B}^{*}(Z,s',t',s-t)| \\
&\le \sum_{t': g-2\ge t' \ge t} \left(\binom{b+q}{q}\cdot t\right)^{t'} \sum_{s': g-1\ge s' > t'} C_1\cdot \binom{n}{b}^{s-t}\cdot \frac{1}{n} \\  
&\le \left(\binom{b+q}{q}\cdot t\right)^{g} \cdot g^2 \cdot C_1\cdot \binom{n}{b}^{s-t}\cdot \frac{1}{n}\\
&\le \frac{C'}{2}\cdot \binom{n}{b}^{s-t}\cdot \frac{1}{n},
\end{align*}
where for the last inequality we used that $C'$ is large enough. Note that if $t\ge 2$, then $t' \ge 2$ and hence the above implies in that case that $|{\rm Girth}_g(\mc{B})(U)|\le C' \cdot \binom{n}{b}^{s-t}\cdot \frac{1}{n}$ and hence (1) holds as desired.

So we consider the sum when $t'=1$; note this implies that $t=1$. Then by Lemma~\ref{lem:GirthConfigPrelim}(2), we find that
$$|\mc{B}^{*}(Z,s',t',s-1)|\le C_1\cdot \binom{n}{b}^{s-1}\cdot \frac{\Delta}{n}.$$
Hence, we calculate the $t'=1$ term of our sum as follows:
\begin{align*}
\sum_{t'=1:~g-2\ge t' \ge t} ~&\sum_{Z\in \mc{I}_{t'}(U)}~\sum_{s': g-1\ge s' > t'}~~|\mc{B}^{*}(Z,s',t',s-t)| \\
&\le \left(\binom{b+q}{q}\cdot t\right) \sum_{s': g-1\ge s' > t'}  C_1\cdot \binom{n}{b}^{s-1}\cdot \frac{\Delta}{n} \\
&\le \frac{C'}{2}\cdot \binom{n}{b}^{s-1}\cdot \frac{\Delta}{n},
\end{align*}
where for the last inequality we used that $C'$ is large enough. 
Thus when $t=1$, we find that
$$|{\rm Girth}_g(\mc{B})(U)|\le \frac{C'}{2} \cdot \binom{n}{b}^{s-1}\cdot \frac{1}{n} + \frac{C'}{2}\cdot \binom{n}{b}^{s-1}\cdot \frac{\Delta}{n}\le C'\cdot \binom{n}{b}^{s-1}\cdot \frac{\Delta}{n}$$
where we used that $\Delta \ge 1$ and hence (2) holds as desired. 
\end{proof}

\subsection{Proof of Lemma~\ref{lem:RandomQuantumIntrinsic}}

We are now prepared to prove Lemma~\ref{lem:RandomQuantumIntrinsic} as follows.

\begin{lateproof}{lem:RandomQuantumIntrinsic}
We choose $a$ large enough as needed throughout the proof, and then subject to that, choose $n$ large enough as needed.

Let $\mathcal{B}_{{\rm off}}, \mathcal{B}_{{\rm on}}$ be as in Definition~\ref{def:rootedbooster} for $B_0$. We let $(B_H(S))_{\rm on}, (B_H(S))_{\rm off}$ denote the decompositions of $B_H(S)$ corresponding to $\mathcal{B}_{\rm on}$ and $\mathcal{B}_{\rm off}$ respectively. 

Let $b:= v(B_0)-q$ (that is, the number of non-root vertices). 

Since $(K_n^r,A,X,\mc{H})$ is $ga$-bounded and $C$-refined, we have that $$\Delta(X\cup A)\le (C+1)\cdot \Delta \le (C+1)\cdot \frac{n}{\log^{ga} n},$$
since $n$ is large enough. Hence it follows that since $n$ is large enough, we have for each $H\in \mc{H}$ that
$$\frac{1}{2}\cdot \binom{n}{b} \le |\mathcal{S}(H)| \le \binom{n}{b}.$$ Similarly it follows that $$|\mathcal{H}| \le n^{r-1}\cdot \Delta(A) \cdot C \le C^2\cdot \Delta \cdot n^{r-1}.$$

We will argue that all of the following hold with probability at least $0.99$ and hence by the union bound all of them hold together with probability at least $.9$ as desired:

\begin{enumerate}
    \item[(1)] For each $H\in \mathcal{H}$, $|(\mathcal{B}_{H})_p| \ge \frac{1}{4}\cdot \log^{a} n$. 
    \item[(2)] $\Delta(\bigcup \mc{B}_p)\le \Delta\cdot \log^{a} \Delta$. 
    \item[(3)] For each $H\in \mathcal{H}$, $|(\mathcal{B}_{H})_p\setminus {\rm Disjoint}((\mathcal{B}_{H})_p)| \le \frac{1}{12} \cdot \log^{a} n$.
    \item[(4)] For each $H\in \mathcal{H}$, $|(\mathcal{B}_{H})_p\setminus {\rm HighGirth}_g((\mathcal{B}_{H})_p)| \le \frac{1}{12} \cdot \log^{a} n$.
\end{enumerate}

We proceed to show this via a series of claims. For most of the claims, we show that some bad event (or set of bad events) happen with small enough probability. For the third and fourth claims to show such a small probability, we invoke Corollary~\ref{cor:KimVu}; namely, for each bad event $\mathcal{A}$, we construct an auxiliary multi-hypergraph $J_{\mathcal{A}}$ that is $k_{\mathcal{A}}$-uniform; in each case, $V(J_{\mathcal{A}}) = \mathcal{B}_0$ (the set of all boosters). But for each $\mathcal{A}$, $E(J_{\mathcal{A}})$ will be specific to that event. This will be done so that the event $\mathcal{A}$ is a subset of the event ``$e((J_{\mathcal{A}})_p) \ge 2K_{\mathcal{A}}$'' for some constant $K_{\mathcal{A}}$ such that $e(J_{\mathcal{A}}) \le K_{\mc{A}}$ and that also satisfies $(*)_i$ for all $i\in [k_{\mc{A}}]$. To that end, let $n_0 := |\mathcal{B}|$. Note that 
$$n_0 \le \binom{n}{b} \cdot |\mc{H}| \le \binom{n}{b} \cdot C^2 \cdot \Delta \cdot n^{r-1} \le n^{r+b}$$
since $n$ is large enough. Hence
$$\log n_0 \le (r+b) \log n \le \log^2 n$$
since $n$ is large enough. Similarly $n_0 \ge \frac{1}{2} \cdot \binom{n}{b}\cdot |\mc{H}| \ge n$ for $n$ large enough.

We now proceed with the claims.

\begin{claim}\label{cl:1Holds}
(1) holds with probability at least $0.99$.    
\end{claim}
\begin{proofclaim}
By Linearity of Expectation, we have that $$\Expect{|(\mc{B}_{H})_p|} = p\cdot |\mathcal{B}_{H}| \ge  \frac{p}{2} \cdot \binom{n}{b} = \frac{\log^{a} n}{2}.$$
Since the events $B_H(S) \in (\mc{B}_{H})_p$ are independent for each $B_H(S) \in \mathcal{B}_{H}$, we have by the Chernoff bound that 
$$\Prob{|(\mc{B}_{H})_p| \le \frac{\log^{a} n}{4}} \le e^{-\frac{\log^{a} n}{16} } < \frac{1}{100 \cdot |\mathcal{H}|},$$
since $n$ is large enough. Hence by the union bound, we have that (1) holds with probability at least $0.99$ as desired.
\end{proofclaim}

\begin{claim}\label{cl:2Holds}
(2) holds with probability at least $0.99$.    
\end{claim}
\begin{proofclaim}
For each $R\in \binom{V(K_n^r)}{r-1}$, let $A_{2,R}$ be the event that $d_{B_1}(R) > a\Delta \cdot \log^{a} \Delta$. Let $\mathcal{A}_2 := \bigcup_{R\in \binom{V(K_n^r)}{r-1}} A_{2,R}$. We desire to show that $\Prob{\mc{A}_2} \le 0.01$ and hence by the union bound it suffices to show that for each $R\in \binom{V(K_n^r)}{r-1}$
$$\Prob{A_{2,R}} \le \frac{1}{100\cdot n^{r-1}}.$$

To that end, fix $R\in \binom{V(K_n^r)}{r-1}$. 

Note $\bigcup \mc{B}_p\subseteq K_n\setminus (X\cup A)$ by definition. Also note there exist at most $n$ edges $e$ of $K_n^r\setminus (X\cup A)$ containing $R$. Let $C'$ be as in Lemma~\ref{lem:BoostersContainingSetOfCliques}. By Lemma~\ref{lem:BoostersContainingSetOfCliques}, we have that $|\mc{B}(e)|\le  C' \cdot \Delta\cdot n^{b-1}$. By Linearity of Expectation, we have that 
\begin{align*}
\Expect{d_{B_1}(R)}&\le \sum_{H\in \mathcal{H}} \sum_{B_H(S)\in \mathcal{B}_{H}} \Prob{B_H(S)\in (\mc{B}_{H})_p}~\cdot~|\{e\in B_H(S): R\subseteq e\}| \\
&\le \sum_{e\in K_n^r\setminus (X\cup A): R\subseteq e} p\cdot |\mc{B}(e)| \le p\cdot C' \cdot \Delta \cdot n^b\\
&\le b^b\cdot C'\cdot \Delta \cdot \log^{a} n \le \frac{1}{2}\cdot a\Delta\cdot \log^{a} \Delta,
\end{align*}
where for the last inequalities we used that $p = \frac{\log^{a} n}{\binom{n}{b}}$ and $\binom{n}{b} \ge \left(\frac{n}{b}\right)^b$ and that $a$ is large enough.
Note that
$$|\{ e\in B_H(S): R\subseteq e\}| \le v(B_H(S)) \le b+q.$$

By the Chernoff bounds, we have that 
$$\Prob{d_{B_1}(R) > a\Delta\cdot \log^{a} \Delta} \le e^{-\frac{1}{8}\cdot \frac{a\Delta\cdot \log^{a} \Delta}{2(b+q)}} \le e^{-\frac{1}{16(b+q)}\cdot n^{1-(1/r)}} \le \frac{1}{100\cdot n^{r-1}},$$
as desired where we used that $\Delta \ge n^{1-(1/r)}$ and $n$ is large enough. 
\end{proofclaim}

\begin{claim}\label{cl:3Holds}
(3) holds with probability at least $0.99$.    
\end{claim}
\begin{proofclaim}
For each $H\in \mathcal{H}$, let $A_{3,H}$ be the event that $|(\mathcal{B}_{H})_p\setminus {\rm Disjoint}((\mathcal{B}_{H})_p)| > \frac{\log^{a} n}{12}$.
Let $\mathcal{A}_3 := \bigcup_{H\in \mc{H}} A_{3,H}$. We desire to show that $\Prob{\mc{A}_3} \le 0.01$ and hence by the union bound it suffices to show that for each $H\in \mc{H}$
$$\Prob{A_{3,H}} \le \frac{1}{100 |\mc{H}|}.$$

To that end, fix $H \in \mathcal{H}$. For each $S\in \mc{S}(H)$, let 
$$N_{\mathcal{B},H}(S) := \{ B_{H'}(S') \in \mc{B}\setminus \mc{B}_{H}: B_{H'}(S')\cap B_{H}(S) \ne \emptyset\}.$$
We define $J$ as the $2$-uniform hypergraph with $$V(J) := \mc{B},~~E(J) := \bigg\{ \{B_H(S), B_{H'}(S')\}: S\in \mc{S}(H),~B_{H'}(S')\in N_{\mathcal{B},H}(S) \bigg\}.$$
Let $n_0:= n^{b+r}$. Note that $v(J) =|\mc{B}| \le n^{b+r}$ which follows since $n$ is large enough. Also $\log^2 n\ge \log n_0 $ since $n$ is large enough.

Note that for $B_H(S)\in \mc{B}_{H}$, we have $B_H(S)\not\in {\rm Disjoint}(\mc{B}_p)$ if and only if $N_{\mathcal{B},H}(S)\cap \mathcal{B}_p\ne \emptyset$. Hence
$$|(\mathcal{B}_{H})_p\setminus {\rm Disjoint}((\mathcal{B}_{H})_p)| \le \sum_{S\in \mc{S}(H)} |N_{\mathcal{B},H}(S)\cap \mathcal{B}_p| = e(J_p).$$
Thus $A_{3,H}$ is a subset of the event that $e(J_p) > \frac{\log^{a} n}{12}$.

We calculate using Lemma~\ref{lem:BoostersContainingSetOfCliques} that
$$ |N_{\mathcal{B}_0,H}(S)| \le \sum_{e\in B_H(S)} |\mc{B}(e)| \le \binom{b+q}{r} \cdot C' \cdot \Delta \cdot n^{b-1} \le \frac{\binom{b+q}{r} \cdot C'}{a}\cdot \frac{n^b}{\log^{a} n},$$
since there at most $\binom{b+q}{r}$ edges $e$ of $B_H(S)$.

Let $K:= \frac{1}{24} \cdot \frac{\binom{n}{b}^2}{\log^{a} n}$. Thus
$$e(J) \le \sum_{S\in \mc{S}(H)} |N_{\mathcal{B}_0,H}(S)| \le |\mc{S}(H)|\cdot \frac{\binom{b+q}{r} \cdot C'}{a}\cdot \frac{n^b}{\log^{a} n} \le K,$$
since $|\mc{S}(H)|\le \binom{n}{b}$ and $a$ is large enough. We note that $J$ is in fact not a multi-hypergraph and hence $\Delta_2(J) = 1$.
For $B\in \mc{B}_H$,,we find that
$$|J(B)| \le |N_{\mc{B},H}(S)| \le \binom{b+q}{r}\cdot (C'\cdot \Delta \cdot n^{b-1}).$$
For $B'\in \mc{B}\setminus \mc{B}_H$, we find that
$$|J(B')| \le |S\in \mc{S}(H): B_H(S)\cap B'\ne \emptyset| \le \sum_{e\in B'} |\mc{B}(e)| \le \binom{b+q}{r}\cdot (C'\cdot \Delta \cdot n^{b-1}).$$
since there are at most $\binom{b+q}{r}$ edges $e$ in $B'$ and for each such $e$, we have by Lemma~\ref{lem:BoostersContainingSetOfCliques} that $|\mc{B}(e)| \le C'\cdot \Delta \cdot n^{b-1}$ from before. Combining we find that
$$\Delta_1(J) \le \binom{b+q}{r}\cdot (C'\cdot \Delta \cdot n^{b-1}).$$

Now we check that $(*_i)$ holds for $i\in \{1,2\}$ as follows. First
$$p^2 \cdot K = \frac{\log^{2a} n}{\binom{n}{b}^2} \cdot \frac{1}{24} \cdot \frac{\binom{n}{b}^2}{\log^{a} n} = \frac{\log^{a} n}{24} \ge (1)\cdot \log^{10} n_0 = \Delta_2(J)\cdot \log^{10} n_0,$$
since $n$ is large enough, $a\ge 21$ and $\log^2 n \ge \log n_0$.
Second
$$p \cdot K = \frac{\log^{a} n}{\binom{n}{b}} \cdot \frac{1}{24} \cdot \frac{\binom{n}{b}^2 }{\log^{a} n } = \frac{\binom{n}{b}}{24} \ge (C'\cdot \Delta \cdot n^{b-1})\cdot \log^6 n_0 = \Delta_1(J)\cdot \log^6 n_0,$$
since $n$ is large enough and $\Delta\le \frac{n}{a\cdot \log^{a} n} \le \frac{n}{\log^{13} n}$ as $a\ge 13$.

Now we apply to Corollary~\ref{cor:KimVu} to $J$ to find that with probability at most $(n_0)^{-\log n_0}$ (which is at most $n^{-\log n}$ since $n\le n_0$), we have that
$$e(J_p) > 2\cdot p^2\cdot K = \frac{\log^{a} n}{12}.$$ 
Since $n$ is large enough then, we have that
$$\Prob{e(J_p) > \frac{\log^a n}{12}} \le n^{-\log n} \le  \frac{1}{100\cdot |\mc{H}|}$$
as desired since $|\mc{H}| \le n^{r+1}$ and $n$ is large enough.
\end{proofclaim}

\begin{claim}\label{cl:4Holds}
(4) holds with probability at least $0.99$.    
\end{claim}
\begin{proofclaim}
For each $H\in \mathcal{H}$, let $A_{4,H}$ be the event that $|(\mc{B}_{H})_p\setminus {\rm HighGirth}_g( (\mc{B}_{H})_p)| > \frac{\log^{a} n}{12}$.
Let $\mathcal{A}_4 := \bigcup_{H\in \mc{H}} A_{4,H}$. We desire to show that $\Prob{\mc{A}_4} \le 0.01$ and hence by the union bound it suffices to show that for each $H\in \mc{H}$
$$\Prob{A_{4,H}} \le \frac{1}{100 |\mc{H}|}.$$

To that end, fix $H \in \mathcal{H}$. We define a hypergraph 
$$J:=\bigcup_{B_H\in \mc{B}_H} {\rm Girth}_g(\mc{B})(B_H).$$
Note that $J\subseteq {\rm Girth}_g(\mc{B})$ since every edge of ${\rm Girth}_g(\mc{B})$ contains at most one element of $\mc{B}_H$ (and hence $J$ is not a multi-hypergraph). Let $n_0:= n^{g(b+r)}$. Note that $v(J) \le e({\rm Girth}_g(\mc{B})) \le n^{g(b+r)}$ which follows since $n$ is large enough. Also $\log^2 n\ge \log n_0$ since $n$ is large enough.

By the definition of $J$, we have that 
$$|(\mc{B}_{H})_p\setminus {\rm HighGirth}_g( (\mc{B}_{H})_p)| \le e(J_p) \le \sum_{k\in \{2,\ldots, g-1\}} e\left( J^{(k)}_p\right).$$
Thus for $k\in \{2,\ldots, g-1\}$, we let $A_{4,H,k}$ be the event that $e\left(J^{(k)}_p\right) > \frac{\log^{a} n}{12g}$. Hence $A_{4,H}$ is a subset of $\bigcup_{k \in \{2,\ldots, g-1\}} A_{4,H,k}$. Thus it suffices to show for each  $k\in \{2,\ldots, g\}$ that
$$\Prob{A_{4,H,k}} \le \frac{1}{100\cdot g \cdot |\mc{H}|}.$$

Let $C''$ be as in Lemma~\ref{lem:GirthConfigBoosterDegrees}. First we calculate using Lemma~\ref{lem:GirthConfigBoosterDegrees}(2) that
\begin{align*}
e\left(J^{(k)}\right) &\le \sum_{B_H\in \mc{B}_H} \left|{\rm Girth}_g(\mc{B})^{(k)}(B_H)\right|\le |\mc{B}_H|\cdot \Delta_{1}\left({\rm Girth}_g(\mc{B})^{(k)}\right) \le \binom{n}{b} \cdot C'' \cdot \binom{n}{b}^{k-1} \cdot \frac{\Delta}{n} \\
&\le C'' \cdot \binom{n}{b}^{k} \cdot \frac{\Delta}{n}
\end{align*}
Since $J^{(k)}$ is a subgraph of ${\rm Girth}_g(\mc{B})^{(k)}$, we calculate using Lemma~\ref{lem:GirthConfigBoosterDegrees}(1)-(2) that for all $i\in [k-1]$
$$\Delta_i\left(J^{(k)}\right) \le \Delta_i\left({\rm Girth}_g(\mc{B})^{(k)}\right)\le C'' \cdot \binom{n}{b}^{k-i} \cdot \frac{\Delta}{n}.$$
Meanwhile we have $\Delta_k(J^{(k)})=1$ since $J^{(k)}$ is in fact not a multi-hypergraph. 

Let $\kappa_k:= \frac{1}{24g} \cdot \frac{\binom{n}{b}^k}{\log^{a\cdot (k-1)} n}$. Note that $\kappa_k\ge e\left(J^{(k)}\right)$ since $\frac{n}{\Delta} \ge \log^{ga} n$. Now we check that $(*_i)$ holds for $i\in [k]$ as follows. Namely for all $i\in [k-1]$, we have as $p=\frac{\log^a n}{\binom{n}{b}}$ that
$$\frac{\Delta_i\left(J^{(k)}\right)}{\kappa_k}\cdot \log^{4k+2}n_0\le 24g\cdot C''\cdot \frac{1}{\binom{n}{b}^i}\cdot \frac{\Delta}{n} \cdot \log^{a(k-1)+(8k+4)}n \le \frac{\log^{ia} n}{\binom{n}{b}^i} = p^i$$ since $\frac{n}{\Delta} \ge \log^{ga}{n} \ge 24g\cdot C''\cdot \log^{a(k-1)+8k+4} n$ as  $k\le g-1$ and $a$ is large enough (at least $8k+5$) and $n$ is large enough. Now we calculate for $(*_k)$ as follows:
$$\frac{\Delta_k\left(J^{(k)}\right)}{\kappa_k}\cdot \log^{4k+2}n_0\le 24g\cdot \frac{1}{\binom{n}{b}^k}\cdot \log^{a(k-1)+(8k+4)}n \le \frac{\log^{ak} n}{\binom{n}{b}^k} = p^k$$ since $a\ge 8k+5$ and $\log n \ge 24g$ as $a$ and $n$ are large enough.

Thus by Corollary~\ref{cor:KimVu}, we find with probability at most
$(n_0)^{-\log n_0}$ (which is at most $n^{-\log n}$ since $n\le n_0$), we have that
$$e\left(J^{(k)}_p\right) > 2\cdot p^k\cdot \kappa_k = \frac{\log^{a} n}{12g}.$$ 
Since $n$ is large enough then, we have that
$$\Prob{e\left(J^{(k)}_p\right) > \frac{\log^{a} n}{12}} \le n^{-\log n} \le  \frac{1}{100\cdot g\cdot |\mc{H}|}$$
as desired since $|\mc{H}| \le n^{r+1}$ and $n$ is large enough.
\end{proofclaim}

By Claims~\ref{cl:1Holds}-~\ref{cl:4Holds}, we have by the union bound that with probability at least $0.9$ all of (1)-(4) hold. If (1), (3) and (4) hold, then we have that for each $H\in \mathcal{H}$, ${\rm Disjoint}(\mc{B}_H) \cap {\rm HighGirth}_g(\mc{B}_H) \ne \emptyset$. Hence with probability at least $.9$ we have the above and also that $\Delta(\bigcup \mc{B}_p)\le \Delta\cdot \log^{a} \Delta$ by (2).
\end{lateproof}

\section{Extrinsic Properties of a Random Sparsification of a Full Quantum Omni-Booster}\label{s:Extrinsic}

\subsection{An Auxiliary Hypergraph for Extrinsic Girth}

\begin{definition}\label{def:CliqueBoosterConfig}
Let $q>r\ge 2$ and $g\ge 3$ be integers. Let $(G,A,X,\mc{H})$ be a $K_q^r$-sponge. Let $\mc{B} \in {\rm Full}_{B_0}(G,A,X)$ for some rooted $K_q^r$-booster $B_0$.

An \emph{$g$-clique-booster-configuration} of $\mc{B}$ is a subset $U\subseteq {\rm Design}_{K_q^r}(G)\cup \mc{B}$ such that $U'= U\cap {\rm Design}_{K_q^r}(G)\ne \emptyset$, $U\cap \mc{B} =\{U_1,\ldots,U_m\}\subseteq \mc{B}$ with $m\ge 1$, $|U\cap \mc{B}_H|\le 1$ for all $H\in \mc{H}$, and there exists $F\in {\rm Girth}_{K_q^r}^g(G)$ with $U'\subseteq F$ and a partition of $F\setminus U'$ into $m$ nonempty sets, $F_1,\ldots, F_m$, such that for each $i\in [m]$, we have $F_i\subseteq (U_i)_{\rm on}$ or $F_i\subseteq (U_i)_{\rm off}$.
\end{definition}

\begin{definition}\label{def:ExtendConfig}
Let $q>r\ge 2$ and $g\ge 3$ be integers. Let $(G,A,X,\mc{H})$ be a $K_q^r$-sponge. Let $\mc{B} \in {\rm Full}_{B_0}(G,A,X)$ for some rooted $K_q^r$-booster $B_0$.
Define the \emph{girth-$g$ extended configuration hypergraph}, ${\rm Extend}_g(\mc{B})$, as
$$V({\rm Extend}_g(\mc{B})) := V({\rm Girth}^g_{K_q^r}(G))\cup V({\rm Girth}_g(\mc{B})),$$
\begin{align*}
E({\rm Extend}_g(\mc{B})):= &\{ \mc{Z} \subseteq V({\rm Extend}^g(\mc{B})): \mc{Z} \text{ is a $g$-clique-booster-configuration }\}
\end{align*}
For integers $s_1, s_2\ge 1$, we define 
$${\rm Extend}_g(\mc{B})^{(s_1,s_2)}:= \{Z_1\cup Z_2\in {\rm Extend}_g(\mc{B}): Z_1\subseteq V({\rm Girth}^g_{K_q^r}(G)), Z_2\subseteq V({\rm Girth}_g(\mc{B})), |Z_1|=s_1, |Z_2|=s_2\},$$ 
and for integers $t_1,t_2$ with $s_1\ge t_1\ge 0$ and $s_2\ge t_2\ge 0$, we define
$$\Delta_{(t_1,t_2)}\left({\rm Extend}_g(\mc{B})^{(s_1,s_2)}\right) := \max_{U_1\subseteq V({\rm Girth}^g_{K_q^r}(G)), |U_1|=t_1}~~\max_{U_2\subseteq V({\rm Girth}_g(\mc{B})), |U_2|=t_2}~~|{\rm Extend}_g(\mc{B})^{(s_1,s_2)}(U_1\cup U_2)|.$$
\end{definition}

Indeed, there is bi-regularity apparent in the upper bounds for the degrees and codegrees of ${\rm Extend}_g(\mc{B})$. Hence the definitions above. The following lemma characterizes the degrees and codegrees as follows.

\begin{lem}\label{lem:ExtendConfigBoosterDegrees}
For all integers $q>r\ge 1$, $C\ge 1$ and $g\ge 3$, and $K_q^r$ booster $B_0$ of rooted girth at least $g$, there exists an integer $C'\ge 1$ such that the following holds: Let $(G,A,X,\mc{H})$ be a $C$-refined $K_q^r$-sponge and let $\Delta := \max\left\{ \Delta(X),~v(G)^{1-\frac{1}{r}}\cdot \log v(G)\right\}$ and let $b:= v(B_0)-q$. If $\mc{B} \in {\rm Full}_{B_0}(G,A,X)$, then for integers $s_1 \ge t_1 \ge 0$, $s_2\ge t_2\ge 0$ with $s_1+s_2 > t_1+t_2 \ge 1$, the following holds:
\begin{enumerate}
    \item[(1)] if $t_1+t_2 \ge 2$, or if $s_1=1$ and $t_1=0$, then $$\Delta_{(t_1,t_2)}\left({\rm Extend}_g(\mc{B})^{(s_1,s_2)}\right)\le C'\cdot \binom{n}{q-r}^{s_1-t_1}\cdot \binom{n}{b}^{s_2-t_2}\cdot \frac{1}{n},$$
    \item[(2)] and otherwise, we have 
    $$\Delta_{(t_1,t_2)}\left({\rm Extend}_g(\mc{B})^{(s_1,s_2)}\right)\le C'\cdot \binom{n}{q-r}^{s_1-t_1}\cdot \binom{n}{b}^{s_2-t_2}\cdot \frac{\Delta}{n}.$$    
\end{enumerate}    
\end{lem}
\begin{proof}
We choose $C'$ large enough as needed throughout the proof. For brevity, let $\mc{D}:={\rm Design}_{K_q^r}(G)$ and $\mc{G}:= {\rm Girth}_{K_q^r}^g(G)$. 

Fix $U_1\subseteq V(\mc{D})$ with $|U_1|=t_1$ and $U_2\subseteq V({\rm Girth}_g(\mc{B}))$ with $|U_2|=t_2$. Let $U:= U_1\cup U_2$. We define
\begin{align*}
\mc{I}(U_2) &:= \left\{Z\subseteq \bigcup_{B\in U_2} (B_{\rm off}\cup B_{\rm on}) \text{ with } Z\cap ( B_{\rm on}\cup B_{\rm off}) \ne \emptyset \text{ for all } B\in U_2: Z\text{ is a } \mc{G}\text{-avoiding matching of } \mc{D}\right\},\\
\mc{I}_{t_2'}(U) &:= \{Z\in \mc{I}(U): |Z|=t'\}.    
\end{align*}
We note that $|\mc{I}_{t_2'}(U)|\le \left(\binom{b+q}{q}\cdot t_2\right)^{t_2'}$. Thus using this notation, we find that
$$|{\rm Girth}_g(\mc{B})(U)| \le \sum_{t_2': g-2\ge t_2' \ge t_2} ~\sum_{Z\in \mc{I}_{t_2'}(U_2)}~\sum_{s': g-1\ge s' > t_1+t_2'}~|\mc{B}^{*}(U_1\cup Z,s',s_1+t_2',s_2-t_2)|.$$

Let $C_1$ be as in Lemma~\ref{lem:GirthConfigPrelim}. We now partition this sum according to whether $t_1+t_2'\ge 2$. First suppose $t_1+t_2'\ge 2$. Then by Lemma~\ref{lem:GirthConfigPrelim}(1), we find that
\begin{align*}
|\mc{B}^{*}(U_1\cup Z,s',s_1+t_2',s_2-t_2)| &\le C_1\cdot \binom{n}{q-r}^{(s_1+t_2')-(t_1+t_2')} \cdot \binom{n}{b}^{s_2-t_2}\cdot \frac{1}{n}\\
&= C_1\cdot \binom{n}{q-r}^{s_1-t_1} \cdot \binom{n}{b}^{s_2-t_2}\cdot \frac{1}{n}.
\end{align*}
Substituting the above bounds, we calculate that
\begin{align*}
\sum_{t_2'\ge 2-t_1: g-2\ge t_2' \ge t_2}&~\sum_{Z\in \mc{I}_{t_2'}(U_2)}~\sum_{s': g-1\ge s' > t_1+t_2'}~|\mc{B}^{*}(U_1\cup Z,s',s_1+t_2',s_2-t_2)| \\
&\le \left(\binom{b+q}{q}\cdot t_2\right)^{g} \cdot g^2 \cdot C_1\cdot \binom{n}{q-r}^{s_1-t_1} \cdot \binom{n}{b}^{s_2-t_2}\cdot \frac{1}{n}\\
&\le \frac{C'}{2}\cdot \binom{n}{q-r}^{s_1-t_1} \cdot \binom{n}{b}^{s_2-t_2}\cdot \frac{1}{n},
\end{align*}
where for the last inequality we used that $C'$ is large enough. Note that if $t_1+t_2\ge 2$, then $t_1+t_2' \ge 2$ and hence the above implies in that case that $|{\rm Girth}_g(\mc{B})(U)|\le C' \cdot \binom{n}{q-r}^{s_1-t_1} \cdot \binom{n}{b}^{s_2-t_2}\cdot \frac{1}{n}$. 

So we assume $t_1+t_2=1$. That said, if $s_1=1$ and $t_1=0$, then since $\mc{G}$ has no configurations of size two by definition, we may assume that $t_2'\ge 2$ as otherwise the term is zero when $t_2'=1$. Hence (1) holds as desired.

So we proceed to show (2). We use the calculation above when $t_1+t_2'\ge 2$. So we consider the sum when $t_1+t_2'=1$; as noted this implies that $t_1+t_2=1$. Then by Lemma~\ref{lem:GirthConfigPrelim}(2), we find that
$$|\mc{B}^{*}(U_1\cup Z,s',s_1+t_2',s_2-t_2)| \le C_1\cdot \binom{n}{q-r}^{s_1-t_1} \cdot \binom{n}{b}^{s_2-t_2}\cdot \frac{\Delta}{n}.$$
Hence, we calculate the $t'=1$ term of our sum as follows:
\begin{align*}
\sum_{t_2'= 1-t_1: g-2\ge t_2' \ge t_2}&~\sum_{Z\in \mc{I}_{t_2'}(U_2)}~\sum_{s': g-1\ge s' > t_1+t_2'}~|\mc{B}^{*}(U_1\cup Z,s',s_1+t_2',s_2-t_2)| \\
&\le \left(\binom{b+q}{q}\cdot t_2\right) \cdot g^2 \cdot C_1\cdot \binom{n}{q-r}^{s_1-t_1} \cdot \binom{n}{b}^{s_2-t_2}\cdot \frac{\Delta}{n}\\
&\le \frac{C'}{2}\cdot \binom{n}{q-r}^{s_1-t_1} \cdot \binom{n}{b}^{s_2-t_2}\cdot \frac{\Delta}{n},
\end{align*}
where for the last inequality we used that $C'$ is large enough. 
Thus when $t_1+t_2=1$, we find that
\begin{align*}
|{\rm Girth}_g(\mc{B})(U)| &\le \frac{C'}{2} \cdot \binom{n}{q-r}^{s_1-t_1} \cdot \binom{n}{b}^{s_2-t_2}\cdot \frac{1}{n} + \frac{C'}{2}\cdot \binom{n}{q-r}^{s_1-t_1} \cdot \binom{n}{b}^{s_2-t_2}\cdot \frac{\Delta}{n} \\
&\le C'\cdot \binom{n}{q-r}^{s_1-t_1} \cdot \binom{n}{b}^{s_2-t_2}\cdot \frac{\Delta}{n}
\end{align*}
where we used that $\Delta \ge 1$ and hence (2) holds as desired. 
\end{proof}

\subsection{Property for Maximum Common $2$-Degree Proof}

We also need the following lemma which we will use to prove the projection treasury in the outcome of Lemma~\ref{lem:RandomQuantumExtrinsic} has maximum common $2$-degree at most $D^{1-(\beta/2)}$.

\begin{lem}\label{lem:UncommonDegree}
For all integers $q>r\ge 1$, $C\ge 1$ and $g\ge 3$, and $K_q^r$ booster $B_0$ of rooted girth at least $g$, there exists an integer $C'\ge 1$ such that the following holds: Let $(G,A,X,\mc{H})$ be a $C$-refined $K_q^r$-sponge and let $\Delta := \max\left\{ \Delta(X),~v(G)^{1-\frac{1}{r}}\cdot \log v(G)\right\}$ and let $b:= v(B_0)-q$. If $\mc{B} \in {\rm Full}_{B_0}(G,A,X)$, then for a $q$-clique $Q\in V({\rm Girth}_{K_q^r}^g(G))$ and vertex $v$ of $K_n$ with $v\not\in V(Q)$, we have for any integer $k\ge 1$ that
$$\left|\left\{\mc{Z}\in {\rm Extend}_q(\mc{B})^{(2,k)}(Q): v\in V(\mc{Z})\right\}\right| \le C'\cdot \binom{n}{q-r}\cdot \binom{n}{b}^k \cdot \frac{1}{n}.$$
\end{lem}
\begin{nonlateproof}{lem:UncommonDegree}
We choose $C'$ large enough as needed throughout the proof. For brevity, let $\mc{D}:={\rm Design}_{K_q^r}(G)$ and $\mc{G}:= {\rm Girth}_{K_q^r}^g(G)$. 

Let 
\begin{align*}
\mc{F} := \Bigg\{ \bigg(F, F_0, (F_1,\dots, F_k), (B_1,\ldots, B_k)\bigg): 
~&F\in \mc{G}(Q),~F_0\in F\setminus \{Q\}, (F_1,\ldots, F_k)\in \mc{P}_{k}(F\setminus \{Q,F_0\}),\\
~&B_i\in \mc{B}(F_i) ~\forall i\in[k],~v\in V(F_0)\cup \bigcup_{i\in [k]} V(B_i)\Bigg\},    
\end{align*}
where recall that $\mc{P}_k(S)$ denotes the set of partitions of $S$ into $k$ nonempty sets. Note that 
$$\left|\left\{\mc{Z}\in {\rm Extend}_q(\mc{B})^{(2,k)}(Z): v\in V(\mc{Z})\right\}\right| \le |\mc{F}|.$$

We will partition $\mc{F}$ into various sets and bound each separately. First let
$$\mc{F}_1 := \bigg\{ \bigg(F, F_0, (F_1,\dots, F_k), (B_1,\ldots, B_k)\bigg)\in \mc{F}: ~v\in V(F) \bigg\}.$$
Next let 
$$\mc{F}_2 := \bigg\{ \bigg(F, F_0, (F_1,\dots, F_k), (B_{H_1}(S_1),\ldots, B_{H_k}(S_k))\bigg)\in \mc{F}\setminus \mc{F}_1: ~\exists i\in [k] \text{ such that } v\in V\left(B_{H_i}(S_i)\right)\setminus V(H_i) \bigg\}.$$
Next let 
\begin{align*}
\mc{F}_3 := \bigg\{ \bigg(F, F_0, (F_1,\dots, F_k), (B_{H_1}(S_1),\ldots, B_{H_k}(S_k))\bigg)\in \mc{F}\setminus (\mc{F}_1\cup \mc{F}_2): ~\exists &i\in [k] \text{ such that } v\in V(H_i) \text { and }\\
&B_{H_i}(S_i)\in \mc{B}_{\rm small}(F_i) \bigg\}.
\end{align*}
Finally we set
$$\mc{F}_4 := \mc{F}\setminus (\mc{F}_1\cup \mc{F}_2\cup \mc{F}_3)$$
and we note that for all $(F, F_0, (F_1,\dots, F_k), (B_{H_1}(S_1),\ldots, B_{H_k}(S_k)))\in \mc{F}_4$, we have that $\exists i\in [k] \text{ such that } v\in V(H_i) \text { and } B_{H_i}(S_i)\in \mc{B}_{\rm large}(F_i).$
Thus
$$|\mc{F}| = \sum_{i\in [4]} |\mc{F}_i|.$$

\begin{claim}
$$|\mc{F}_1| \le \frac{C'}{4} \cdot \binom{n}{q-r}\cdot \binom{n}{b}^k \cdot \frac{1}{n}.$$    
\end{claim}
\begin{proofclaim}
Let $s\ge k+2$ be an integer. Then we calculate that
$$|\{F\in \mc{G}^{(s)}(Q): v\in V(F)\}| \le (qs)^{qs} \cdot n^{(q-r)(s-1)-1}.$$
Thus for $\bigg(F, F_0, (F_1,\dots, F_k), (B_1,\ldots, B_k)\bigg)\in \mc{F}_1$ where $|F|=s$ there are most $(qs)^{qs} \cdot n^{(q-r)(s-1)-1}$ choices of $F$, then at most $s$ choices of $F_0$, then at most $s^k$ choices of $(F_1,\dots, F_k)$ and finally at most $|\mc{B}(F_1,\ldots, F_k)|$ choices of $(B_1,\ldots, B_k)$. Let $C_0$ be as in Lemma~\ref{lem:BoostersContainingSetOfCliques}. By Lemma~\ref{lem:BoostersContainingSetOfCliques}, we find that
$$|\mc{B}(F_1,\ldots, F_{k})| \le (2C_0)^{k}\cdot n^{b\cdot k-(q-r)\cdot \sum_{i\in [k]} |F_i|}.$$
Note that for $|F|=s$, we have $\sum_{i\in [k]} |F_i|=s-2$. Combining we find that
\begin{align*}
|\mc{F}_1| &\le (qs)^{qs} \cdot n^{(q-r)(s-1)-1}\cdot s \cdot s^k \cdot (2C_0)^{k}\cdot n^{b\cdot k-(q-r)(s-2)} \\
&\le \frac{C'}{4} \cdot \binom{n}{q-r}\cdot \binom{n}{b}^k \cdot \frac{1}{n}
\end{align*}
as desired where we used that $C'$ is large enough.
\end{proofclaim}

\begin{claim}
$$|\mc{F}_2| \le \frac{C'}{4} \cdot \binom{n}{q-r}\cdot \binom{n}{b}^k \cdot \frac{1}{n}.$$    
\end{claim}
\begin{proofclaim}
Let $s\ge k+2$ be an integer. Then we calculate that
$$|\mc{G}^{(s)}(Q)| \le (qs)^{qs} \cdot n^{(q-r)(s-1)-1}.$$
Thus for $\bigg(F, F_0, (F_1,\dots, F_k), (B_{H_1}(S_1),\ldots, B_{H_k}(S_k))\bigg)\in \mc{F}_2$ where $|F|=s$ there are most $(qs)^{qs} \cdot n^{(q-r)(s-1)}$ choices of $F$, then at most $s$ choices of $F_0$, then at most $s^k$ choices of $(F_1,\dots, F_k)$, and then at most $k$ choices of $i\in [k]$ such that $v\in V\left(B_{H_i}(S_i)\right)\setminus V(H_i)$. 

Let $C_0$ be as in Lemma~\ref{lem:BoostersContainingSetOfCliques}. For $j\in [k]\setminus \{i\}$, we have by Lemma~\ref{lem:BoostersContainingSetOfCliques} that
$$|\mc{B}(F_j)| \le 2C_0\cdot n^{b-(q-r)\cdot |F_i|}.$$
Meanwhile for $i$, we find via reasoning similar to that in the proof of Lemma~\ref{lem:BoostersContainingSetOfCliques} that
$$|\{B_{H_i}(S_i)\in \mc{B}(F_i): v\in V\left(B_{H_i}(S_i)\right)\setminus V(H_i)\}| \le 2C_1\cdot n^{b-(q-r)\cdot |F_i|-1},$$
for some fixed constant $C_1$ depending on $q,r$, $C$ and $B_0$. 

Note that for $|F|=s$, we have $\sum_{i\in [k]} |F_i|=s-2$. Combining we find that
\begin{align*}
|\mc{F}_1| &\le (qs)^{qs} \cdot n^{(q-r)(s-1)}\cdot s \cdot s^k \cdot k \cdot (2C_0)^{k-1}\cdot (2C_1)\cdot n^{b\cdot k-(q-r)(s-2)-1} \\
&\le \frac{C'}{4} \cdot \binom{n}{q-r}\cdot \binom{n}{b}^k \cdot \frac{1}{n}
\end{align*}
as desired where we used that $C'$ is large enough.    
\end{proofclaim}

\begin{claim}
$$|\mc{F}_3| \le \frac{C'}{4} \cdot \binom{n}{q-r}\cdot \binom{n}{b}^k \cdot \frac{1}{n}.$$    
\end{claim}
\begin{proofclaim}
Let $s\ge k+2$ be an integer. Then we calculate that
$$|\mc{G}^{(s)}(Q)| \le (qs)^{qs} \cdot n^{(q-r)(s-1)}.$$
Thus for $\bigg(F, F_0, (F_1,\dots, F_k), (B_{H_1}(S_1),\ldots, B_{H_k}(S_k))\bigg)\in \mc{F}_3$ where $|F|=s$ there are most $(qs)^{qs} \cdot n^{(q-r)(s-1)}$ choices of $F$, then at most $s$ choices of $F_0$, then at most $s^k$ choices of $(F_1,\dots, F_k)$, and then at most $k$ choices of $\exists i\in [k] \text{ such that } v\in V(H_i) \text { and } B_{H_i}(S_i)\in \mc{B}_{\rm large}(F_i)$. 

Let $C_0$ be as in Lemma~\ref{lem:BoostersContainingSetOfCliques}. For $j\in [k]$, we have by Lemma~\ref{lem:BoostersContainingSetOfCliques} that
$$|\mc{B}(F_j)| \le 2C_0\cdot n^{b-(q-r)\cdot |F_i|}.$$
Meanwhile for $i$, we find via reasoning similar to that in the proof of Lemma~\ref{lem:BoostersContainingSetOfCliques} that (except that the number of choices for $H$ containing the intersection of $F_i$ with the roots of the booster but also the vertex $v$ is a factor of $1/n$ smaller instead of $\Delta(A)/n$ smaller)
$$|\{B_{H_i}(S_i)\in \mc{B}(F_i): v\in V(H_i) \text { and }B_{H_i}(S_i)\in \mc{B}_{\rm small}(F_i)\}| \le 2C_2\cdot n^{b-(q-r)\cdot |F_i|-1},$$
for some fixed constant $C_2$ depending on $q,r$, $C$ and $B_0$. 

Note that for $|F|=s$, we have $\sum_{i\in [k]} |F_i|=s-2$. Combining we find that
\begin{align*}
|\mc{F}_1| &\le (qs)^{qs} \cdot n^{(q-r)(s-1)}\cdot s \cdot s^k\cdot k \cdot (2C_0)^{k-1}\cdot (2C_2)\cdot n^{b\cdot k-(q-r)(s-2)-1} \\
&\le \frac{C'}{4} \cdot \binom{n}{q-r}\cdot \binom{n}{b}^k \cdot \frac{1}{n}
\end{align*}
as desired where we used that $C'$ is large enough.        
\end{proofclaim}

\begin{claim}
$$|\mc{F}_4| \le \frac{C'}{4} \cdot \binom{n}{q-r}\cdot \binom{n}{b}^k \cdot \frac{1}{n}.$$    
\end{claim}
\begin{proofclaim}
Let $s\ge k+2$ be an integer. For $\bigg(F, F_0, (F_1,\dots, F_k), (B_{H_1}(S_1),\ldots, B_{H_k}(S_k))\bigg)\in \mc{F}_4$, since there exists $i\in [k] \text{ such that } v\in V(H_i) \text { and }B_{H_i}(S_i)\in \mc{B}_{\rm small}(F_i)$, it follows that there exists $e\in X\cup A$ with $e\not\in Q$ and $H\in \mc{H}$ such that $e\subseteq V(F)$ and $e\cup \{v\} \subseteq V(H)$. 

Then we calculate that
\begin{align*}
|\{F\in \mc{G}^{(s)}(Q)&: \exists e\in X\cup A \text{ with } e\not\in Q \text{ and } \exists H\in \mc{H} \text{ such that } e\subseteq V(F)\text{ and }e\cup \{v\} \subseteq V(H)\}| \\
&\le (qs)^{qs} \cdot C\cdot q\cdot n^{(q-r)(s-1)-1},    
\end{align*}
since there exist only $C\cdot q$ choices for the last vertex $w$ of $e$ since $w$ is contained in one of that at most $C$ choices of $H$ containing $(e\setminus \{w\})\cup \{v\}$ as $\mc{H}$ is $C$-refined (and for each such $H$, we have $v(H)\le q$). 

Thus for $\bigg(F, F_0, (F_1,\dots, F_k), (B_1,\ldots, B_k)\bigg)\in \mc{F}_4$ where $|F|=s$ there are most $(qs)^{qs} \cdot n^{(q-r)(s-1)-1}$ choices of $F$, then at most $s$ choices of $F_0$, then at most $s^k$ choices of $(F_1,\dots, F_k)$ and finally at most $|\mc{B}(F_1,\ldots, F_k)|$ choices of $(B_1,\ldots, B_k)$. Let $C_0$ be as in Lemma~\ref{lem:BoostersContainingSetOfCliques}. By Lemma~\ref{lem:BoostersContainingSetOfCliques}, we find that
$$|\mc{B}(F_1,\ldots, F_{k})| \le (2C_0)^{k}\cdot n^{b\cdot k-(q-r)\cdot \sum_{i\in [k]} |F_i|}.$$
Note that for $|F|=s$, we have $\sum_{i\in [k]} |F_i|=s-2$. Combining we find that
\begin{align*}
|\mc{F}_1| &\le (qs)^{qs} \cdot n^{(q-r)(s-1)-1}\cdot s \cdot s^k \cdot (2C_0)^{k}\cdot n^{b\cdot k-(q-r)(s-2)} \\
&\le \frac{C'}{4} \cdot \binom{n}{q-r}\cdot \binom{n}{b}^k \cdot \frac{1}{n}
\end{align*}
as desired where we used that $C'$ is large enough.
\end{proofclaim}

\noindent Hence the lemma follows from the various claims above.
\end{nonlateproof}

\subsection{Proof of Lemma~\ref{lem:RandomQuantumExtrinsic}}

\begin{lateproof}{lem:RandomQuantumExtrinsic}
We choose $a$ large enough as needed throughout the proof, and then subject to that, choose $n$ large enough as needed. Throughout the proof, let $C'$ be as in Lemma~\ref{lem:ExtendConfigBoosterDegrees}.

Let $\mathcal{B}_{{\rm off}}, \mathcal{B}_{{\rm on}}$ be as in Definition~\ref{def:rootedbooster} for $B_0$. We let $(B_H(S))_{\rm on}, (B_H(S))_{\rm off}$ denote the decompositions of $B_H(S)$ corresponding to $\mathcal{B}_{\rm on}$ and $\mathcal{B}_{\rm off}$ respectively. 

Let $b:= v(B_0)-q$ (that is, the number of non-root vertices). Let $\mc{D}:= {\rm Design}_{K_q^r}(K_n^r\setminus (X\cup A))$, $\mc{R}:= {\rm Reserve}_{K_q^r}(K_n^r, K_n^r\setminus (X\cup A), X)$, and $\mc{G}:= {\rm Girth}_{K_q^r}^g(K_n^r)$. 

Since $(K_n^r,A,X,\mc{H})$ is $ga$-bounded and $C$-refined, we have that $$\Delta(X\cup A)\le (C+1)\cdot \Delta \le (C+1)\cdot \frac{n}{\log^{ga} n},$$
since $n$ is large enough. Hence it follows that since $n$ is large enough, we have for each $H\in \mc{H}$ that
$$\frac{1}{2}\cdot \binom{n}{b} \le |\mathcal{S}(H)| \le \binom{n}{b}.$$ 
Similarly it follows that $$|\mathcal{H}| \le n^{r-1}\cdot \Delta(A) \cdot C \le C^2\cdot \Delta \cdot n^{r-1}.$$

Let $D:= \binom{n}{q-r}$ and $\beta:= \frac{1}{2g(q-r)}$. Since we desire to show that with probability at least $0.9$, $\mc{B}_p$ satisfies that $\mc{T}:= {\rm Proj}_g(B,A,G,X) = (G_1,G_2,H)$ is $\Bigg(\binom{n}{q-r},~2\alpha \cdot \binom{n}{q-r},~\frac{1}{4g(q-r)},~2\alpha\Bigg)-{\rm regular}$, it suffices by the union bound to show that each of the following hold with probability at least $0.99$:
\begin{enumerate}
    \item[(1)] Every vertex $v\in V(G_1)$ has degree at least $(1-2\alpha)\cdot D$ in $G_1$.
    \item[(2)] Every vertex of $v\in V(G_2)\cap V(G_1)$ has degree at least $D^{1-2\alpha}$ in $G_2$.
    \item[(3)] $\Delta\left(H^{(i)}\right)\le 2\alpha\cdot D^{i-1}\cdot \log D$ for all $2\le i \le g$.
    \item[(4)] $G_1\cup G_2$ has codegrees at most $D^{1-(\beta/2)}$.
    \item[(5)] $\Delta_{t}\left(H^{(s)}\right)\le D^{s-t-(\beta/2)}$ for all $2\le t < s \le g$.
    \item[(6)] The maximum $2$-codegree of $G_1\cup G_2$ with $H$ is at most $D^{1-(\beta/2)}$.
    \item[(7)] The maximum common $2$-degree of $H$ with respect to $G_1\cup G_2$ is at most $D^{1-(\beta/2)}$.
\end{enumerate}

\begin{claim}\label{cl:5Holds}
(1) holds with probability at least $0.99$.    
\end{claim}
\begin{proofclaim}
For each $v\in V(G_1)$, let $A_{1,v}$ be the event that $d_{G_1}(v) < (1-2\alpha)\cdot D$.
Let $\mathcal{A}_1 := \bigcup_{v\in V(G_1)} A_{1,v}$. We desire to show that $\Prob{\mc{A}_1} \le 0.01$ and hence by the union bound it suffices to show that for each $v\in V(G_1)$
$$\Prob{A_{1,v}} \le \frac{1}{100\cdot n^r} \le \frac{1}{100 \cdot v(G_1)},$$
since $v(G_1)\le n^r$.

To that end, fix $v\in V(G_1)$. Since $(G,A,X,\mc{H})$ is $(\alpha,g)$-regular, we have that  $d_{\mc{D}}(v) \ge (1-\alpha)\cdot \binom{n}{q-r}$. Hence $A_{1,v}$ is a subset of the event that $d_{\mc{D}}(v) - d_{G_1}(v) \ge 2\alpha\cdot \binom{n}{q-r}$. 

We define a multi-hypergraph $J$ 
$$J:= \bigcup_{Q\in \mc{D}: v\in Q}~\bigcup_{k\in [g-2]} \left\{ Z\setminus \{Q\}: Z\in{\rm Extend}_g(\mc{B})^{(1,k)}(Q)\right\}.$$
Let $n_0:= n^{g(b+r+q)}$. Note that $v(J) \le e({\rm Extend}_g(\mc{B})) \le n^{g(b+r+q)}$ which follows since $n$ is large enough. Also $\log^2 n\ge \log n_0$ since $n$ is large enough.

Note that for $Q\in \mc{D}$, we have $Q\not\in G_1$ if and only if there exists $Z\in{\rm Extend}_g(\mc{B})^{(1,k)}(Q)$ for some $k\in [g-2]$ such that $Z\setminus \{Q\}\subseteq \mc{B}_p$. Thus $A_{1,v}$ is a subset of the event $e(J_p) \ge 2\alpha D$.

Note that $e(J_p) = \sum_{k\in [g-2]} e\left(J^{(k)}_p\right)$. For $k\in [g-2]$, let $A_{1,v,k}$ be the event that $e\left(J^{(k)}_p\right) > \frac{\alpha\cdot D}{g}$. Hence $A_{1,v}$ is a subset of $\bigcup_{k \in [g-2]} A_{1,v,k}$. Thus it suffices to show for each  $k\in [g-2]$ that
$$\Prob{A_{1,v,k}} \le \frac{1}{100\cdot g \cdot n^r}.$$

First we calculate using Lemma~\ref{lem:ExtendConfigBoosterDegrees}(2) that
\begin{align*}
e\left(J^{(k)}\right) &\le \sum_{Q\in \mc{D}: v\in Q} \left|{\rm Extend}_g(\mc{B})^{(1,k)}(Q)\right|\le d_{\mc{D}}(v) \cdot \Delta_{(1,0)}\left({\rm Extend}_g(\mc{B})^{(1,k)}\right) \\
&\le D\cdot C'\cdot \binom{n}{b}^k\cdot \frac{\Delta}{n}.
\end{align*}
Now fix $i\in [k]$ and $U\subseteq V(J^{(k)})$ with $|U|=i$. Then
$$|J^{(k)}(U)| \le |{\rm Extend}_g(\mc{B})^{(1,k)}(U)| \le 
\Delta_{(0,i)}\left({\rm Extend}_g(\mc{B})^{(1,k)}\right).$$
Hence by Lemma~\ref{lem:ExtendConfigBoosterDegrees}(1)-(2), we have for all $i\in [k]$ that 
$$\Delta_i\left(J^{(k)}\right) \le D\cdot C'\cdot \binom{n}{b}^{k-i}\cdot \frac{\Delta}{n}.$$

Let $\kappa_k:= \frac{\alpha\cdot D}{2g} \cdot \frac{\binom{n}{b}^k}{\log^{ka} n}$. Note that $\kappa_k\ge e\left(J^{(k)}\right)$ since $\frac{n}{\Delta} \ge \log^{ga} n \ge C' \cdot \frac{2g}{\alpha} \cdot \log^{ka} n$ as $k\le g-1$ and $n$ and $a$ are large enough. Now we check that $(*_i)$ holds for $i\in [k]$ as follows. Namely for all $i\in [k]$, we have as $p=\frac{\log^a n}{\binom{n}{b}}$ that
$$\frac{\Delta_i\left(J^{(k)}\right)}{\kappa_k}\cdot \log^{4k+2}n_0\le \frac{2g}{\alpha}\cdot C'\cdot \frac{1}{\binom{n}{b}^i}\cdot \frac{\Delta}{n} \cdot \log^{ka+(8k+4)}n \le \frac{\log^{ia} n}{\binom{n}{b}^i} = p^i$$ since $\frac{n}{\Delta} \ge \log^{ga}{n} \ge \frac{2g}{\alpha}\cdot C'\cdot \log^{ka+8k+4} n$ as  $k\le g-1$ and $a$ is large enough (at least $8k+5$) and $n$ is large enough. 

Thus by Corollary~\ref{cor:KimVu}, we find with probability at most
$(n_0)^{-\log n_0}$ (which is at most $n^{-\log n}$ since $n\le n_0$), we have that
$$e\left(J^{(k)}_p\right) > 2\cdot p^k\cdot \kappa_k = \frac{\alpha \cdot D}{g}.$$ 
Since $n$ is large enough then, we have that
$$\Prob{e\left(J^{(k)}_p\right) > \frac{\alpha\cdot D}{g}} \le n^{-\log n} \le  \frac{1}{100\cdot g\cdot n^r}$$
as desired since $n$ is large enough.
\end{proofclaim}

\begin{claim}\label{cl:6Holds}
(2) holds with probability at least $0.99$.    
\end{claim}
\begin{proofclaim}
For each $v\in V(G_1)\cap V(G_2)$, let $A_{2,v}$ be the event that $d_{G_2}(v) < D^{1-2\alpha}$.
Let $\mathcal{A}_2 := \bigcup_{v\in V(G_1)\cap V(G_2)} A_{2,v}$. We desire to show that $\Prob{\mc{A}_2} \le 0.01$ and hence by the union bound it suffices to show that for each $v\in V(G_1)\cap V(G_2)$
$$\Prob{A_{2,v}} \le \frac{1}{100\cdot n^r} \le \frac{1}{100\cdot |V(G_1)\cap V(G_2)|},$$
since $v(G_1)\le n^r$.

To that end, fix $v\in V(G_1)\cap V(G_2)$. Since $(G,A,X,\mc{H})$ is $(\alpha,g)$-regular, we have that $d_{\mc{R}}(v) \ge D^{1-\alpha}$.
Thus the event $A_{2,v}$ is a subset of the event that $d_{\mc{R}}(v) - d_{G_2}(v) > d_{\mc{R}}(v)-D^{1-2\alpha}$; note the latter quantity is at least $\frac{1}{2}\cdot d_{\mc{R}}(v)$ since $n$ (and hence $D$) is large enough.

We define a hypergraph $J$ 
$$J:= \bigcup_{Q\in \mc{R}: v\in Q}~\bigcup_{k\in [g-2]} \left\{ Z\setminus \{Q\}: Z\in{\rm Extend}_g(\mc{B})^{(1,k)}(Q)\right\}.$$
Let $n_0:= n^{g(b+r+q)}$. Note that $v(J) \le e({\rm Extend}_g(\mc{B})) \le n^{g(b+r+q)}$ which follows since $n$ is large enough. Also $\log^2 n\ge \log n_0$ since $n$ is large enough.

Note that for $Q\in \mc{R}$, we have $Q\not\in G_1$ if and only if there exists $Z\in{\rm Extend}_g(\mc{B})^{(1,k)}(Q)$ for some $k\in [g-2]$ such that $Z\setminus \{Q\}\subseteq \mc{B}_p$. Thus $A_{2,v}$ is a subset of the event $e(J_p) > \frac{1}{2}\cdot d_{\mc{R}}(v)$.

Note that $e(J_p) = \sum_{k\in [g-2]} e\left(J^{(k)}_p\right)$. For $k\in [g-2]$, let $A_{2,v,k}$ be the event that $e\left(J^{(k)}_p\right) > \frac{d_{\mc{R}}(v)}{2g}$. Hence $A_{2,v}$ is a subset of $\bigcup_{k \in [g-2]} A_{2,v,k}$. Thus it suffices to show for each  $k\in [g-2]$ that
$$\Prob{A_{2,v,k}} \le \frac{1}{100\cdot g \cdot n^r}.$$

First we calculate using Lemma~\ref{lem:ExtendConfigBoosterDegrees}(2) that
\begin{align*}
e\left(J^{(k)}\right) &\le \sum_{Q\in \mc{R}:~v\in Q} \left|{\rm Extend}_g(\mc{B})^{(1,k)}(Q)\right|\le d_{\mc{R}}(v) \cdot \Delta_{(1,0)}\left({\rm Extend}_g(\mc{B})^{(1,k)}\right) \\
&\le d_{\mc{R}}(v)\cdot C'\cdot \binom{n}{b}^k\cdot \frac{\Delta}{n}.
\end{align*}
Now fix $i\in [k]$ and $U\subseteq V(J^{(k)})$ with $|U|=i$. Then if $i\in [k-1]$, we have
$$|J^{(k)}(U)| \le \sum_{Q\in \mc{R}: v\in Q} |{\rm Extend}_g(\mc{B})^{(1,k)}(\{Q\}\cup U)| \le d_{\mc{R}}(v)\cdot 
\Delta_{(1,i)}({\rm Extend}_g(\mc{B})^{(1,k)}).$$
Hence by Lemma~\ref{lem:ExtendConfigBoosterDegrees}(1) for all $i\in [k-1]$ we have
$$\Delta_i\left(J^{(k)}\right) \le d_{\mc{R}}(v)\cdot C'\cdot \binom{n}{b}^k\cdot \frac{1}{n}.$$
If $i=k$, we have
$$|J^{(k)}(U)| \le |{\rm Extend}_g(\mc{B})^{(1,k)}(U)| \le 
\Delta_{(0,k)}({\rm Extend}_g(\mc{B})^{(1,k)}).$$
Hence by Lemma~\ref{lem:ExtendConfigBoosterDegrees}(1) we have
$$\Delta_k\left(J^{(k)}\right) \le C'\cdot D\cdot \frac{1}{n}.$$
Since $\alpha < \frac{1}{2(q-r)}$ and $D=\binom{n}{q-r}$, we find that
$$\frac{D}{n} \le D^{1-\frac{1}{q-r}} \le D^{1-2\alpha} \le d_{\mc{R}}(v)\cdot \frac{1}{4g\cdot C'\cdot \log^{8k+4} n},$$
where the last inequality follows since $d_{\mc{R}}(v) \ge D^{1-\alpha}$ and $n$ (and hence $D$) is large enough.

Let $\kappa_k:= \frac{d_{\mc{R}}(v)}{4g} \cdot \frac{\binom{n}{b}^k}{\log^{ka} n}$. Note that $\kappa_k\ge e\left(J^{(k)}\right)$ since $\frac{n}{\Delta} \ge \log^{ga} n \ge 4g\cdot \log^{ka} n$ since $k\le g-1$ and $n$ and $a$ are large enough. Now we check that $(*_i)$ holds for $i\in [k]$ as follows. Namely for all $i\in [k-1]$, we have as $p=\frac{\log^a n}{\binom{n}{b}}$ that
$$\frac{\Delta_i\left(J^{(k)}\right)}{\kappa_k}\cdot \log^{4k+2}n_0\le 4g\cdot C'\cdot \frac{1}{\binom{n}{b}^i}\cdot \frac{\Delta}{n} \cdot \log^{ka+(8k+4)}n \le \frac{\log^{ia} n}{\binom{n}{b}^i} = p^i$$ since $\frac{n}{\Delta} \ge \log^{ga}{n} \ge 4g\cdot C'\cdot \log^{ka+8k+5} n$ as  $k\le g-1$ and $a$ is large enough (at least $8k+5$) and $n$ is large enough. 
For $i=k$, we have
$$\frac{\Delta_k\left(J^{(k)}\right)}{\kappa_k}\cdot \log^{4k+2}n_0\le 4g\cdot C'\cdot \frac{D/n}{d_{\mc{R}}(v)}\cdot \frac{1}{\binom{n}{b}^k} \cdot \log^{ka+(8k+4)}n \le \frac{\log^{ka} n}{\binom{n}{b}^k} = p^i$$ since $\frac{D}{n}\cdot 4g\cdot C'\cdot \log^{8k+4} n \le d_{\mc{R}}(v)$ as noted above. 

Thus by Corollary~\ref{cor:KimVu}, we find with probability at most
$(n_0)^{-\log n_0}$ (which is at most $n^{-\log n}$ since $n\le n_0$), we have that
$$e\left(J^{(k)}_p\right) > 2\cdot p^k\cdot \kappa_k = \frac{d_{\mc{R}}(v)}{2g}.$$ 
Since $n$ is large enough then, we have that
$$\Prob{e\left(J^{(k)}_p\right) > \frac{d_{\mc{R}}(v)}{2g}} \le n^{-\log n} \le  \frac{1}{100\cdot g\cdot n^r}$$
as desired since $n$ is large enough.
\end{proofclaim}

\begin{claim}\label{cl:7Holds}
(3) holds with probability at least $0.99$.    
\end{claim}
\begin{proofclaim}
For each $e\in V(\mc{D})\cup V(\mc{R})$ and $j\in \{2,\ldots, g-1\}$, let $A_{3,e,j}$ be the event that $d_{H^{(j)}}(e) > 2\alpha\cdot D^{j-1}\cdot \log D$.
Let $\mathcal{A}_3 := \bigcup_{e\in \mc{D}\cup \mc{R}} \bigcup_{j\in \{2,\ldots,g-1\}} A_{3,e,j}$. We desire to show that $\Prob{\mc{A}_3} \le 0.01$ and hence by the union bound it suffices to show that for each $e\in \mc{D}\cup \mc{R}$ and $j\in \{2,\ldots, g-1\}$ that
$$\Prob{A_{3,e,j}} \le \frac{1}{100\cdot g\cdot n^q} \le \frac{1}{100 \cdot g\cdot (v(\mc{D})+v(\mc{R}))},$$
since $v(\mc{D})+v(\mc{R})\le n^q$.

To that end, fix $e\in V(\mc{D})\cup V(\mc{R})$ and $j\in \{2,\ldots, g-1\}$. Since $(G,A,X,\mc{H})$ is $(\alpha,g)$-regular, we have that  $d_{\mc{G}^{(j)}}(e) \le \alpha\cdot D^{j-1}\cdot \log D$. Hence $A_{3,e,j}$ is a subset of the event that $d_{H^{(j)}}(v) - d_{\mc{G}^{(j)}}(v) > \alpha\cdot D^{j-1}\cdot \log D$. 

We define a multi-hypergraph $J$ 
$$J:= \bigcup_{k\in [g-2]} \left\{ Z\setminus V(\mc{G}): Z\in{\rm Extend}_g(\mc{B})^{(j,k)}(e)\right\}.$$
Let $n_0:= n^{g(b+r+q)}$. Note that $v(J) \le e({\rm Extend}_g(\mc{B})) \le n^{g(b+r+q)}$ which follows since $n$ is large enough. Also $\log^2 n\ge \log n_0$ since $n$ is large enough.

Note that for $F\subseteq V(H)$ with $|F|=j$ and $e\in F$, we have $F\in H$ if and only if there exists $Z\in{\rm Extend}_g(\mc{B})^{(j,k)}(F)$ for some $k\in [g-2]$ such that $Z\setminus F\subseteq \mc{B}_p$. Thus $A_{3,e,j}$ is a subset of the event $e(J_p) > \alpha \cdot D^{j-1}\cdot \log D$.
Note that $e(J_p) = \sum_{k\in [g-2]} e\left(J^{(k)}_p\right)$. For $k\in [g-2]$, let $A_{3,e,j,k}$ be the event that $e\left(J^{(k)}_p\right) > \frac{\alpha\cdot D^{i-1}\cdot \log D}{g}$. Hence $A_{3,e,j}$ is a subset of $\bigcup_{k \in [g-2]} A_{3,e,j,k}$. Thus it suffices to show for each  $k\in [g-2]$ that
$$\Prob{A_{3,e,j,k}} \le \frac{1}{100\cdot g^2 \cdot n^q}.$$

First we calculate using Lemma~\ref{lem:ExtendConfigBoosterDegrees}(2) that
\begin{align*}
e\left(J^{(k)}\right) &\le \left|{\rm Extend}_g(\mc{B})^{(j,k)}(e)\right|\le \Delta_{(1,0)}\left({\rm Extend}_g(\mc{B})^{(j,k)}\right) \le C'\cdot D^{j-1}\cdot \binom{n}{b}^{k}\cdot \frac{\Delta}{n}.
\end{align*}
Now fix $i\in [k]$ and $U\subseteq V(J^{(k)})$ with $|U|=i$. Then
$$|J^{(k)}(U)| \le |{\rm Extend}_g(\mc{B})^{(j,k)}(\{e\}\cup U)| \le \Delta_{(1,i)}\left({\rm Extend}_g(\mc{B})^{(j,k)}\right).$$
Hence by Lemma~\ref{lem:ExtendConfigBoosterDegrees}(1) since $j>1$, we find that for all $i\in [k]$
$$\Delta_i\left(J^{(k)}\right) \le C'\cdot D^{j-1}\cdot \binom{n}{b}^{k-i}\cdot \frac{1}{n}.$$

Let $\kappa_k:= \frac{\alpha\cdot D^{j-1}\cdot \log D}{2g} \cdot \frac{\binom{n}{b}^k}{\log^{ka} n}$. Note that $\kappa_k\ge e\left(J^{(k)}\right)$ since $\frac{n}{\Delta} \ge \log^{ga} n \ge \frac{2g}{\alpha}\cdot C'\cdot \log^{ka} n$ since $k\le g-1$ and $n$ and $a$ are large enough. Now we check that $(*_i)$ holds for $i\in [k]$ as follows. Namely for all $i\in [k]$, we have as $p=\frac{\log^a n}{\binom{n}{b}}$ that
$$\frac{\Delta_i\left(J^{(k)}\right)}{\kappa_k}\cdot \log^{4k+2}n_0\le \frac{2g}{\alpha}\cdot C'\cdot \frac{1}{\binom{n}{b}^i}\cdot \frac{1}{n} \cdot \log^{ka+(8k+4)}n \le \frac{\log^{ia} n}{\binom{n}{b}^i} = p^i$$ since $n \ge \frac{2g}{\alpha}\cdot C'\cdot \log^{ak+8k+4} n$ as $n$ is large enough. 

Thus by Corollary~\ref{cor:KimVu}, we find with probability at most
$(n_0)^{-\log n_0}$ (which is at most $n^{-\log n}$ since $n\le n_0$), we have that
$$e\left(J^{(k)}_p\right) > 2\cdot p^k\cdot \kappa_k = \frac{\alpha \cdot D^{j-1}\cdot \log D}{g}.$$ 
Since $n$ is large enough then, we have that
$$\Prob{e\left(J^{(k)}_p\right) > \frac{\alpha\cdot D^{j-1}\cdot \log D}{g}} \le n^{-\log n} \le  \frac{1}{100\cdot g^2\cdot n^q}$$
as desired since $n$ is large enough.
\end{proofclaim}

\begin{claim}\label{cl:8Holds}
(4) always holds.    
\end{claim}
\begin{proofclaim}
Since ${\rm Treasury}(K_n^r,X)$ is $\bigg(\binom{n}{q-r},~(1-\alpha)\cdot \binom{n}{q-r},~\frac{1}{2g(q-r)},~\alpha\bigg)$-regular by assumption, we have that ${\rm Design}_{K_q^r}(G\setminus X)\cup {\rm Reserve}_{K_q^r}(G, G\setminus X, X)$ has codegrees at most $\binom{n}{q-r}^{1- \frac{1}{2g(q-r)}} \le D^{1-\beta}$
since $\beta = \frac{1}{2g(q-r)}$ and $D = \binom{n}{q-r}$. Since $G_1 \subseteq {\rm Design}_{K_q^r}(G\setminus (B\cup A\cup X)) \subseteq {\rm Design}_{K_q^r}(G\setminus X)$ and $G_2\subseteq {\rm Reserve}_{K_q^r}(G,G\setminus (B\cup A\cup X),X) \subseteq {\rm Reserve}_{K_q^r}(G, G\setminus X, X)$, it follows that $G_1\cup G_2$ has codegrees at most $D^{1-\beta} \le D^{1-(\beta/2)}$ as desired.
\end{proofclaim}

\begin{claim}\label{cl:9Holds}
(5) holds with probability at least $0.99$.    
\end{claim}
\begin{proofclaim}
Let $\mc{Q}_t$ be the set of $Q\subseteq V(\mc{D})\cup V(\mc{R})$ with $|Q|=t$ such that $Q$ is a $\mc{G}$-avoiding matching of $\mc{D}\cup \mc{R}$.
For each $2\le t < g-1$, $Q\in \mc{Q}_t$, and $s\in \{|Q|+1,\ldots, g-1\}$, let $A_{5,Q,s}$ be the event that $|H^{(s)}(Q)| > D^{s-t-(\beta/2)}$.
Let $\mathcal{A}_5 := \bigcup_{2\le t< g-1} \bigcup_{s: t< s\le g-1} \bigcup_{Q\in \mc{Q}_t} A_{5,Q,s}$. We desire to show that $\Prob{\mc{A}_5} \le 0.01$ and hence by the union bound it suffices to show that for each $t$ with $2\le t < g-1$, $s$ with $t< s \le g-1$ and $Q\in \mc{Q}_t$ that
$$\Prob{A_{5,Q,s}} \le \frac{1}{100\cdot g^2\cdot n^{qg}} \le \frac{1}{100 \cdot g\cdot \sum_{t: 2\le t < g-1} |\mc{Q}_t|},$$
since $\sum_{t: 2\le t < g-1} |\mc{Q}_t|\le g\cdot n^{qg}$.

To that end, fix $t \in \{2,\ldots, g-2\}$, $s\in \{t+1,\ldots, g-1\}$ and $Q\in \mc{Q}_t$. Since $(G,A,X,\mc{H})$ is $(\alpha,g)$-regular, we have that  $|\mc{G}^{(s)}(Q)| \le D^{s-t-\beta}$. Hence $A_{5,Q,s}$ is a subset of the event that $|H^{s}(Q)| - |\mc{G}^{s}(Q)| > g\cdot D^{s-t-\beta}$ since $D^{s-t-(\beta/2)} > (g+1)\cdot D^{s-t-\beta}$ since $D$ is large enough. 

We define a multi-hypergraph $J$ 
$$J:= \bigcup_{k\in [g-2]} \left\{ Z\setminus V(\mc{G}): Z\in{\rm Extend}_g(\mc{B})^{(s,k)}(Q)\right\}.$$
Let $n_0:= n^{g(b+r+q)}$. Note that $v(J) \le e({\rm Extend}_g(\mc{B})) \le n^{g(b+r+q)}$ which follows since $n$ is large enough. Also $\log^2 n\ge \log n_0$ since $n$ is large enough.

Note that for $F\subseteq V(H)$ with $|F|=s$ and $Q\subseteq F$, we have $F\in H$ if and only if there exists $Z\in{\rm Extend}_g(\mc{B})^{(s,k)}(F)$ for some $k\in [g-2]$ such that $Z\setminus F\subseteq \mc{B}_p$. Thus $A_{5,Q,s}$ is a subset of the event $e(J_p) > g\cdot D^{s-t-\beta}$.
Note that $e(J_p) = \sum_{k\in [g-2]} e\left(J^{(k)}_p\right)$. For $k\in [g-2]$, let $A_{5,Q,s,k}$ be the event that $e\left(J^{(k)}_p\right) > D^{s-t-\beta}$. Hence $A_{5,Q,s}$ is a subset of $\bigcup_{k \in [g-2]} A_{5,Q,s,k}$. Thus it suffices to show for each  $k\in [g-2]$ that
$$\Prob{A_{5,Q,s,k}} \le \frac{1}{100\cdot g^3 \cdot n^{qg}}.$$

First we calculate using Lemma~\ref{lem:ExtendConfigBoosterDegrees}(1) that
\begin{align*}
e\left(J^{(k)}\right) &\le \left|{\rm Extend}_g(\mc{B})^{(s,k)}(Q)\right|\le \Delta_{(t,0)}\left({\rm Extend}_g(\mc{B})^{(s,k)}\right) \le C'\cdot D^{s-t}\cdot \binom{n}{b}^{k}\cdot \frac{1}{n}.
\end{align*}
Now fix $i\in [k]$ and $U\subseteq V(J^{(k)})$ with $|U|=i$. Then
$$|J^{(k)}(U)| \le |{\rm Extend}_g(\mc{B})^{(s,k)}(Q\cup U)| \le \Delta_{(t,i)}\left({\rm Extend}_g(\mc{B})^{(s,k)}\right).$$
Hence by Lemma~\ref{lem:ExtendConfigBoosterDegrees}(1) since $s > t\ge 2$, we find that for all $i\in [k]$
$$\Delta_i\left(J^{(k)}\right) \le C'\cdot D^{s-t}\cdot \binom{n}{b}^{k-i}\cdot \frac{1}{n}.$$

Let $\kappa_k:= \frac{D^{s-t-\beta}}{2} \cdot \frac{\binom{n}{b}^k}{\log^{ka} n}$. Note that $\kappa_k\ge e\left(J^{(k)}\right)$ since $n \ge D^{2\beta} \ge 2C'\cdot D^{\beta} \cdot \log^{ka} n$ since $\beta \le \frac{1}{2(q-r)}$ and $n$ is large enough. Now we check that $(*_i)$ holds for $i\in [k]$ as follows. Namely for all $i\in [k]$, we have as $p=\frac{\log^a n}{\binom{n}{b}}$ that
$$\frac{\Delta_i\left(J^{(k)}\right)}{\kappa_k}\cdot \log^{4k+2}n_0\le 2C'\cdot \frac{1}{\binom{n}{b}^i}\cdot \frac{D^{\beta}}{n} \cdot \log^{ka+(8k+4)}n \le \frac{\log^{ia} n}{\binom{n}{b}^i} = p^i$$ since $n \ge 2C'\cdot D^{\beta} \cdot \log^{ak+8k+4} n$ as $n\ge D^{2\beta}$ and $n$ is large enough. 

Thus by Corollary~\ref{cor:KimVu}, we find with probability at most
$(n_0)^{-\log n_0}$ (which is at most $n^{-\log n}$ since $n\le n_0$), we have that
$$e\left(J^{(k)}_p\right) > 2\cdot p^k\cdot \kappa_k = D^{s-t-\beta}.$$ 
Since $n$ is large enough then, we have that
$$\Prob{e\left(J^{(k)}_p\right) > D^{s-t-\beta} }\le n^{-\log n} \le  \frac{1}{100\cdot g^3\cdot n^{qg}}$$
as desired since $n$ is large enough.
\end{proofclaim}

\begin{claim}\label{cl:10Holds}
(6) holds with probability at least $0.99$.    
\end{claim}
\begin{proofclaim}
For each $v\in V(\mc{D})\cup V(\mc{R})$ and $e\in V(\mc{G})$ with $v\not\in e$, let $N_{H}(v,e):= \{ F\in E(H^{(2)}): F=\{e,f\}, v\in f\}$ and let $A_{6,v,e}$ be the event that $|N_H(v,e)| > D^{1-(\beta/2)}$.
Let $\mathcal{A}_6 := \bigcup_{v\in V(\mc{D})\cup V(\mc{R})} \bigcup_{e\in V(\mc{G}): v\not\in e} A_{6,v,e}$. We desire to show that $\Prob{\mc{A}_6} \le 0.01$ and hence by the union bound it suffices to show that for each $v\in V(\mc{D})\cup V(\mc{R})$ and $e\in V(\mc{G})$ with $v\not\in e$  that
$$\Prob{A_{6,v,e}} \le \frac{1}{100\cdot n^{q+r}} \le \frac{1}{100 \cdot (v(\mc{D})+v(\mc{R}))\cdot v(\mc{G})},$$
since $v(\mc{D})+v(\mc{R})\le n^r$ and $v(\mc{G})\le n^q$.

To that end, fix $v\in V(\mc{D})\cup V(\mc{R})$ and $e\in V(\mc{G})$ with $v\not\in e$. Since $(G,A,X,\mc{H})$ is $(\alpha,g)$-regular, we have that  $|N_{\mc{G}}(v,e)| \le D^{1-\beta}$. Hence $A_{5,Q,s}$ is a subset of the event that $|N_{H}(v,e)| - |N_{\mc{G}}(v,e)| > g\cdot D^{1-\beta}$ since $D^{1-(\beta/2)} > (g+1)\cdot D^{1-\beta}$ since $D$ is large enough. 

We define a multi-hypergraph $J$ 
$$J:= \bigcup_{Q\in V(\mc{G}): v\in Q}~\bigcup_{k\in [g-2]} \left\{ Z\setminus V(\mc{G}): Z\in{\rm Extend}_g(\mc{B})^{(2,k)}(Q\cup \{e\})\right\}.$$
Let $n_0:= n^{g(b+r+q)}$. Note that $v(J) \le e({\rm Extend}_g(\mc{B})) \le n^{g(b+r+q)}$ which follows since $n$ is large enough. Also $\log^2 n\ge \log n_0$ since $n$ is large enough.

Note that for $F=\{e,f\}$ where $f\in V(\mc{G})$ with $v\in f$, we have $F\in H$ if and only if there exists $Z\in{\rm Extend}_g(\mc{B})^{(2,k)}(F)$ for some $k\in [g-2]$ such that $Z\setminus F\subseteq \mc{B}_p$. Thus $A_{6,v,e}$ is a subset of the event $e(J_p) > g\cdot D^{1-\beta}$.
Note that $e(J_p) = \sum_{k\in [g-2]} e\left(J^{(k)}_p\right)$. For $k\in [g-2]$, let $A_{6,v,e,k}$ be the event that $e\left(J^{(k)}_p\right) > D^{1-\beta}$. Hence $A_{6,v,e}$ is a subset of $\bigcup_{k \in [g-2]} A_{6,v,e,k}$. Thus it suffices to show for each  $k\in [g-2]$ that
$$\Prob{A_{6,v,e,k}} \le \frac{1}{100\cdot g \cdot n^{q+r}}.$$

First we calculate using Lemma~\ref{lem:ExtendConfigBoosterDegrees}(1) since $k\ge 1$ that
\begin{align*}
e\left(J^{(k)}\right) &\le \sum_{Q\in V(\mc{G}):v\in Q} \left|{\rm Extend}_g(\mc{B})^{(2,k)}(Q\cup \{e\})\right|\le \sum_{Q\in V(\mc{G}):v\in Q} \Delta_{(2,0)}\left({\rm Extend}_g(\mc{B})^{(2,k)}\right) \\
&\le \sum_{Q\in V(\mc{G}):v\in Q} C'\cdot \binom{n}{b}^{k}\cdot \frac{1}{n} \le C'\cdot D\cdot \binom{n}{b}^{k}\cdot \frac{1}{n}.
\end{align*}
Now fix $i\in [k]$ and $U\subseteq V(J^{(k)})$ with $|U|=i$. Then
$$|J^{(k)}(U)| \le |{\rm Extend}_g(\mc{B})^{(2,k)}(\{e\}\cup U)| \le \Delta_{(1,i)}\left({\rm Extend}_g(\mc{B})^{(2,k)}\right).$$
Hence by Lemma~\ref{lem:ExtendConfigBoosterDegrees}(1), we find for all $i\in [k]$ since $i\ge 1$ that
$$\Delta_i\left(J^{(k)}\right) \le C'\cdot D\cdot \binom{n}{b}^{k-i}\cdot \frac{1}{n}.$$

Let $\kappa_k:= \frac{D^{1-\beta}}{2} \cdot \frac{\binom{n}{b}^k}{\log^{ka} n}$. Note that $\kappa_k\ge e\left(J^{(k)}\right)$ since $n \ge D^{2\beta} \ge 2C'\cdot D^{\beta} \cdot \log^{ka} n$ since $\beta \le \frac{1}{2(q-r)}$ and $n$ is large enough. Now we check that $(*_i)$ holds for $i\in [k]$ as follows. Namely for all $i\in [k]$, we have as $p=\frac{\log^a n}{\binom{n}{b}}$ that
$$\frac{\Delta_i\left(J^{(k)}\right)}{\kappa_k}\cdot \log^{4k+2}n_0\le 2C'\cdot \frac{1}{\binom{n}{b}^i}\cdot \frac{D^{\beta}}{n} \cdot \log^{ka+(8k+4)}n \le \frac{\log^{ia} n}{\binom{n}{b}^i} = p^i$$ since $n \ge 2C'\cdot D^{\beta} \cdot \log^{ak+8k+4} n$ as $n\ge D^{2\beta}$ and $n$ is large enough. 

Thus by Corollary~\ref{cor:KimVu}, we find with probability at most
$(n_0)^{-\log n_0}$ (which is at most $n^{-\log n}$ since $n\le n_0$), we have that
$$e\left(J^{(k)}_p\right) > 2\cdot p^k\cdot \kappa_k = D^{1-\beta}.$$ 
Since $n$ is large enough then, we have that
$$\Prob{e\left(J^{(k)}_p\right) > D^{1-\beta} }\le n^{-\log n} \le  \frac{1}{100\cdot g\cdot n^{q+r}}$$
as desired since $n$ is large enough.
\end{proofclaim}

\begin{claim}\label{cl:11Holds}
(7) holds with probability at least $0.99$.    
\end{claim}
\begin{proofclaim}
For $e\in E(\mc{D})=V(\mc{G})$, let $\mc{N}(e):= \{e'\in V(\mc{G}):~|V_{K_n^r}(e)\cap V_{K_n^r}(e')| \ge r\}$ where to clarify, $V_{K_n^r}(e), V_{K_n^r}(e')$ denote the set of vertices in $K_n^r$ in the $q$-cliques $e$ and $e'$ respectively. Note that if $e'\in \mc{N}(e)$, then $e$ and $e'$ do not form a matching of $\mc{D}$, and hence there does not exist $S\in E(\mc{G})$ such that $e,e'\in S$.

For each $e_1\ne e_2\in V(\mc{G})$ with $e_2\not\in \mc{N}(e_1)$, let $N_{H}(e_1,e_2):= \{ f\in V(H): e_1f, e_2f\in E(H)\}$ and let $A_{7,e_1,e_2}$ be the event that $|N_H(e_1,e_2)| > D^{1-(\beta/2)}$.
Let $\mathcal{A}_7 := \bigcup_{e_1\ne e_2\in V(\mc{G}):~e_2\not\in \mc{N}(e_1)} A_{7,e_1,e_2}$. Note that if $\mc{A}_7$ does not hold, then $H$ is $D^{1-(\beta/2)}$-uncommon, that is (7) holds. Thus we desire to show that $\Prob{\mc{A}_7} \le 0.01$ and hence by the union bound it suffices to show that for all $e_1\ne e_2\in V(\mc{G})$ with $e_2\not\in \mc{N}(e_1)$ that
$$\Prob{A_{7,e_1,e_2}} \le \frac{1}{100\cdot n^{2q}} \le \frac{1}{100 \cdot v(\mc{G})^2},$$
since $v(\mc{G})\le n^q$.

To that end, fix $e_1\ne e_2\in V(\mc{G})$ with $e_2\not\in \mc{N}(e_1)$. Since $(G,A,X,\mc{H})$ is $(\alpha,g)$-regular, we have that  $|N_{\mc{G}}(e_1,e_2)| \le D^{1-\beta}$. Hence $A_{7,e_1,e_2}$ is a subset of the event that $|N_{H}(e_1,e_2)| - |N_{\mc{G}}(e_1,e_2)| > g^3\cdot D^{1-\beta}$ since $D^{1-(\beta/2)} > (g^3+1)\cdot D^{1-\beta}$ since $D$ is large enough.

For $f\in \mc{G}$ and $i\in \{1,2\}$, let 
$$\mc{N}_{j}(e_i,f) := \left\{Z: Z\cup\{e_i,f\}\in{\rm Extend}_g(\mc{B})^{(2,j)}(\{e_i,f\})\right\}.$$
We define a multi-hypergraph $J$ 
$$J:= \bigcup_{f\in V(\mc{G})}~\bigcup_{k_1, k_2\in [g-2]} \bigg\{ Z_1\cup Z_2: Z_i\in \mc{N}_{k_i}(e_i,f)~\forall i\in\{1,2\} \bigg\}.$$
Let $n_0:= n^{2g(b+r+q)}$. Note that $v(J) \le e({\rm Extend}_g(\mc{B}))^2 \le n^{2g(b+r+q)}$ which follows since $n$ is large enough. Also $\log^2 n\ge \log n_0$ since $n$ is large enough. Note that for $i\in \{1,2\}$ and $F_i=\{e_i,f\}$ where $f\in V(\mc{G})$, we have $F_i\in H$ if and only if there exists $Z_i\in{\rm Extend}_g(\mc{B})^{(2,k_i)}(F_i)$ for some $k_i\in [g-2]$ such that $Z_i\setminus F_i\subseteq \mc{B}_p$. Thus $A_{7,e_1,e_2}$ is a subset of the event $e(J_p) > g^3\cdot D^{1-\beta}$.

Note that $e(J_p) = \sum_{k\in [2(g-2)]} e\left(J^{(k)}_p\right)$. For $k\in [g-2]$, let $A_{7,e_1,e_2,k}$ be the event that $e\left(J^{(k)}_p\right) > D^{1-\beta}$. Hence $A_{7,e_1,e_2}$ is a subset of $\bigcup_{k \in [g-2]} A_{7,e_1,e_2,k}$. Thus it suffices to show for each  $k\in [g-2]$ that
$$\Prob{A_{7,e_1,e_2,k}} \le \frac{1}{100\cdot g \cdot n^{2q}}.$$

Now we seek to upper bound $e(J^{(k)})$. To that end, for $f\in \mc{G}$ and $k_1,k_2\in [g-2]$ and $k_3\in [\min\{k_1,k_2\}]$, let
$$\mc{N}_{k_1,k_2,k_3}(e_1,e_2,f):= \{(Z_1,Z_2): Z_i\in \mc{N}_{k_i}(e_i,f)~\forall i\in\{1,2\},~|Z_1\cap Z_2|=k_3\}.$$
Let 
$$\mc{K}_k := \{(k_1,k_2,k_3): k_1,k_2\in [g-2], k_3\in \{0,\ldots, \min\{k_1,k_2\} \}, k=k_1+k_2-k_3\}.$$ 
First we calculate that
$$e\left(J^{(k)}\right) \le \sum_{f\in V(\mc{G})}~\sum_{(k_1,k_2,k_3)\in \mc{K}_k} |\mc{N}_{k_1,k_2,k_3}(e_1,e_2,f)|.$$
Now we calculate that
\begin{align*}
&\sum_{f\in V(\mc{G})}~\sum_{(k_1,k_2,k_3)\in \mc{K}(k): k_1 < k} |\mc{N}_{k_1,k_2,k_3}(e_1,e_2,f)| \\
&\le \sum_{k_1: 1\le k_1 < k}~~\sum_{(f,Z_1)\in {\rm Extend}_g(\mc{B})^{(2,k_1)}(e_1)} ~~\sum_{Z_3\subseteq Z_1} |{\rm Extend}_g(\mc{B})^{(2,~k-k_1+|Z_3|)}(\{e_2,f\} \cup Z_3)|\\
&\le \sum_{k_1: 1\le k_1 < k}~~\sum_{(f,Z_1)\in {\rm Extend}_g(\mc{B})^{(2,k_1)}(e_1)} ~\sum_{k_3\in [k_1]}~\sum_{Z_3\subseteq Z_1: |Z_3|=k_3} |{\rm Extend}_g(\mc{B})^{(2,~k-k_1+k_3)}(\{e_2,f\} \cup Z_3)| \\
&\le \sum_{k_1: 1\le k_1 < k}~~\sum_{(f,Z_1)\in {\rm Extend}_g(\mc{B})^{(2,k_1)}(e_1)} ~\sum_{k_3\in [k_1]}~\sum_{Z_3\subseteq Z_1: |Z_3|=k_3} \Delta_{(2,k_3)}\left({\rm Extend}_g(\mc{B})^{(2,k-k_1+k_3)}\right).
\end{align*}
By Lemma~\ref{lem:ExtendConfigBoosterDegrees}(1), we find that 
$$\Delta_{(2,k_3)}\left({\rm Extend}_g(\mc{B})^{(2,k-k_1+k_3)}\right) \le C'\cdot \binom{n}{b}^{k-k_1}\cdot \frac{1}{n} $$
Hence
\begin{align*}
\sum_{f\in V(\mc{G})} \sum_{(k_1,k_2,k_3)\in \mc{K}(k): k_1 < k} |\mc{N}_{k_1,k_2,k_3}(e_1,e_2,f)| \le \sum_{k_1: 1\le k_1 < k}~|{\rm Extend}_g(\mc{B})^{(2,k_1)}(e_1)|\cdot 2^{k_1} \cdot  C'\cdot \binom{n}{b}^{k-k_1}\cdot \frac{1}{n}.
\end{align*}
By Lemma~\ref{lem:ExtendConfigBoosterDegrees}(2), we find that 
$$\Delta_{(1,0)}\left({\rm Extend}_g(\mc{B})^{(2,k_1)}\right) \le C'\cdot \binom{n}{q-r}\cdot \binom{n}{b}^{k_1}\cdot \frac{1}{n} $$
Combining we find that
\begin{align*}
\sum_{f\in V(\mc{G})}~\sum_{(k_1,k_2,k_3)\in \mc{K}(k): k_1 < k} |\mc{N}_{k_1,k_2,k_3}(e_1,e_2,f)| &\le \sum_{k_1: 1\le k_1 < k} C'\cdot \binom{n}{q-r}\cdot \binom{n}{b}^{k_1}\cdot \frac{1}{n}\cdot 2^{k_1} \cdot  C'\cdot \binom{n}{b}^{k-k_1}\cdot \frac{1}{n}\\
&\le k\cdot 2^k \cdot (C')^2 \cdot \binom{n}{q-r}\cdot \binom{n}{b}^k \cdot \frac{1}{n}.
\end{align*}
By symmetry, we find that
$$\sum_{f\in V(\mc{G})}~\sum_{(k_1,k_2,k_3)\in \mc{K}(k): k_2 < k} |\mc{N}_{k_1,k_2,k_3}(e_1,e_2,f)| \le k\cdot 2^k \cdot (C')^2 \cdot \binom{n}{q-r}\cdot \binom{n}{b}^k \cdot \frac{1}{n}.$$

So it remains to calculate $\sum_{f\in V(\mc{G})} |\mc{N}_{k,k,k}(e_1,e_2,f)|$. The key fact is that since $|e_1 \cap e_2|\le r$ as $e_2\not\in \mc{N}(e_1)$, we find that $(Z_1,Z_2)\in \bigcup_{f\in V(\mc{G})} \mc{N}_{k,k,k}(e_1,e_2,f)$ only if $Z_1=Z_2$ and $Z_1\cup \{f\}\in {\rm Extend}_q(\mc{B})^{(2,k)}(e_1)$ with $v\in V(Z_1\cup \{f\})$ for some $v\in e_2\setminus e_1$ (as otherwise $|(Z_2\cup \{f\})\cap e_2|\le r$ which is impossible if $Z_2\cup \{e_2,f\} \in E(\mc{G})$ since then $v(Z_2\cup \{f\}) \le (q-r)(|Z_2\cup \{f\}|)+r$ contradicting that $Z_2\cup \{e_2,f\}$ is an Erd\H{o}s configuration).

Let $C''$ be as in Lemma~\ref{lem:UncommonDegree}. Hence by Lemma~\ref{lem:UncommonDegree}, we find that
$$\sum_{f\in V(\mc{G})} |\mc{N}_{k,k,k}(e_1,e_2,f)| \le \sum_{v\in e_2\setminus e_1} \left|\left\{\mc{Z}\in {\rm Extend}_q(\mc{B})^{(2,k)}(e_1): v\in V(\mc{Z})\right\}\right| \le q\cdot C'\cdot \binom{n}{q-r}\cdot \binom{n}{b}^k \cdot \frac{1}{n}.$$

Combining, we find that
\begin{align*}
e\left(J^{(k)}\right) &\le \sum_{f\in V(\mc{G})}~\sum_{(k_1,k_2,k_3)\in \mc{K}_k} |\mc{N}_{k_1,k_2,k_3}(e_1,e_2,f)| \\
&\le 2\cdot k\cdot 2^k \cdot (C')^2 \cdot \binom{n}{q-r}\cdot \binom{n}{b}^k \cdot \frac{1}{n}+ q\cdot C'\cdot \binom{n}{q-r}\cdot \binom{n}{b}^k \cdot \frac{1}{n} \\
&\le 3qk\cdot 2^k\cdot (C')^2 \cdot \binom{n}{q-r}\cdot \binom{n}{b}^k \cdot \frac{1}{n}.
\end{align*}

Now fix $i\in [k]$ and $U\subseteq V(J^{(k)})$ with $|U|=i$. Fix a partition $(U_1,U_2,U_3)$ of $U$ into 3 (possibly empty) sets and let $i_j:= |U_j|$ for $i\in [3]$. Then we define for $(k_1,k_2,k_3)\in \mc{K}_k$ where $k_j \ge i_j$ for all $j\in [3]$, the following
$$\mc{N}_{k_1,k_2,k_3}(e_1,e_2,f)[U_1,U_2,U_3] := \{(Z_1,Z_2)\in \mc{N}_{k_1,k_2,k_3}(e_1,e_2,f): U_1\subseteq Z_1\setminus Z_2,~U_2\subseteq Z_2\setminus Z_1,~U_3\subseteq Z_1\cap Z_2\}.$$
If $i_1+i_3\ge 1$, then
\begin{align*}
&|\mc{N}_{k_1,k_2,k_3}(e_1,e_2,f)[U_1,U_2,U_3]| \\
&\le \sum_{(f,Z_1)\in {\rm Extend}_g(\mc{B})^{(2,k_1)}(\{e_1\}\cup U_1\cup U_3)}~\sum_{Z_3\subseteq Z_1: U_3\subseteq Z_3, |Z_3|=k_3} |{\rm Extend}_g(\mc{B})^{(2,k_2)}(\{e_2,f\}\cup Z_3\cup U_2)| \\
&\le \sum_{(f,Z_1)\in {\rm Extend}_g(\mc{B})^{(2,k_1)}(\{e_1\}\cup U_1\cup U_3)} 2^{k_1} \cdot \Delta_{(2,i_2+k_3)}\left({\rm Extend}_g(\mc{B})^{(2,k_2)}\right)\\
&\le 2^{k_1}\cdot \Delta_{(1,i_1+i_3)}\left({\rm Extend}_g(\mc{B})^{(2,k_1)}\right)\cdot \Delta_{(2,i_2+i_3)}\left({\rm Extend}_g(\mc{B})^{(2,k_2)}\right)\\
&\le 2^{g}\cdot C'\cdot \binom{n}{q-r}\cdot \binom{n}{b}^{k_1-(i_1+i_3)}\cdot \frac{1}{n} \cdot C'\cdot \binom{n}{b}^{k_2-(i_2+k_3)}\\
&= 2^{g}\cdot (C')^2 \cdot \binom{n}{q-r}\cdot \binom{n}{b}^{k-i}\cdot \frac{1}{n},
\end{align*}
where we used Lemma~\ref{lem:ExtendConfigBoosterDegrees}(1) and (2) each once. Similarly the same holds if $i_2+i_3\ge 1$ by symmetry. Since $i_1+i_2+i_3=i\ge 1$, at least one of these two outcomes hold. Hence in all cases
$$|\mc{N}_{k_1,k_2,k_3}(e_1,e_2,f)[U_1,U_2,U_3]|\le 2^{g}\cdot (C')^2 \cdot \binom{n}{q-r}\cdot \binom{n}{b}^{k-i}\cdot \frac{1}{n}$$
Thus by considering all $(k_1,k_2,k_3)\in \mc{K}_k$ and partitions $(U_1,U_2,U_3)$ of $U$ into $3$ possibly non-empty sets, we find that
\begin{align*}
|J^{(k)}(U)| &\le g^4\cdot 2^{g}\cdot (C')^2 \cdot \binom{n}{q-r}\cdot \binom{n}{b}^{k-i}\cdot \frac{1}{n}.
\end{align*}.
Hence, we find for all $i\in [k]$ that
$$\Delta_i\left(J^{(k)}\right) \le g^4\cdot 2^{g}\cdot (C')^2 \cdot \binom{n}{q-r}\cdot \binom{n}{b}^{k-i}\cdot \frac{1}{n}.$$

Let $\kappa_k:= \frac{D^{1-\beta}}{2} \cdot \frac{\binom{n}{b}^k}{\log^{ka} n}$. Note that $\kappa_k\ge e\left(J^{(k)}\right)$ since $n \ge D^{2\beta} \ge 2\cdot 3qk\cdot 2^k\cdot (C')^2\cdot D^{\beta} \cdot \log^{ka} n$ since $\beta \le \frac{1}{2(q-r)}$ and $n$ is large enough. Now we check that $(*_i)$ holds for $i\in [k]$ as follows. Namely for all $i\in [k]$, we have as $p=\frac{\log^a n}{\binom{n}{b}}$ that
$$\frac{\Delta_i\left(J^{(k)}\right)}{\kappa_k}\cdot \log^{4k+2}n_0\le 2g^4\cdot 2^g\cdot (C')^2\cdot \frac{1}{\binom{n}{b}^i}\cdot \frac{D^{\beta}}{n} \cdot \log^{ka+(8k+4)}n \le \frac{\log^{ia} n}{\binom{n}{b}^i} = p^i$$ since $n \ge 2g^4\cdot 2^g\cdot (C')^2\cdot D^{\beta} \cdot \log^{ak+8k+4} n$ as $n\ge D^{2\beta}$ and $n$ is large enough. 

Thus by Corollary~\ref{cor:KimVu}, we find with probability at most
$(n_0)^{-\log n_0}$ (which is at most $n^{-\log n}$ since $n\le n_0$), we have that
$$e\left(J^{(k)}_p\right) > 2\cdot p^k\cdot \kappa_k = D^{1-\beta}.$$ 
Since $n$ is large enough then, we have that
$$\Prob{e\left(J^{(k)}_p\right) > D^{1-\beta} }\le n^{-\log n} \le  \frac{1}{100\cdot g\cdot n^{2q}}$$
as desired since $n$ is large enough.\end{proofclaim}

By Claims~\ref{cl:5Holds}-~\ref{cl:11Holds}, we have by the union bound that with probability at least $0.9$ all of (1)-(7) hold. That is, with probability at least $0.9$, we have that ${\rm Proj}_g(B,A,G,X)$ is $\bigg(\binom{n}{q-r},~(1-2\alpha) \cdot \binom{n}{q-r},~\frac{1}{4g(q-r)},~2\alpha\bigg)$-regular as desired.
\end{lateproof}

\section{High Cogirth Pairs of High Girth Steiner Systems}\label{s:CogirthPair}

In this section, we discuss the modifications necessary to our proof of Theorem~\ref{thm:HighGirthSteiner} to yield a proof of Theorem~\ref{thm:HighCogirthSteiner}. We note that despite much thought, it did not seem possible to prove the following stronger version of Theorem~\ref{thm:HighCogirthSteiner}: that for every $(n,q,r)$-Steiner system $S_1$ of girth at least $g$ there exists an $(n,q,r)$-Steiner system $S_2$ of girth at least $g$ whose cogirth with $S_1$ is at least $g$. The issue is that the extendability of high girth designs is highly non-trivial due to configurations between what is already given and what needs to be produced. Indeed, it is plausible to us that the stronger version mentioned above may not be true. 

Similarly one of the original motivations for finding $K_q^r$ decomposition of graphs with large minimum degree is that they imply that $K_q^r$ packings of small maximum degree can be extended to a full $K_q^r$ decomposition. This implication is no longer true for high girth designs due to the same configuration issue. While some minimum degree version of Theorem~\ref{thm:HighCogirthSteiner} is true, it is conceivable to us that the extension version may not be true. That said, the extension versions mentioned above may be true if the given object is sufficiently quasi-random. However, proving that our proof of Theorem~\ref{thm:HighGirthSteiner} yields a sufficiently quasi-random high girth Steiner system seemed hard; so we adopt a different approach to prove Theorem~\ref{thm:HighCogirthSteiner} as follows.

To prove Theorem~\ref{thm:HighCogirthSteiner}, we will build \emph{in parallel} two $(n,q,r)$-Steiner systems $S_1$ and $S_2$ of girth at least $g$ and cogirth at least $g$. To that end, we let $G_1$ and $G_2$ be two disjoint copies of $K_n^r$. Here is a proof overview then:

\begin{enumerate}
    \item[(1)] `Reserve' a random subset $X$ of $E(K_n^r)$ independently with some small probability $p$ via Lemma~\ref{lem:RandomXTreasury}. For $i\in \{1,2\}$, let $X_i$ be the subset of $G_i$ corresponding to $X$. 
    \item[(2)] Construct a $K_q^r$ omni-absorber $A$ of $X$ as in Theorem~\ref{thm:Omni}. For $i\in \{1,2\}$, let $A_i$ be the subset of $G_i$ corresponding to $A$.
    \item[(2')] For each $i\in\{1,2\}$ construct a $K_q^r$-omni-booster $B_i$ of $A_i$ that has girth at least $g$, that together have cogirth at least $g$, and do shrink ``too many" configurations (which now includes `cogirth' configurations). Let $A_i' := A_i\cup B_i$.
    \item[(3)] For each $i\in \{1,2\}$, ``regularity boost'' the cliques of $G_i\setminus (A_i'\cup X_i)$ using Lemma~\ref{lem:RegBoost}.
    \item[(4)] Apply Theorem~\ref{thm:ForbiddenSubmatchingReserves} to the union of the resulting treasuries to find two $K_q^r$ packings, one $Q_1$ of $(G_1\setminus A_1')$ and the other $Q_2$ of $(G_2\setminus A_2')$ where $Q_1$ covers $(G_1\setminus (A_1'\cup X_1')$ and $Q_2$ covers $(G_2\setminus (A_2'\cup X_2'))$ and then for each $i\in\{1,2\}$ extend $Q_i$ to a $K_q^r$ decomposition $S_i$ of $G_i$ by definition of $K_q^r$ omni-absorber.  
\end{enumerate}

We omit the actual proof of Theorem~\ref{thm:HighCogirthSteiner} since it tracks closely to the proof of Theorem~\ref{thm:HighGirthSteiner} with the modifications mentioned above. Instead, we devote the rest of this section to the necessary modifications to prove a cogirth analog of Theorem~\ref{thm:HighGirthAbsorber}.

\subsection{Cogirth Configuration Hypergraphs}

Thus the main additional difficulty in the plan above lies in Step (2'). The first key is that we need to introduce an additional configuration hypergraph to handle cogirth as follows.

\begin{definition}[Erd\H{o}s cogirth configuration]
Let $q>r\ge 1$, $g\ge 3$. Let $G$ be an $r$-graph. Let $w_1,w_2\ge 1$ be integers. A \emph{$(w_1,w_2)$-cogirth-$g$-configuration} of a $G$ is $(\mc{B}_1, \mc{B}_2)$ where for $i\in \{1,2\}$, $\mc{B}_i$ is a $K_q^r$-packing of $G$ of girth at least $g$ such that $\mc{B}_1\cup \mc{B}_2$ contains a $(i(q-r)+r-1,i)$-configuration for some $2\le i \le g-1$.

A \emph{$(w_1,w_2)$-cogirth-$g$-Erd\H{o}s-configuration} of a $G$ is a $(w_1,w_2)$-cogirth-$g$-configuration of $G$ that does not contain inclusion-wise a $(w_1',w_2')$-cogirth-$g$-configuration of $G$ for some $w_1'\in [1,w_1]$, $w_2'\in [1,w_2]$ such that $w_1'+w_2' < w_1+w_2$.
\end{definition}

\begin{definition}[Cogirth Configuration Hypergraph]
Let $q> r\ge 1$ be integers and let $g\ge 3$ be an integer. Let $G$ be an $r$-uniform hypergraph and for $i\in \{1,2\}$, let $\mathcal{D}_i$ be a copy of ${\rm Design}_{K_q^r}(G)$. 
Then the \emph{cogirth-$g$ $K_q^r$-configuration hypergraph} of $G$, denoted ${\rm Cogirth}_{K_q^r}^g(G)$ is the configuration hypergraph $H$ of $\mathcal{D}_1\cup \mc{D}_2$ with $V(H):=E(\mc{D}_1)\cup E(\mc{D}_2)$ and \begin{align*}
E(H) := \{S_1\cup S_2:~&S_i\subseteq \mc{D}_i~\forall i\in\{1,2\}, \\
&(S_1, S_2) \text{ is a $(w_1,w_2)$-cogirth-g-Erd\H{o}s-configuration of } G \text{ for some $w_1,w_2\ge 1$}\}.    
\end{align*}
\end{definition}

We note the following upper bound on the degrees and codegrees of ${\rm Cogirth}_{K_q^r}^g(G)$. 

\begin{proposition}\label{prop:CogirthDegrees}
For all integers $q> r\ge 1$ and $g\ge 3$, there exists an integer $C\ge 1$ such that for all integers $n$ and $1\le t < s \le g$, we have that 
$$\Delta_{t}\bigg({\rm Cogirth}_{K_q^r}^g(K_n^r)^{(s)}\bigg)\le C\cdot \binom{n}{q-r}^{s-t} \cdot \frac{1}{n}.$$
\end{proposition}
\begin{proof}
For brevity, let $\mc{C}:= {\rm Cogirth}_{K_q^r}^g(K_n^r)$. Fix $U\subseteq V(\mc{C})$ with $|U|=t$. Note $|\mc{C}^{(s)}(U)|=0$ unless $U=U_1\cup U_2$ where for all $i\in \{1,2\}$, $U_i$ is a (possibly empty) $K_q^r$-packing of $G$ of girth at least $g$, and $|V(U_1)\cup V(U_2)| ]\ge t(q-r)+r$. But then
\begin{align*}
|\mc{C}^{(s)}(U)| &\le \binom{qs}{q}^s \cdot n^{s(q-r)+r-1 - (t(q-r)+r)} \le  \binom{qs}{q}^s \cdot n^{(s-t)(q-r)-1}\\
&\le C\cdot \binom{n}{q-r}^{s-t} \cdot \frac{1}{n},
\end{align*}
where for the last inequality we used that $C$ is large enough.
\end{proof}

\subsection{High Cogirth Omni-Absorbers and Cogirth Treasuries}

Analogous to Definition~\ref{def:GirthTreasury}, we define a new cogirth treasury as follows.

\begin{definition}
Let $q > r \ge 1$ be integers and let $g\ge 3$ be an integer. Let $G$ be an $r$-uniform hypergraph and let $X.G'\subseteq G$. The \emph{cogirth-$g$ $K_q^r$-design treasury} of $G$ with reserve $X$, denoted ${\rm CoTreasury}_q^g(G,G_1',G_2',X)$, is the treasury $(\mc{D}_1\cup \mc{D}_2,~ \mc{R}_1\cup \mc{R}_2,~\mc{G}_1\cup \mc{G}_2\cup \mc{C})$.
where for $i\in\{1,2\}$, $\mc{D}_i$ is a copy of ${\rm Design}_{K_q^r}(G_i'\setminus X)$, $\mc{R}_i$ is a copy of ${\rm Reserve}_{K_q^r}(G, G_i'\setminus X, X)$, and $\mc{G}_i$ is a copy of ${\rm Girth}_{K_q^r}^g(G)$; and $\mc{C}$ is ${\rm Cogirth}_{K_q^r}^g(K_n^r)$.
\end{definition}

Analogous to Definition~\ref{def:OmniAbsorberCollectiveGirth}, we define a notion of collective cogirth for two omni-absorber as follows.

\begin{definition}\label{def:OmniAbsorberCollectiveCogirth}
Let $X$ be an $r$-uniform hypergraph. We say two $K_q^r$-omni-absorbers $A_1,A_2$ for $X$ with decomposition functions $\mathcal{Q}_{A_1}, \mc{Q}_{A_2}$ have \emph{collective cogirth at least $g$} if $A_1$ and $A_2$ each have collective girth at least $g$ and for every pair of $K_q^r$-divisible subgraphs $L_1,L_2$ of $X$, $\mathcal{Q}_{A_1}(L_1)$ and $\mathcal{Q}_{A_2}(L_2)$ have cogirth at least $g$. 
\end{definition}

Analogous to Definition~\ref{def:OmniAbsorberProj}, we define a common projection of two omni-absorbers as follows.

\begin{definition}[Cogirth $g$ Projection of Two Omni-Absorbers]\label{def:OmniAbsorberCoProj}
Let $G$ be an $r$-uniform hypergraph, let $X$ be a subgraph of $G$, and let $A_1,A_2\subseteq G$ be two $K_q^r$-omni-absorbers of $X$ with decomposition function $\mathcal{Q}_{A_1}$ and $\mc{Q}_{A_2}$ respectively. We define the \emph{cogirth-$g$ projection treasury} of $A_1$ and $A_2$ on to $G$ and $X$ as follows: we let
$${\rm CoProj}_g(A_1,A_2,G,X):= {\rm CoTreasury}_q^g(G,~G\setminus A_1,~G\setminus A_2,~X)  \perp \mathcal{M}(A_1,A_2)$$
where
$$\mathcal{M}(A_1,A_2) := \{ M_1\cup M_2: M_1\in \mc{M}(A_1),~M_2\in \mc{M}(A_2)\}.$$
\end{definition}
Note that in the definition above, $\mc{M}(A_1)$ is a set of matchings in the first copy of ${\rm Design}_{K_q^r}(G)$ while $\mc{M}(A_2)$ is a set of matchings in the second copy of ${\rm Design}_{K_q^r}(G)$.

Analogous to Theorem~\ref{thm:HighGirthAbsorber}, here then is our high cogirth omni-absorber theorem.

\begin{thm}[High Cogirth Omni-Absorber Theorem]\label{thm:HighCogirthAbsorber}
For all integers $q>r\ge 1$ and $g\ge 3$ and real $\alpha \in \left(0,~\frac{1}{2(q-r)}\right)$, there exist integers $a,n_0\ge 1$ such that the following holds for all $n\ge n_0$: 

Suppose that Theorem~\ref{thm:HighCogirthBooster} holds for $r':=r-1$. If $X$ is a spanning subgraph of $K_n^r$ with $\Delta(X) \le \frac{n}{\log^{ga} n}$ such that ${\rm CoTreasury}^g_q(K_n^r,K_n^r,X)$ is $\bigg(\binom{n}{q-r},~\alpha\cdot \binom{n}{q-r},~\frac{1}{2g(q-r)},~\alpha\bigg)$-regular and we let $\Delta:= \max\left\{ \Delta(X),~n^{1-\frac{1}{r}}\cdot \log n\right\}$, then there exists two $K_q^r$-omni-absorbers $A_1,A_2$ for $X$ with $\Delta(A_i)\le a\Delta\cdot \log^{a} \Delta$ for $i\in \{1,2\}$ and collective cogirth at least $g$, together with a subtreasury $T$ of ${\rm CoProj}_g(A_1,A_2 K_n^r, X)$ that is 
$$\Bigg(\binom{n}{q-r},~4\alpha \cdot \binom{n}{q-r},~\frac{1}{8g(q-r)},~3\alpha\Bigg){\rm - regular}.$$
\end{thm}

We note the parameters differ slightly from those of Theorem~\ref{thm:HighGirthAbsorber}, namely they allow slightly more irregularity, slightly smaller reserve degrees and slightly larger  codegrees. This is because we could reuse our previous work (Lemmas~\ref{lem:RandomQuantumIntrinsic} and~\ref{lem:RandomQuantumExtrinsic}) but also incorporate the cogirth configuration hypergraph (for which we need to sligtly relax the parameters).

\subsection{Cogirth Omni-Boosters}

Analogous to Definition~\ref{def:OmniBoosterProj}, we next define a projection treasury for pair of omni-boosters as follows.

\begin{definition}[Cogirth $g$ Projection of Two Omni-Boosters]\label{def:OmniBoosterCoproj}
Let $(G,A,X,\mc{H})$ be a $K_q^r$-sponge. Let $B_1,B_2$ be two $K_q^r$-omni-boosters for $A$ with booster families $\mathcal{B}_i=(B_{i,H}:H\in\mathcal{H})$ for $i\in \{1,2\}$. 

We define the \emph{cogirth-$g$ projection treasury} of $B_1$ and $B_2$ on to $G$ and $X$ as:
$${\rm Proj}_g(B_1,B_2,A,G,X) := {\rm CoTreasury}_q^g(G,~G\setminus (A\cup B_1),~G\setminus (A\cup B_2),~X)\perp \mathcal{M}(\mc{B}_1,\mc{B}_2)$$
where
$$\mathcal{M}(\mc{B}_1,\mc{B}_2) := \{ M_1\cup M_2: M_1\in \mc{M}(\mc{B}_1),~ M_2\in \mc{M}(\mc{B}_2)\}$$
\end{definition}
Again we note that in the definition above, $\mc{M}(\mc{B}_1)$ is a set of matchings in the first copy of ${\rm Design}_{K_q^r}(G)$ while $\mc{M}(\mc{B}_2)$ is a set of matchings in the second copy of ${\rm Design}_{K_q^r}(G)$.

Thus Theorem~\ref{thm:HighCogirthAbsorber} is proved via the following pair of omni-boosters theorem, analogous to Theorem~\ref{thm:HighGirthOmniBooster}. 

\begin{thm}[High Cogirth Omni-Booster Theorem]\label{thm:HighCogirthOmniBooster}
For all integers $q>r\ge 1$, $C\ge 1$ and $g\ge 3$ and real $\alpha \in \left(0,~\frac{1}{2(q-r)}\right)$, there exist integers $a,n_0\ge 1$ such that the following holds for all $n\ge n_0$: 

Suppose that Theorem~\ref{thm:HighCogirthBooster} holds for $r':=r-1$. If $(K_n^r, A, X)$ is an $(\alpha,g)$-regular $ga$-bounded $C$-refined $K_q^r$-sponge, then there exist two $K_q^r$-omni-boosters $B_1,B_2$ for $A$ and $X$ with $\Delta(B_i)\le a\Delta\cdot \log^{a} \Delta$ for $i\in \{1,2\}$ and collective cogirth at least $g$, and a subtreasury $T$ of ${\rm CoProj}_g(B_1,B_2,A,G,X)$ such that $T$ is 
$$\Bigg(\binom{n}{q-r},~4\alpha \cdot \binom{n}{q-r},~\frac{1}{8g(q-r)},~3\alpha\Bigg)-{\rm regular}.$$
\end{thm}

Theorem~\ref{thm:HighCogirthOmniBooster} is proved meanwhile via pairs of quantum omni-boosters. 

\subsection{Cogirth Quantum Omni-Boosters}

Thus analgous to Definition~\ref{def:QuantumOmniBoosterProj}, we next define a cogirth projection for a pair of quantum omni-boosters as follows. 

\begin{definition}[Cogirth $g$ Projection of Quantum Omni-Booster]
Let $(G,A,X,\mc{H})$ be a $K_q^r$-sponge. Let $B_1,B_2$ be two quantum $K_q$-omni-boosters for $A$ with quantum booster collection $\mathcal{B}_i = (\mathcal{B}_{i,H}: H\in\mathcal{H})$ for $i\in \{1,2\}$.

We define the \emph{cogirth-$g$ projection treasury} of $B_1$ and $B_2$ on to $G$ and $X$ as: 
$${\rm Proj}_g(B_1,B_2,A,G,X):= {\rm CoTreasury}_q^g(G,~G\setminus A,~G\setminus A,~X) \perp \mathcal{M}(\mc{B}_1,\mc{B}_2)$$
where
$$\mathcal{M}(\mc{B}_1,\mc{B}_2) := \{ M_1\cup M_2: M_1\in \mc{M}(\mc{B}_1),~ M_2\in \mc{M}(\mc{B}_2)\}.$$
\end{definition}

Analogous to a related part of Definition~\ref{def:QuantumOmniBooster}, we next define a high cogirth hypergraph for a pair of quantum boosters as follows.

\begin{definition}[High Cogirth Hypergraph for Two Quantum Omni-Boosters]
Let $q > r\ge 1$ be integers. Let $(G,A,X,\mc{H})$ be a $K_q^r$-sponge. Let $\mc{B} \in {\rm Full}_{B_0}(G,A,X)$ for some rooted $K_q^r$-booster $B_0$.  Let $\mc{B}_1, \mc{B}_2$ be two copies of $\mc{B}$ on disjoint copies of $G$.

For an integer $g\ge 2$ and $i\in \{1,2\}$, we define
\begin{align*}{\rm HighCogirth}_g(\mathcal{B}_{i,H}) := \bigg\{B_{i,H} \in \mathcal{B}_{i,H}: \nexists &R\in {\rm Cogirth}_{K_q^r}^g(G) \text{ where } R\subseteq M \text{ for some } M \in \mc{M}(\mc{B}_1,\mc{B}_2) \\
&\text{ and } R\cap \big( (B_{i,H})_{\rm on}\cup (B_{i,H})_{\rm off}\big)\ne \emptyset. \bigg\}.
\end{align*}
\end{definition}

Thus to prove Theorem~\ref{thm:HighCogirthOmniBooster}, we will prove the following theorem about a pair of quantum omni-boosters, analogous to Theorem~\ref{thm:HighGirthQuantumOmniBooster}.

\begin{thm}[High Cogirth Quantum Omni-Booster Theorem]\label{thm:HighCogirthQuantumOmniBooster}
For all integers $q>r\ge 1$, $C\ge 1$ and $g\ge 3$ and real $\alpha \in \left(0,~\frac{1}{2(q-r)}\right)$, there exist integers $a, n_0\ge 1$ such that the following holds for all $n\ge n_0$: 

Suppose that Theorem~\ref{thm:HighCogirthBooster} holds for $r':=r-1$. If $(K_n^r, A, X,\mc{H})$ is an $(\alpha,g)$-regular $ga$-bounded $C$-refined $K_q^r$-sponge, then there exists two quantum $K_q^r$-omni-boosters $B_1,B_2$ for $A$ and $X$ with quantum booster collections $\mc{B}_i=(\mathcal{B}_{i,H}: H\in \mathcal{H})$ for $i\in \{1,2\}$ such that both of the following hold:
\begin{enumerate}
    \item[(1)] $\Delta(B_i)\le \Delta\cdot \log^{a} \Delta$ for all $i\in\{1,2\}$ and for each $H\in \mathcal{H}$, 
    $${\rm Disjoint}(\mc{B}_{i,H}) \cap {\rm HighGirth}_g(\mc{B}_{i,H}) \cap {\rm HighCogirth}_g(\mc{B}_{i,H}) \ne \emptyset,$$
    \item[(2)] ${\rm Proj}_g(B_1,B_2,A,K_n^r,X)$ is $\Bigg(\binom{n}{q-r},~3\alpha \cdot \binom{n}{q-r},~\frac{1}{4g(q-r)},~3\alpha\Bigg)-{\rm regular}$.
\end{enumerate} 
\end{thm}

The proof of Theorem~\ref{thm:HighCogirthQuantumOmniBooster} proceeds similarly to that of Theorem~\ref{thm:HighGirthQuantumOmniBooster}; namely we choose a subset $\mc{B}_{i,p}$ of $\mc{B}_i$ by choosing each element independently with probability $p := \frac{\log^a n}{\binom{n}{v(B_0)-q}}$. Then in fact, most of the desired properties directly from Lemmas~\ref{lem:RandomQuantumIntrinsic} and~\ref{lem:RandomQuantumExtrinsic}. The only additional properties needed to prove Theorem~\ref{thm:HighCogirthQuantumOmniBooster} are to ensure that $|(\mc{B}_{i,H})_p \setminus {\rm HighCogirth}_g((\mc{B}_{i,H})_p)| < \frac{1}{12}\cdot \log^a n$ for all $H\in \mc{H}$ and $i\in \{1,2\}$ and to prove that projecting out by $\mc{M}(\mc{B}_1,\mc{B}_2)$ in ${\rm Cogirth}_{K_q^r}^g(G)$ does not change the regularity too much. This part proceeds via proofs similar to those given in Lemmas~\ref{lem:RandomQuantumIntrinsic} and~\ref{lem:RandomQuantumExtrinsic} but with different auxiliary hypergraphs which we define in the next subsections. That said, we omit the proofs of these desired properties as they follow via similar applications of Kim-Vu (in fact, they are easier since we have the better degree and codegrees bounds from Proposition~\ref{prop:CogirthDegrees}).

\subsection{Cogirth Booster Configuration Hypergraph}

Analogous to Definition~\ref{def:GirthBoosterConfig}, we define the following. 

\begin{definition}
Let $q>r\ge 2$ and $g\ge 3$ be integers. Let $(G,A,X,\mc{H})$ be a $K_q^r$-sponge. Let $\mc{B} \in {\rm Full}_{B_0}(G,A,X)$ for some rooted $K_q^r$-booster $B_0$. Let $\mc{B}_1, \mc{B}_2$ be two copies of $\mc{B}$ on disjoint copies of $G$.

An \emph{$g$-cogirth-booster-configuration} of $\mc{B}_1\cup \mc{B}_2$ is a subset $U =U_1\cup U_2$ where $\{U_{i,1},\ldots,U_{i,m_i}\}\subseteq \mc{B}_i$ for all $i\in \{1,2\}$ with $m_1,m_2\ge 1$ and such that $|U\cap \mc{B}_{i,H}|\le 1$ for all $H\in \mc{H}$ and there exists $F_1\cup F_2\in {\rm Cogirth}_{K_q^r}^g(G)$ and for $i\in \{1,2\}$, a partition of $F_i$ into $m_i$ nonempty sets, $F_{i,1},\ldots, F_{i,m_i}$ such that for each $j\in [m_i]$, we have $F_{i,j}\subseteq (U_{i,j})_{\rm on}$ or $F_{i,j}\subseteq (U_{i,j})_{\rm off}$.
\end{definition}

Analogous to Definition~\ref{def:GirthConfig}, we define the following. 

\begin{definition}
Let $q>r\ge 2$ and $g\ge 3$ be integers. Let $(G,A,X,\mc{H})$ be a $K_q^r$-sponge. Let $\mc{B} \in {\rm Full}_{B_0}(G,A,X)$ for some rooted $K_q^r$-booster $B_0$. Let $\mc{B}_1,\mc{B}_2$ be two copies of $\mc{B}$ on disjoint copies of $G$.

Define the \emph{cogirth-$g$ configuration hypergraph} of $\mc{B}_1\cup  \mc{B}_2$, denoted ${\rm Cogirth}_g(\mc{B}_1,\mc{B}_2)$, as
$$V({\rm Cogirth}_g(\mc{B}_1,\mc{B}_2)) := \bigcup_{H\in \mc{H}} \mc{B}_{i,H},$$
$$E({\rm Cogirth}_g(\mc{B}_1,\mc{B}_2)):= \{ \mc{Z}\subseteq V({\rm Cogirth}_g(\mc{B}_1,\mc{B}_2)): \mc{Z} \text{ is a $g$-cogirth-booster-configuration of } \mc{B}_1\cup \mc{B}_2\}.$$
\end{definition}

Analogous to Lemma~\ref{lem:GirthConfigBoosterDegrees}, we have the following lemma, except here there is no special case with a factor of $\frac{\Delta}{n}$ given the better values in Proposition~\ref{prop:CogirthDegrees}.

\begin{lem}\label{lem:CogirthConfigBoosterDegrees}
For all integers $q>r\ge 1$, $C\ge 1$ and $g\ge 3$, and $K_q^r$ booster $B_0$ of rooted girth at least $g$, there exists an integer $C'\ge 1$ such that the following holds: Let $(G,A,X,\mc{H})$ be a $C$-refined $K_q^r$-sponge and let $\Delta := \max\left\{ \Delta(X),~v(G)^{1-\frac{1}{r}}\cdot \log v(G)\right\}$ and let $b:=v(B_0)-q$. If $\mc{B} \in {\rm Full}_{B_0}(G,A,X)$ and $\mc{B}_1,\mc{B}_2$ are two copies of $\mc{B}$ on disjoint copies of $G$, then for all $2g\ge s > t \ge 1$, we have
$$\Delta_{t}\left({\rm Cogirth}_g(\mc{B})^{(s)}\right)\le C'\cdot \binom{n}{b}^{s-t}\cdot \frac{1}{n}.$$  
\end{lem}

\subsection{Cogirth Extended Configuration Hypergraph}

Analogous to Definition~\ref{def:CliqueBoosterConfig}, we define the following. 

\begin{definition}
Let $q>r\ge 2$ and $g\ge 3$ be integers. Let $(G,A,X,\mc{H})$ be a $K_q^r$-sponge. Let $\mc{B} \in {\rm Full}_{B_0}(G,A,X)$ for some rooted $K_q^r$-booster $B_0$.  Let $\mc{B}_1, \mc{B}_2$ be two copies of $\mc{B}$ on disjoint copies of $G$.

An \emph{$g$-clique-booster-cogirth-configuration} of $\mc{B}$ is a subset $U\subseteq V({\rm Cogirth}^g_{K_q^r}(G))\cup \mc{B}$ such that $U'= U\cap V({\rm Cogirth}^g_{K_q^r}(G))\ne \emptyset$, $U\cap \mc{B} =\{U_1,\ldots,U_m\}\subseteq \mc{B}$ with $m\ge 1$, $|U\cap \mc{B}_{i,H}|\le 1$ for all $H\in \mc{H}$, and there exists $F\in {\rm Cogirth}_{K_q^r}^g(G)$ with $U'\subseteq F$ and a partition of $F\setminus U'$ into $m$ nonempty sets, $F_1,\ldots, F_m$, such that for each $i\in [m]$, we have $F_i\subseteq (U_i)_{\rm on}$ or $F_i\subseteq (U_i)_{\rm off}$.
\end{definition}

Analogous to Definition~\ref{def:ExtendConfig}, we define the following. 

\begin{definition}
Let $q>r\ge 2$ and $g\ge 3$ be integers. Let $(G,A,X,\mc{H})$ be a $K_q^r$-sponge. Let $\mc{B} \in {\rm Full}_{B_0}(G,A,X)$ for some rooted $K_q^r$-booster $B_0$. Let $\mc{B}_1,\mc{B}_2$ be two copies of $\mc{B}$ on two disjoint copies of $G$.

Define the \emph{girth-$g$ extended cogirth configuration hypergraph}, ${\rm CoExtend}_g(\mc{B})$, as
$$V\bigg({\rm CoExtend}_g(\mc{B}_1,\mc{B}_2)\bigg) := V\left({\rm Cogirth}^g_{K_q^r}(G)\right)\cup V\bigg({\rm Cogirth}_g(\mc{B}_1,\mc{B}_2)\bigg),$$
\begin{align*}
E({\rm CoExtend}_g(\mc{B}_1,\mc{B}_2)):= &\{ \mc{Z} \subseteq V({\rm CoExtend}^g(\mc{B}_1,\mc{B}_2)): \mc{Z} \text{ is a $g$-clique-booster-cogirth-configuration }\}
\end{align*}
For integers $s_1, s_2\ge 1$, we define 
\begin{align*}
{\rm CoExtend}_g(\mc{B}_1,\mc{B}_2)^{(s_1,s_2)}:= \{Z_1\cup Z_2\in {\rm CoExtend}_g(\mc{B}_1,\mc{B}_2):~&Z_1\subseteq V({\rm Cogirth}^g_{K_q^r}(G)), Z_2\subseteq V({\rm Cogirth}_g(\mc{B}_1,\mc{B}_2)), \\
&|Z_1|=s_1,~|Z_2|=s_2\},    
\end{align*}
and for integers $t_1,t_2$ with $s_1\ge t_1\ge 0$ and $s_2\ge t_2\ge 0$, we define
\begin{align*}
&\Delta_{(t_1,t_2)}\left({\rm CoExtend}_g(\mc{B}_1,\mc{B}_2)^{(s_1,s_2)}\right) \\
&:= \max_{U_1\subseteq V({\rm Cogirth}^g_{K_q^r}(G)), |U_1|=t_1}~~\max_{U_2\subseteq V({\rm Cogirth}_g(\mc{B}_1,\mc{B}_2)), |U_2|=t_2}~~|{\rm CoExtend}_g(\mc{B}_1,\mc{B}_2)^{(s_1,s_2)}(U_1\cup U_2)|.    
\end{align*}
\end{definition}

Once again, there is a bi-regularity apparent in the upper bounds for the degrees and codegrees of ${\rm CoExtend}_g(\mc{B}_1,\mc{B}_2)$. Hence the definitions above. Analogous to Lemma~\ref{lem:ExtendConfigBoosterDegrees}, we have the following lemma, except here there is no special case with a factor of $\frac{\Delta}{n}$ given the better values in Proposition~\ref{prop:CogirthDegrees}.

\begin{lem}\label{lem:CoExtendConfigBoosterDegrees}
For all integers $q>r\ge 1$, $C\ge 1$ and $g\ge 3$, and $K_q^r$ booster $B_0$ of rooted girth at least $g$, there exists an integer $C'\ge 1$ such that the following holds: Let $(G,A,X,\mc{H})$ be a $C$-refined $K_q^r$-sponge and let $\Delta := \max\left\{ \Delta(X),~v(G)^{1-\frac{1}{r}}\cdot \log v(G)\right\}$ and let $b:= v(B_0)-q$. If $\mc{B} \in {\rm Full}_{B_0}(G,A,X)$ and $\mc{B}_1,\mc{B}_2$ are two copies of $\mc{B}$ on disjoint copies of $G$, then for all $2g\ge s > t \ge 1$, then for integers $s_1 \ge t_1 \ge 0$, $s_2\ge t_2\ge 0$ with $s_1+s_2 > t_1+t_2 \ge 1$, we have
$$\Delta_{(t_1,t_2)}\left({\rm CoExtend}_g(\mc{B})^{(s_1,s_2)}\right)\le C'\cdot \binom{n}{q-r}^{s_1-t_1}\cdot \binom{n}{b}^{s_2-t_2}\cdot \frac{1}{n}.$$
\end{lem}

Given the better values above, we do not require the analog of Lemma~\ref{lem:UncommonDegree} for cogirth. Now the proof of Theorem~\ref{thm:HighCogirthQuantumOmniBooster} follows similarly to that Theorem~\ref{thm:HighGirthQuantumOmniBooster} except that in the analog of Lemma~\ref{lem:RandomQuantumExtrinsic} we note that we do not require the analogues of Claims~\ref{cl:10Holds} and~\ref{cl:11Holds} since a better version of Claim~\ref{cl:7Holds} holds given the better values above. 

\section{Concluding Remarks and Further Directions}\label{s:Conclusion}

We note that our proof of Theorem~\ref{thm:HighGirthSteiner} straightforwardly generalizes to prove a minimum degree generalization of the High Girth Existence Conjecture as follows.

\begin{thm}[Existence of High Girth  Designs - Minimum Degree Version]\label{thm:HighGirthMinDegree}
For all integers $q > r \geq 2$ and every integer $g\ge 2$, there exist $n_0 \ge 1$ and $\varepsilon > 0$ such that for all $K_q^r$-divisible $r$-uniform hypergraphs $G$ with $v(G)\ge n_0$ and $\delta(G) \ge (1-\varepsilon)\cdot v(G)$, there exists a $K_q^r$-decomposition of $G$ with girth at least $g$.
\end{thm}

We omit the proof details. That said, it would be interesting what the best possible minimum degree for Theorem~\ref{thm:HighGirthMinDegree}; in particular when $r=2$, could it match the best known bounds for $K_q$-decompositions of graphs of high minimum degree?


As to other research directions, it would be interesting to investigate whether high girth designs exist in random settings. In particular, would results with Tom Kelly about thresholds for the existence of Steiner systems in $G(n,p)$~\cite{DKPIII} and in the $q$-uniform random binomial graph~\cite{DKPIV} that we also obtained via refined absorption hold for high girth Steiner systems? Given the intricacies involved in all these results such a task might seem daunting. That said, it is not inconceivable that one theorem that generalizes all these results and even incorporates nearly optimal bounds on high minimum degree might be true.


\bibliographystyle{plain}
\bibliography{mdelcourt, ref}

\end{document}